\documentclass[11pt]{amsart}
\usepackage{verbatim,amssymb,amsmath,graphicx,hyperref,color}
\usepackage{float}
\usepackage[normalem]{ulem}

  \scrollmode

\newcommand{\R}{{\mathbb R}}
\newcommand{\N}{{\mathbb N}}

\newcommand{\EE}{{\mathbb E}}
\newcommand{\PP}{{\mathbb P}}

\newcommand{\eul}{{\widehat X}}
\newcommand{\euly}{{\widehat Y}}

\newcommand{\ind}{1}
\newcommand{\usn}{\underline {s}_n}
\newcommand{\utn}{\underline {t}_n}

\newcommand{\sgn}{\operatorname{sgn}}
\newcommand{\eps}{\varepsilon}

\newcommand{\reach}{\text{reach}}
\newcommand{\normal}{\mathfrak{n}}
\newcommand{\pr}{\text{pr}}
\newcommand{\unp}{\text{unp}}
\newcommand{\tr}{\text{tr}}
\newcommand{\cl}{\text{cl}}
\newcommand{\inter}{\text{int}}
\newcommand{\supp}{\text{supp}}

\theoremstyle{plain}
\newtheorem{Thm}{Theorem}
\newtheorem{Prop}{Proposition}
\newtheorem{Lem}{Lemma}

\theoremstyle{definition}
\newtheorem{Rem}{Remark}
\newtheorem{exs}{Example}

\newcommand{\nor}{\mathfrak n}
\usepackage{mathtools}
\DeclarePairedDelimiter{\norm}{\|}{\|}
\DeclarePairedDelimiter{\abs}{\lvert}{\rvert}

\usepackage{dsfont}
\newcommand{\one}{\mathds{1}}

\begin{document}

\title[
On the Euler-Maruyama scheme for  SDEs with discontinuous drift coefficient]{
	On the performance of the Euler-Maruyama scheme for multidimensional SDEs with discontinuous drift coefficient}

\author[M\"uller-Gronbach]
{Thomas M\"uller-Gronbach}
\address{
Faculty of Computer Science and Mathematics\\
University of Passau\\
Innstrasse 33 \\
94032 Passau\\
Germany} \email{thomas.mueller-gronbach@uni-passau.de}

\author[Rauh\"ogger]
{Christopher Rauh\"ogger}
\address{
	Faculty of Computer Science and Mathematics\\
	University of Passau\\
	Innstrasse 33 \\
	94032 Passau\\
	Germany} \email{christopher.rauhoegger@uni-passau.de}

\author[Yaroslavtseva]
{Larisa Yaroslavtseva}
\address{
	Institute of Mathematics and Scientific Computing\\
	University of Graz\\
	Heinrichstra{\ss}e 36 \\
	8010 Graz\\
	Austria} \email{larisa.yaroslavtseva@uni-graz.at}
	
\begin{abstract} 
	We study strong approximation of $d$-dimensional stochastic differential equations (SDEs) with a discontinuous drift coefficient.
		More precisely, we 
		essentially 				assume that the drift coefficient is piecewise Lipschitz continuous with an exceptional set $\Theta\subset \R^d$
			that
			 is an orientable $C^4$-hypersurface of positive reach, the diffusion coefficient is  assumed to be 
 Lipschitz continuous and,    in a neighborhood of $\Theta$, both coefficients are bounded and the diffusion coefficient 
 has a non-degenerate portion  orthogonal to $\Theta$.
 
	In recent years, a number of results have been proven in the literature for
		strong approximation of such SDEs
		and,
				in particular, the performance of the Euler-Maruyama scheme was studied.
				For $d=1$ 
and finite $\Theta$
	it was shown in \cite{MGY20} that  the Euler-Maruyama scheme achieves an $L_p$-error rate of at least $1/2$ 
for all $p\geq 1$
as in the classical case of 
Lipschitz continuous coefficients. For $d>1$, it was only known so far, that the Euler-Maruyama scheme achieves 
		an $L_2$-error rate of at least $1/4-$  if, additionally, the coefficients $\mu$ and $\sigma$ are  globally bounded, see \cite{LS18}.

		In this article, we prove that
in the above setting
 the Euler-Maruyama scheme in fact achieves   an $L_{p}$-error rate of at least $1/2-$ for all $d\in \N$ and all $p \geq 1$.
The proof of this result is based on the well-known approach of transforming such an SDE into an SDE with globally Lipschitz continuous coefficients,  a new It\^{o} formula for a class of functions 
which are
not globally $C^2$  and  a detailed analysis of the expected total time that the actual position of the time-continuous Euler-Maruyama scheme and its position at the preceding time point on the underlying grid are on 'different sides' of  the hypersurface $\Theta$.
		\end{abstract}

\maketitle	
	\section{Introduction}

	Let 
	$ ( \Omega, \mathcal{F}, \PP ) $ 
	be a complete probability space with a 
	filtration $ ( \mathcal{F}_t )_{ t \in [0,1] } $
	that satisfies the usual conditions, let $d\in\N$ and 
	consider a $d$-dimensional autonomous stochastic differential equation (SDE)
	\begin{equation} \label{SDE}
		\begin{aligned}
			dX_{t} &= \mu(X_{t})dt + \sigma(X_{t})dW_{t}, \quad t \in [0,1], \\
			X_{0} &= x_{0},
		\end{aligned}
	\end{equation}
	where $x_0\in\R^d$,   $\mu\colon\R^d\to\R^d$ and   $\sigma\colon \R^d\to\R^{d \times d}$ are measurable functions and $
	W \colon [0,1] \times \Omega \to \R^d
	$
	is a $d$-dimensional
	$ ( \mathcal{F}_t )_{ t \in [0,1] } $-Brownian motion
	on $ ( \Omega, \mathcal{F}, \PP ) $.
	
	It is well-known that if the coefficients $\mu$ and $\sigma$ are Lipschitz continuous then the SDE \eqref{SDE} has a unique strong solution $X$. 
	Moreover, the Euler-Maruyama scheme
	with $n$ equidistant steps,
	 given 
	 by 
	$\eul_{n,0}=x_0$ and
	\[
	\eul_{n,(i+1)/n}=\eul_{n,i/n}+\mu(\eul_{n,i/n})\cdot 1/n+\sigma(\eul_{n,i/n})\cdot (W_{(i+1)/n}-
	W_{i/n})
	\]
	for  $i\in\{0,\ldots,n-1\}$, achieves at the final time $1$ an $L_p$-error rate of at least $1/2$ for all $p\geq 1$ in terms of the number $n$ of evaluations of $W$, i.e., 	for all $p\geq 1$  there exists $c>0$ such that for all $n\in\N$, 
	\begin{equation}\label{Leul}
	\bigl(\EE\bigl[\|X_1-\eul_{n,1}\|^p\bigr]\bigr)^{1/p}\leq \frac{c}{n^{1/2}},
	\end{equation}
	where $\|x\|$ denotes  the Euclidean norm of  $x\in\R^d$.
	
	In this article we study the performance of the Euler-Maruyama scheme $\eul_{n,1}$ in the case when the drift coefficient $\mu$ is discontinuous. Such SDEs  arise e.g.  in insurance, mathematical finance  and stochastic control problems, see e.g.~\cite{B80,Ka11,SS16} for examples.

We 
essentially make the following assumptions:
	the drift coefficient  $\mu$ is piecewise Lipschitz continuous with an exceptional set $\Theta$ that is an orientable $C^4$-hypersurface of positive reach, the diffusion coefficient $\sigma$ is 
	Lipschitz continuous and, in a  neighborhood of $\Theta$, the coefficients $\mu$ and $\sigma$ are bounded and $\sigma$
	has a non-degenerate portion  orthogonal to $\Theta$.
	See conditions (A) and (B) in Section \ref{ErrEs} for the precise assumptions on $\mu$ and $\sigma$.
	
	 For such {systems of} SDEs, existence and uniqueness of a strong solution 
	 is essentially known. See
	 Theorem 3.21 in {the pioneering paper} \cite{LS17} and 
	 Theorem 6
	 in the associated correction note~\cite{LScorr1}. See, however, 
Remark~\ref{wrongito}
in Section~\ref{ErrEs} 
for a discussion of some gaps
 in the proofs of the latter two theorems and 
	 Theorem~\ref{exist} in Section~\ref{ErrEs} 
	 for a complete proof of existence and uniqueness.

	Moreover, in \cite{LS17, LS18, NSS19}  $L_2$-appoximation of $X_1$ was studied {for the first time for multidimensional equations of this type}. More precisely, in \cite{LS17} 
	an  $L_2$-error rate of at least $1/2$ was shown for a transformation-based Euler-Maruyama scheme.	This scheme is obtained by first applying a suitable transformation to the SDE \eqref{SDE} to obtain an SDE with Lipschitz continuous coefficients, then using the Euler-Maruyama scheme to approximate the solution of the transformed SDE and finally  applying the inverse of the above transformation to the 
	Euler-Maruyama scheme 
	for the transformed SDE 
	to obtain an approximation to $X_1$.
	In {the pioneering paper} \cite{ NSS19}, an adaptive Euler-Maruyama scheme was constructed that
	 adapts its step size to the actual distance of the scheme to the exceptional set $\Theta$ of $\mu$ -- it uses smaller  time steps the smaller the distance to $\Theta$ is. {This elegant and immediately intuitive approach} was shown to achieve an $L_2$-error rate of at least $1/2-$ (i.e, $1/2-\delta$ for every $\delta>0$) in terms of the average number of evaluations of $W$. See, however, Remark \ref{remoccup}
	 for a  gap in the proof of the latter result.

In contrast to the classical Euler-Maruyama scheme, 
the two schemes from  \cite{LS17} and \cite{ NSS19} are not easy to implement
in general.
 In both cases, the exceptional set $\Theta$ must be known and projections to $\Theta$ of the actual position of the scheme or its distance to $\Theta$ have to be computed. Moreover, the transformation-based Euler-Maruyama scheme from \cite{LS17} also requires  evaluation of the inverse of the 
transformation at each step of the Euler-Maruyama scheme for the transformed SDE. This inverse  is, however, not known explicitly in general.

	In  \cite{LS18}, the performance of the 
classical
 Euler-Maruyama scheme $\eul_{n,1}$ for such SDEs was studied  and 
	an $L_2$-error rate of at least $1/4-$ was proven if the coefficients $\mu$ and $\sigma$ are additionally bounded.
See, however, Remark \ref{remoccup}
     for a  gap in the proof of the latter result.
	 Note that the $L_2$-error rate of at least $1/4-$ is significantly smaller than the $L_2$-error rate of at least $1/2$ known for the Euler-Maruyama scheme $\eul_{n,1}$ in the classical case of Lipschitz continuous coefficients. It was therefore a challenging question whether the error bound from \cite{LS18} can be improved and if so, whether the Euler-Maruyama scheme $\eul_{n,1}$ even achieves an $L_p$-error rate of at least $1/2$ for all $p\geq 1$ in the above setting. 
	
	Recently, this question was answered to the positive in \cite{MGY20}   for one-dimensional SDEs, i.e., for $d=1$, in the case when the drift coefficient  $\mu$ has finitely many  points of discontinuity, i.e., the exceptional set $\Theta$ of $\mu$ is  of the form
	\begin{equation}\label{The}
	\Theta=\{x_1, \ldots, x_K\}\subset\R.
	\end{equation}
	In this case, 
	with $x_1 < \dots < x_K$, the drift coefficient
	$\mu$  is Lipschitz continuous on each of the intervals $(x_{k}, x_{k+1})$, $k=0, \ldots, K$, where 
	$x_0=-\infty$ and $x_{K+1}=\infty$, 
	and   $\sigma$ is Lipschitz continuous and non-degenerate at the points of discontinuity of $\mu$, see Remark \ref{remdisc}.
	For such SDEs the upper bound \eqref{Leul} was proven in \cite{MGY20}.
	
	   In the present article, we answer 
	   the above
	   question to the positive (up to an arbitrary small exponent $\delta>0$)  for all $d\in\N$. More precisely, we show that 
	for all $p\geq 1$ and all $\delta>0$ there exists $c>0$ such that for all $n\in\N$, 
	\begin{equation}\label{UBE}
		\bigl(\EE\bigl[\|X_1-\eul_{n,1}\|^p\bigr]\bigr)^{1/p}\leq \frac{c}{n^{1/2-\delta}},
	\end{equation}
	i.e., the Euler-Maruyama scheme $\eul_{n,1}$ achieves an $L_p$-error rate of at least $1/2-$ for all $p\geq 1$.
    This upper bound 
    is a direct consequence of
    our main result, Theorem \ref{Thm1}, which states that for all $p\geq 1$ and all $\delta>0$ the supremum error of the time-continuous Euler-Maruyama scheme achieves  
    a rate of at least  $1/2-\delta$ in the $L_p$-sense, see Section \ref{ErrEs}.

	We furthermore study the performance of the piecewise linear  interpolation of the  time-discrete Euler-Maruyama scheme 
		$(\eul_{n,i/n})_{i=0, \ldots, n}$
	globally on the time interval $[0,1]$. Using Theorem \ref{Thm1} we  show that  for all $p\geq 1$ and all $\delta>0$ the supremum error 
	in $p$-th mean
	of the piecewise linear  interpolated  Euler-Maruyama  scheme 
	$(\eul_{n,i/n})_{i=0, \ldots, n}$
	is at least of order $1/2-\delta$ in terms of $n$,
	see Theorem \ref{Thm2} in  Section \ref{ErrEs}.
	
	We add that for $d$-dimensional SDEs \eqref{SDE}, 
	it was recently shown in \cite{DGL22} that  the 
	classical
	Euler-Maruyama scheme $\eul_{n,1}$ also achieves an $L_p$-error rate of at least $1/2-$ for all $p\geq 1$ in the case when the drift coefficient $\mu$ is   measurable and bounded  and the diffusion coefficient $\sigma$ is  bounded, uniformly elliptic and twice continuously differentiable with bounded partial derivatives of order
	$1$ and $2$. Moreover, for 
	SDEs \eqref{SDE} with
	additive noise,
	an $L_p$-error rate of at least  $1/(2\max(2,d,p))+1/2-$ for all $p\geq 1$ was shown in \cite{DGL22}  for the Euler-Maruyama scheme $\eul_{n,1}$ 
	in the case when the drift coefficient is of the form
	$\mu=\sum_{i=1}^m f_i \one_{K_i}$ with bounded
	Lipschitz domains $K_1, \ldots, K_m\subset \R^d$ and bounded Lipschitz continuous functions $f_1, \ldots, f_m\colon\R^d\to\R^d$ 
	for some $m\in\N$.
	The proof of these  results in~\cite{DGL22}
relies on the uniform ellipticity of $\sigma$ 
	and uses the stochastic sewing technique introduced in~\cite{Le2020}. In contrast, the proof of Theorem~\ref{Thm1} is based on a detailed analysis of the expected total time
	that
	 the actual position of the time-continuous Euler-Maruyama scheme and its position at the preceding time point on the grid are on 'different sides' of  the hypersurface $\Theta$, see Proposition~\ref{occtime2}, a new It\^{o} formula for a class of functions
	 $f\colon \R^d \to \R$
	  not globally $C^2$, see Theorem~\ref{Ito_new},  and the transformation approach introduced in~\cite{LS17}. 
	
	We furthermore add that recently, in \cite{MGY19b, Y21}, 
	higher-order methods  for approximation of  one-dimensional SDEs \eqref{SDE} with a discontinuous drift  coefficient were constructed 
	for the first time. 
	More precisely, in \cite{MGY19b} a transformation-based Milstein-type scheme was introduced, which is based on evaluations of $W$ at the uniform grid $\{0, 1/n, \ldots, 1\}$ and   achieves for all $p\geq 1$ an $L_p$-error rate of at least $3/4$   in terms of $n$ in the setting considered in 
	the present 
	article with $d=1$ and $\Theta$ given by \eqref{The} if, additionally, $\mu$ and $\sigma$ have a Lipschitz  continuous derivative on each of the intervals $(x_{k}, x_{k+1})$, $k=0, \ldots, K$. Moreover, in  \cite{Y21} an adaptive transformation-based Milstein-type scheme was constructed which achieves for all $p\geq 1$ an $L_p$-error rate of at least $1$   in terms of the average number of evaluations of $W$ used by the scheme under the same assumptions on the coefficients as in \cite{MGY19b}. Note that for such SDEs an  $L_p$-error rate better  than $3/4$  can not be achieved in general  by no numerical method based on $n$ evaluations of $W$ at fixed time points, see \cite{Ell24, MGY21} for  matching lower error bounds, and an   $L_p$-error rate better  than $1$  can not be achieved in general  by no numerical method based on $n$ sequentially chosen evaluations of $W$ on average, see \cite{hhmg2019, m04} for matching lower error bounds. See also 
	\cite{MGY24} for a recent survey on the complexity of $L_p$-approximation of one-dimensional SDEs with a discontinuous drift coefficient. 
	 The extension of the upper bounds from \cite{MGY19b, Y21}  to an appropriate subclass of $d$-dimensional SDEs considered in the present article 
	 using techniques developed in this
	 article will be  the subject of future work.

We briefly describe the content of the paper. 
The precise assumptions on the coefficients $\mu$ and $\sigma$,
 the existence and uniqueness result, Theorem~\ref{exist},
 as well as our error estimates, Theorem~\ref{Thm1} and Theorem~\ref{Thm2}, are stated in Section~\ref{ErrEs}. Section~\ref{Proofs} contains 
the
proofs of these results.
In Section~\ref{Exmpl}, we present some examples. Section~\ref{Num} is devoted to numerical experiments. 
In Section~\ref{Appendix} we 
state 
a number of 
results from differential geometry that are used for our proofs in Section~\ref{Proofs}.

\section{Setting and main results}\label{ErrEs}

We first briefly recall the notions of 
a hypersurface,
a tangent vector, a normal vector, the orthogonal projection and the reach 
of a set 
from differential geometry as well as the notion of piecewise Lipschitz continuity 
introduced in~\cite{LS17}.

Let $\emptyset\neq \Theta \subset \R^d$ 
and	let $k \in \N_{0} \cup \{\infty \}$. Let $x\in \Theta$, let $U,V \subset \R^{d}$ be open with $x\in U$ and let $\phi\colon U \to V$ be a $C^{k}$-diffeomorphism with $\phi(\Theta \cap U) = \R^{d-1}_{0} \cap V$, where 
\[
\R^{d-1}_0 = \begin{cases} \R^{d-1}\times\{0\}, & \text{if }d\ge 2,\\ \{0\}, &\text{if } d=1.\end{cases}
\]
Then
$(\phi,U)$ is called a $C^{k}$-chart for $\Theta$ at $x$. 
The set
$\Theta$ is called a $C^{k}$-hypersurface if for all $x \in \Theta$ there exists a $C^{k}$-chart for $\Theta$ at $x$.

 If $x\in \Theta$ then $v\in \R^d$ is called a tangent vector to $\Theta$ at $x$ if there exist $\varepsilon\in (0,\infty)$ and a $C^1$-mapping $\gamma\colon (-\varepsilon,\varepsilon) \to \Theta$ such that $\gamma(0)=x\text{ and }\gamma'(0) = v$. 
The set 
	\[
	T_x(\Theta) = \{v\in\R^d\mid v\text{ is a tangent vector to $\Theta$ at $x$}\}
	\]
	is called the tangent cone of $\Theta$ at $x$.
	It is well known that $T_{x}(\Theta)$ is a $(d-1)$-dimensional vector space if $\Theta$ is a $C^1$-hypersurface.

A function $\nor\colon \Theta \to \R^{d}$ is called a normal vector along $\Theta$ if $\nor$ is continuous, $\|\nor\| =1$ and  $\langle \nor(x),v\rangle = 0$ for every $x\in\Theta$ and every tangent vector $v$ to $\Theta$ at $x$, where 
$\langle\cdot, \cdot\rangle$ denotes  the Euclidean scalar product.
The set $\Theta$ is called orientable if there exists a normal vector along $\Theta$.

The Lipschitz continuous mapping
\[
d(\cdot, \Theta)\colon\R^d\to [0,\infty),\,\, x\mapsto \inf\{\|y-x\|\mid y\in \Theta\}
\]
is called the distance function of $\Theta$. The set
\[
\unp(\Theta) = \{x\in \R^d \mid \exists_1 y\in\Theta\colon \|y-x\|=d(x,\Theta)\}
\]
consists of all points in $\R^d$ that have a unique nearest point in $\Theta$ and the mapping
\[
\pr_\Theta\colon \unp(\Theta)\to \Theta,\,\, x\mapsto \text{argmin}_{y\in \Theta} \|x-y\|
\]
is called the orthogonal projection onto $\Theta$.  
For  $\eps\in [0,\infty)$, the $\eps$-neighbourhood of $\Theta$ is given by the open set
\[
\Theta^\eps = \{x\in\R^d\mid d(x,\Theta)< \eps \},
\]
and the quantity
\[
\reach(\Theta) = \sup\{\eps\in [0,\infty)\mid \Theta^\eps\subset \unp(\Theta)\} \in [0,\infty]
\]
is called the reach of $\Theta$. 
The set
$\Theta$ is said to be of positive reach if $\reach(\Theta)>0$.
Note that
$\reach(\Theta)>0$ implies that $\Theta$ is closed.

Next, recall that the length of a continuous function $\gamma\colon  [0,1] \to \R^{d}$ is defined by
\[
l(\gamma) = \sup \left\{ \sum_{k = 1}^{n} \|\gamma(t_{k}) - \gamma(t_{k-1})\| \mid 0 \leq t_{0} < \dots < t_{n} \leq 1,\, n\in \N   \right\} \in [0,\infty]
\]
and that for 
$\emptyset\neq A \subset \R^d$,
 the intrinsic metric $\rho_A\colon A\times A \to [0,\infty]$ is given by
\[ 
\rho_A(x,y) = \inf\{l(\gamma) \mid \gamma\colon  [0,1] \to A \text{ is continuous with } 
\gamma(0) = x \text{ and } \gamma(1) = y \},\quad x,y\in A.
\]
Note that $\rho_A$ is an extended metric, i.e. $\rho_A$ is definite, symmetric and satisfies the triangle inequality 
but may take the value $\infty$.

Let $\emptyset\neq A\subset D\subset \R^d$ and $m,k\in\N$. 
A  function $f\colon  D \to \R^{k\times m}$ is called intrinsic Lipschitz continuous on $A$, if there exists $L \in (0,\infty)$ such that  for all $x,y \in A$ we have $\|f(x) - f(y)\| \leq L \rho_A(x,y)$. In this case, $L$ is called an intrinsic Lipschitz constant for $f$ on $A$. If 
$f$ is intrinsic Lipschitz continuous on $D$
then $f$ is called intrinsic Lipschitz continuous.

A function $f\colon  \R^{d} \to \R^{k\times m}$  is called piecewise Lipschitz continuous if there exists a hypersurface $\emptyset\neq\Theta \subset \R^{d}$ such that $f$ is intrinsic Lipschitz continuous on $\R^d\setminus\Theta$. In this case, the hypersurface $\Theta$ is called an exceptional set for $f$.

We assume that the drift coefficient $\mu$ and the diffusion coefficient $\sigma$ of the SDE \eqref{SDE} satisfy the following conditions.
\begin{itemize}
	\item [(A)] There 
	exist 
	a $C^{4}$-hypersurface $\emptyset\neq\Theta\subset \R^d$
 of positive reach
	and a normal vector $\nor$ along $\Theta$
	such that\\[-.3cm]
	\begin{itemize}
			\item[(i)] there exists an open neighbourhood $U\subset \R^d$ of $\Theta$ such that $\nor$ can be extended to a $C^3$-function $\nor\colon U\to\R^d$ that has bounded partial derivatives up to order $3$ on $\Theta$,\\[-.3cm]
			\item[(ii)] $\inf_{x\in \Theta} \|\normal(x)^\top \sigma(x)\| > 0$,\\[-.3cm]
		\item[(iii)] there exists   an open neighbourhood $U\subset \R^d$ of $\Theta$ such that the function
		\[
		\alpha\colon \Theta\to\R^d,\,\, 	x\mapsto \lim_{h \downarrow 0}\frac{\mu(x-h\normal(x))- \mu(x+h\normal(x))}{2 \|\sigma(x)^\top\normal(x)\|^2}\\[.1cm]
			\]
			can be extended to a $C^3$-function $\alpha\colon U\to\R^d$ that has bounded partial derivatives up to order $3$ on $\Theta$,\\[-.3cm]
					\item[(iv)] there exists $\eps\in (0,\reach(\Theta))$ such that $\mu$ and $\sigma$ are bounded on $\Theta^{\eps}$,\\[-.3cm]
			\item[(v)] $\mu$ is piecewise Lipschitz continuous with exceptional set $\Theta$.\\[-.3cm]
	\end{itemize}
\item[(B)] $\sigma$ is Lipschitz continuous.
\end{itemize}

\begin{Rem}\label{normreg00}
Note that, by Lemma~\ref{normalreg0} in the appendix, every normal vector $\normal$ along a $C^4$-hypersurface $\Theta$ is a $C^3$-mapping. By \cite[Remark 1.1]{Fu95}, the mapping $\nor$ can thus be extended to a $C^3$-mapping on an open neighbourhood of $\Theta$.
Thus, the condition (A)(i) is 
the condition that 
the extension of 
the normal vector $\normal$ has bounded partial
 derivatives 
 up to order  $3$ on $\Theta$. 
\end{Rem}

\begin{Rem}\label{remcond00}
	For the purpose of later use we note that the condition (A)(ii) 
	is equivalent to the condition
		that there exists  $\eps\in (0,\reach(\Theta))$ such that 
	\begin{equation}\label{remcond0}
		\inf_{x\in \Theta^{\eps}} \|\normal(\pr_\Theta(x))^\top \sigma(x)\| > 0.
		\end{equation}
	Indeed, clearly \eqref{remcond0} implies (A)(ii). Next, assume that (A)(ii) holds.  By the Lipschitz continuity of $\sigma$ there exists $K>0$ such that $\|\sigma(x) -\sigma(y)\|\leq K\|x-y\|$ for all $x,y\in\R^d$. Let $\eps \in (0,\reach(\Theta))$ with $\eps \le 	\inf_{x\in \Theta} \|\normal(x)^\top \sigma(x)\|/(K+1)$. Then, for all  $x\in \Theta^\eps$,
	\begin{align*}
		\|\normal(\pr_\Theta(x))^\top \sigma(x)\| & \ge 	\|\normal(\pr_\Theta(x))^\top \sigma(\pr_\Theta(x))\| -	\|\normal(\pr_\Theta(x))^\top (\sigma(x) - \sigma(\pr_\Theta(x)))\| \\
		& \ge (K+1)\eps - \|\sigma(x) - \sigma(\pr_\Theta(x))\| \\
		& \ge (K+1)\eps - K\|x - \pr_\Theta(x)\| > \eps,
	\end{align*}
which yields \eqref{remcond0}.

We furthermore provide a brief motivation of the condition (A)(ii): for $x\in \unp(\Theta)$  and a $d$-dimensional standard normal random vector $Z=(Z_1,\dots,Z_d) $, the random vector $\sigma(x) Z$ has the component 
\[
\langle \sigma(x) Z, \nor(\pr_{\Theta}(x))\rangle \nor(\pr_{\Theta}(x)) = \Bigl(\sum_{j=1}^d  \langle \sigma_j(x), \nor(\pr_{\Theta}(x))\rangle Z_j\Bigr) \cdot \nor(\pr_{\Theta}(x))
\]
in the direction of $\nor(\pr_{\Theta}(x))$, i.e. orthogonal to the tangent space of $\Theta$ at $\pr_{\Theta}(x)$, and, furthermore, 
\[
\langle \sigma(x) Z, \nor(\pr_{\Theta}(x))\rangle \,\sim \,\text{N}(0, V(x))
\]
with variance $V(x) =  \sum_{j=1}^d  \langle \sigma_j(x), \nor(\pr_{\Theta}(x))\rangle^2 = \|\normal(\pr_\Theta(x))^\top \sigma(x)\|^2$,
where $\sigma_j(x)$ denotes the $j$-th column of $\sigma(x)$ for $j\in\{1, \ldots, d\}$.
Observing~\eqref{remcond0}, the
condition (A)(ii) thus 
ensures that there is a neighborhood $\Theta^{\varepsilon}$ of $\Theta$ such that, roughly speaking, on $\Theta^{\varepsilon}$ the solution is pushed away from $\Theta$ with positive minimum probability. This condition 
is essential for many parts of our proofs. In particular,  it implies the non-degeneracy of $\sigma$ on $\Theta^{\varepsilon}$, which is needed to guarantee the existence of a solution of the SDE~\eqref{SDE}. See, e.g.  \cite{LTS15} for  a counter example 
with respect to the 
existence in 
dimensions $d=1$ and $d=2$. 
	\end{Rem}

\begin{Rem}\label{remcond}
	We briefly 	motivate the condition (A)(iii).
	We first note that the function $\alpha$ is well-defined, see Lemma~\ref{exlim}
	in Section \ref{propcoeff}.
	For $x \in \Theta$, the value $\alpha(x)$ is essentially given by the jump of the drift coefficient $\mu$  in $x$, divided by
	twice
		the variance $V(x)$ of the component of  the random vector $\sigma(x) Z$ in the direction of  $\nor(x)$, see the above discussion
		in Remark~\ref{remcond00},
		 and can thus be interepreted as the intrinsic difficulty of pushing the solution away from $x$ in 
		 orthogonal direction to the tangent space of $\Theta$ at $x$ in terms of the irregularity of $\mu$  and the strength of the favourable noise at $x$.
	The function $\alpha$ is used to construct a suitable transformation that removes the discontinuity from the drift coefficient, see Section \ref{trans}. The regularity of $\alpha$ is needed to apply an It\^{o} formula in connection with this transformation, see the proof of Theorem~\ref{exist} in Section \ref{Exist} and the proof of Theorem~\ref{Thm1} in Section \ref{ProofThm1}.
\end{Rem}

\begin{Rem}\label{remdisc}
	In the one-dimensional case, i.e. $d=1$, it is easy to check that the conditions (A) and (B) are equivalent to the following three conditions:
	\begin{itemize}
		\item[(i)] $\emptyset\neq\Theta \subset \R$ is countable with $\delta:=\inf\{|x-y|\mid x,y\in \Theta, x\neq y\} > 0$,
		\item[(ii)] $\sigma\colon\R\to\R$ is Lipschitz continuous with
		\[
		0 < \inf_{x\in \Theta} |\sigma(x)| \le  \sup_{x\in \Theta} |\sigma(x)| < \infty,
		\]
		\item[(iii)] $\mu\colon\R\to\R$ is Lipschitz continuous on each of the countable intervals $(x,y) \subset \R$ with $x,y\in \Theta\cup\{-\infty,\infty\}$ and $(x,y)\cap \Theta = \emptyset$ and for every $x\in\Theta$  there exist $y_x,z_x\in \R$ with $x-\delta < y_x < x< z_x < x+\delta$ such  that
		\[
		\sup_{x\in \Theta} ( |\mu(x)| + |\mu(y_x)| + |\mu(z_x)|) < \infty.
		\]
	\end{itemize}
	A particular instance of (i)-(iii) is given by $\Theta = \{x_1,\dots,x_K\}$ with $-\infty = x_0 < x_1 < \dots < x_K < x_{K+1} = \infty$
	and $K\in\N$
	 such that $\sigma(x_k)\neq 0 $ for every $k\in \{1,\dots,K\}$ and $\mu$ is Lipschitz continuous on $(x_k,x_{k+1})$ for every $k\in \{0,\dots,K\}$. The latter setting is studied in 
	 \cite{LS16, LS17, MGY20, MGY19b, Y21}
\end{Rem}

\begin{Rem}\label{compLS}
	We compare the conditions (A) and (B) with the conditions employed in the correction note~\cite{LScorr1} 
		of~\cite{LS17} to obtain existence and uniqueness of a solution of the SDE~\ref{SDE}, and with the conditions employed in the corrected version~\cite{LScorr2} 
	of~\cite{LS18} to obtain 
	an $L_2$-error estimate
	for the Euler-Maruyama scheme. 
	\begin{itemize}
		 \item[(i)] In both cases, the authors assume, additionally, that the hypersurface $\Theta$ consists of finitely many connected components. 
	 \item[(ii)] On the other hand, 
	in place of (A)(i) and (A)(iii) they only require that the mappings $\nor,\alpha\colon \Theta\to \R^d$ are $C^{3}$ with bounded derivatives up to order $3$. Note that in this case $\nor,\alpha$  can always be extended to 
	$C^{3}$-mappings $\widetilde\nor, \widetilde \alpha$  on an open neighbourhood of $\Theta$, see e.g. \cite[Remark 1.1]{Fu95}, however, boundedness of the derivatives of $\widetilde\nor, \widetilde \alpha$ on $\Theta$ is stronger than boundedness of the derivatives of $\nor,\alpha$ on $\Theta$, because the latter derivatives are only acting on the tangent spaces of $\Theta$.
		\end{itemize}
		Except for (i) and (ii), the assumptions used 
	in~\cite{LScorr1} 
	coincide with our assumptions (A) and (B). 
		In~\cite{LScorr2} 
	the authors furthermore assume, in contrast to (A)(iv), that the coefficients $\mu$ and $\sigma$ are  bounded.
\end{Rem}

We turn to our results.

\begin{Thm} \label{exist}
	Assume that $\mu$ and $\sigma$ satisfy $(A)$ and $(B)$. Then the SDE \eqref{SDE} has a unique strong solution $X$.
\end{Thm}

\begin{Rem}\label{wrongito}
	Theorem~\ref{exist} is already stated 
	and proven
	in~\cite{LS17} and 
	the associated correction note~\cite{LScorr1},
    see Theorem 3.21 and 
	Theorem 6,
	respectively.  
	In both cases, the proofs heavily rely on the use of 
	~\cite[Theorem 2.9]{LTS15},
	 which provides an It\^{o} formula for functions 
	$f\colon\mathcal D\to \R$ 
	with $\mathcal D\subset \R^d$ open, that are not globally $C^2$,
	see the proof of ~\cite[Theorem 3.19]{LS17}.
		The It\^{o} formula~\cite[Theorem 2.9]{LTS15}, however, is easily seen to be wrong. 
		Indeed, take, $X=W$ and $\mathcal D = \{ x \in \mathbb{R}^d \mid \|x\| < 1 \}$. Then the statement in~\cite[Theorem 2.9]{LTS15} reads:
	\begin{equation}\label{wrongito}
		\forall t \geq 0: f(W_t) = f(0) + \sum_{i=1}^{d} \int_{0}^{t \wedge \zeta} \frac{\partial}{\partial x_i} f(W_s) dW_s^i + \frac{1}{2} \sum_{i=1}^{d} \int_{0}^{t \wedge \zeta} \frac{\partial^2}{\partial x_i^2} f(W_s) ds,
	\end{equation}
	where $\zeta=\inf\{t>0\mid W_t\not\in \mathcal D\}$. The term $f(W_t)$ on the left side of \eqref{wrongito} is however  undefined, since $\PP(W_t \notin \mathcal D) > 0$ for every $t > 0$. Replacing the left side of \eqref{wrongito} by  $f(W_{t \wedge \zeta})$ does not help. Since $\mathcal D$ is open and bounded and $W$ has continuous paths we have $W_{\zeta}\notin\mathcal D$, and therefore, for all $t>0$,
	\[
	\PP(W_{t \wedge \zeta}\notin\mathcal D)\geq \PP(\zeta\le t, W_\zeta\notin \mathcal D)
	=
	\PP(\zeta\le t)\geq \PP(W_t\notin\mathcal D)>0.
	\]
	Assuming $f$ to be defined on the whole of $\R^d$ does not help either. Take $f=\ind_{\mathcal D}$. Then $f$ satisfies all of the assumptions of~\cite[Theorem 2.9]{LTS15}  with respect to its behaviour on $\mathcal D$. The right side of \eqref{wrongito} is $\ind_{\mathcal D}(0)=1$ while the (corrected) left side $\ind_{\mathcal D}(W_{t \wedge \zeta})$ is zero with positive probability.
	
	There is 
	a corrected version of~\cite{LTS15} 
	available
	on arXiv, which contains a corrected version of 
	the It\^{o} formula~\cite[Theorem 2.9]{LTS15}
	for functions $f\colon\R^d\to \R$  that are not globally $C^2$,
	see~\cite[Theorem 2.9]{LSTcorr}.  
	However, functions $f$ to which the It\^{o} formula is applied in the proof of~\cite[Theorem 3.19]{LS17} are only defined locally, on an open rectangle, and it is not clear to us whether it is possible to extend these functions to the whole of $\R^d$ so that the assumptions of the It\^{o} formula~\cite[Theorem 2.9]{LSTcorr} are fulfilled.

	 Furthermore, the proof of~\cite[Theorem 3.19]{LS17} seems to be based on an iterative procedure that depends on whether the current state of the solution is in $\R^d\setminus\Theta$ or in $\Theta$. In the first case the classical It\^{o} formula is applied, and in the second case the  It\^{o} formula~\cite[Theorem 2.9]{LTS15} is applied. Despite the technical problems with the application of the latter It\^{o} formula described above, it is unclear to us how this iterative procedure is defined exactly and whether it terminates with probability one.

	 We therefore provide a complete proof of Theorem~\ref{exist} in the present paper, see Section~\ref{Exist}. This proof is based on a new  It\^{o} formula for functions $f\colon\R^d\to \R$  not globally $C^2$,
		 see Theorem~\ref{Ito_new} in Section~\ref{ito},
	 that can also be used to estimate the $L_p$-distance of the Euler-Maruyama scheme for a transformed SDE and the associated transformed Euler-Maruyama scheme for the original SDE. 
	\end{Rem}

For $n\in\N$ let $\eul_{n}=(\eul_{n,t})_{t\in[0,1]}$  denote  the time-continuous Euler-Maruyama scheme with step-size $1/n$ associated to the SDE \eqref{SDE}, i.e. $\eul_{n}$ is recursively given by
$\eul_{n,0}=x_0$ and
\[
\eul_{n,t}=\eul_{n,i/n}+\mu(\eul_{n,i/n})\, (t-i/n)+\sigma(\eul_{n,i/n})\, (W_t-
W_{i/n})
\]
for $t\in (i/n,(i+1)/n]$ and $i\in\{0,\ldots,n-1\}$. 
For $f\colon[0,1]\to\R^d$ let $\|f\|_{\infty}=\sup\{\|f(t)\|\mid t\in[0,1]\}$ denote the supremum norm of $f$.
 We have the following  estimate for the supremum error of $\eul_{n}$.

\begin{Thm}\label{Thm1} Assume that $\mu$ and $\sigma$ satisfy $(A)$ and $(B)$.
	For every $p\in [1,\infty)$ and every $\delta\in (0,\infty)$ there exists $c\in(0, \infty)$ such that for all $n\in\N$, 
	\begin{equation}\label{l3}
		\bigl(\EE\bigl[\|X-\eul_{n}\|_\infty^p\bigr]\bigr)^{1/p}\leq \frac{c}{n^{1/2-\delta}}. 
	\end{equation}
\end{Thm}

Next, we study the performance of the piecewise linear interpolation 
$\overline{X}_n = (\overline X_{n,t})_{t\in [0,1]}$ 
of the time-discrete Euler-Maruyama scheme $(\widehat X_{n,i/n})_{i=0,\dots,n}$, i.e. 
\[
\overline X_{n,t} = (nt-i)\,\widehat X_{n,(i+1)/n} + (i+1-n t)\, \widehat X_{n,i/n}
\]
for $t\in [i/n,(i+1)/n]$
and $i\in\{0,\ldots,n-1\}$. 
We have the following  estimate for the supremum error of $\overline X_n$.

\begin{Thm}\label{Thm2} Assume that $\mu$ and $\sigma$ satisfy $(A)$ and $(B)$.
		For every $p\in [1,\infty)$ and every $\delta\in (0,\infty)$
	there exists $c\in(0, \infty)$ such that for all $n\in\N$, 
	\begin{equation}\label{l4}
		\bigl(\EE\bigl[\|X-\overline X_n\|_\infty^p\bigr]\bigr)^{1/p}\leq 
			\frac{c}{n^{1/2-\delta}}.
		\end{equation}
\end{Thm}

\section{Proofs}\label{Proofs}

In this section we provide proofs of  Theorem~\ref{exist}, Theorem~\ref{Thm1} and Theorem~\ref{Thm2}.

We briefly describe the 
structure
of the section. In Section~\ref{Notation} we introduce some notation. In Section~\ref{propcoeff} we prove the linear growth property of $\mu$ and $\sigma$ as well as the existence of the limit on the right hand side in condition (A)(iii). In 
Section~\ref{trans} we provide the construction of the transformation $G\colon\R^d\to\R^d$ that is used to switch from the SDE \eqref{SDE} to an SDE with Lipschitz continuous coefficients and prove its crucial properties.  Section~\ref{ito}  contains a new It\^{o} formula for a class of functions
$f\colon \R^d \to \R$
not globally $C^2$. Applying this It\^{o} formula with the
transformation $G$ and its inverse $G^{-1}$, we prove in Section~\ref{Exist}  the existence and uniqueness result, Theorem~\ref{exist}. In Section~\ref{MomOcc} we provide moment estimates and occupation time estimates
for the time-continuous Euler-Maruyama scheme $\eul_n$. Using these extimates as well as the 
It\^{o} formula from Section~\ref{ito} we prove  Theorem~\ref{Thm1}  in Section~\ref{ProofThm1}. Section~\ref{ProofThm2} contains the proof of Theorem~\ref{Thm2}.

\subsection{Notation}\label{Notation}
For a matrix $A\in\R^{d\times m}$ we use $\|A\|$ to denote the Frobenius norm of $A$, $\text{Ker}(A) = \{x\in\R^m\mid Ax=0\}$ to denote the null space of $A$, $A_{j}$ to denote the $j$-th column of $A$ for $j \in \{1,\dots,m\}$ and 
\[
\text{vec}(A)=\begin{pmatrix}
		A_{1} \\
		\vdots \\
		A_{m}
	\end{pmatrix}
\in \R^{md}
	\]
to denote the vector 
obtained by concatenation
of the columns of $A$.
In the case $ d=m$ we use $\tr(A)$ to denote the trace of $A$ and $\text{det}(A)$ to denote the determinant of $A$. For $x\in\R^{d^2}$ we put
$\text{mat}(x)=(x_{i+(j-1)d})_{1\leq i,j\leq d}
\in \R^{d\times d}$.
Thus, $\text{vec} (\text{mat} (x) ) =x$.

For $x,y\in\R^d$ we use  
$\langle x,y\rangle$ to denote the Euclidean scalar product of $x$ and $y$ and
$ \overline{x,y} = \{\lambda x + (1-\lambda)y \mid  \lambda \in [0,1] \}
\subset \R^d
$ to denote the straight line connecting  $x$ and $y$. Furthermore, we use $S^{d-1} = \{x\in\R^d\mid \|x\|=1\}$ to denote the unit sphere in $\R^d$. For $r\in [0,\infty)$ and $x\in \R^d$ we use $B_{r}(x) = \{y \in \R^{d} \mid \|x-y\| < r \}$ to denote the open ball and $\overline{B}_{r}(x) = \{y \in \R^{d} \mid \|x-y\| \leq r \}$ to denote the closed ball with center $x$ and radius $r$, respectively.

For a set $U \subset \R^{d}$ we write $\inter(U)$, $\cl(U)$ and $\partial U$ for the interior, the closure and the boundary of $U$, respectively. For  a function $f \colon U \to \R^{m}$ and a set 
$ M\subset U$ 
we use $\|f\|_{\infty,M} = \sup\{\|f(x)\|\mid x\in M\}$ 
to denote the supremum 
of  
the
values of 
$\|f\|$ on $M$
 and we put
$\|f\|_{\infty}=\|f\|_{\infty,U}$.

For a multi-index $\alpha=(\alpha_1,\dots,\alpha_d)^\top\in \N_0^d$ we put $|\alpha| = \alpha_1+\dots + \alpha_d$.
For
a set $ U\subset \R^d$, an open set  $\emptyset\neq M\subset U$,
 $k\in\N_0$ and a function $f=(f_1,\dots,f_m)^\top \colon U \to \R^{m}$, which is a $C^{k}$-function on $M$, we put
\[ 
\|f^{(\ell)}(x)\|_\ell = \max_{i\in\{1,\dots,m\}}\max_{\alpha \in \N_0^{d}, \abs{\alpha} = \ell} |f_i^{(\alpha)}(x) |
\] 
for $\ell \in \{0,1,\dots,k \}$ and $x \in M$. 
If $k\ge 1$ then we use
\[
f'\colon M\to \R^{m\times d},\,\, x\mapsto \Bigl(\frac{\partial f_i}{\partial x_j}(x)\Bigr)_{\substack{1\le i\le m \\ 1\le j \le d}}\in\R^{m\times d}
\]
to denote the first derivative of $f$ on $M$. If $m=1$ and $k\ge 2$ then we use 
\[
f''\colon M\to \R^{d\times d},\,\, x\mapsto \Bigl(\frac{\partial^2 f}{\partial x_{j_1} \partial x_{j_2}}(x)\Bigr)_{1\le j_1,j_2\le d}\in\R^{d\times d}
\]
to denote the second derivative of $f$ on $M$.	

For a function $f\colon \R^{d}\to \R$ we 
use $\supp(f) = \text{cl}(\{x \in \R^{d} \mid f(x) \neq 0 \})$ to denote the support of $f$.
\subsection{Properties of the coefficients}\label{propcoeff}

We first prove the linear growth property of $\mu$ and $\sigma$ stated in Remark~\ref{remdisc}.

\begin{Lem}\label{lingrowth} Let $\emptyset\neq\Theta\subset \R^d$ be a  $C^1$-hypersurface of positive reach 
 and assume that $\mu$ and $\sigma$ satisfy (A)(iv),(v) and (B).
Then there exists $c\in (0,\infty)$ such that for all $x\in\R^d$,
	\[
	\|\mu(x)\| + \|\sigma(x)\| \le c\, (1+\|x\|).
	\]
\end{Lem}

\begin{proof}
	Since $\sigma$ is Lipschitz continuous, we immediately obtain that $\sigma$ is of at most linear growth. 
	
	According to (A)(iv) there exists $\eps\in (0,\reach(\Theta))$ such that  $\mu$ is bounded on $\Theta^{\eps}$. It thus remains to show that $\mu$ is of at most linear growth on $\R^{d}\setminus \Theta^{\eps}$. 
	Fix $\theta\in\Theta$. Let $x\in \R^{d}\setminus \Theta^{\eps}$. Then $B_{d(x,\Theta)}(x) \cap \Theta = \emptyset$ and there exists $ y \in  B_{d(x,\Theta)}(x) \cap \Theta^{\eps}$. We conclude that $\overline{x,y} \subset B_{d(x,\Theta)}(x) \subset \R^{d}\setminus \Theta$, which implies 
	$\rho_{\R^{d}\setminus \Theta}(x,y) = \|x-y\|$,
	see Lemma~\ref{Lipconnec0} in the appendix.
	
	By (A)(v), $\mu$ is intrinsic Lipschitz continuous on $\R^{d}\setminus \Theta$. Let $L \in (0,\infty)$ 
	be a corresponding intrinsic Lipschitz constant. 
	Put $c_{1} = \sup_{z \in \Theta^{\eps}} \|\mu(z)\|\in [0,\infty)$. Then
	\begin{align*}
		\|\mu(x)\| & \le \|\mu(x) - \mu(y)\| + \|\mu(y)\| \le L \|x-y\| + c_{1} \\
		&< Ld(x,\Theta) + c_{1} \leq L\|x-\theta\| + c_{1} \leq (L+L\|\theta\|+c_1)(1+\|x\|),
	\end{align*}
	which completes the proof.
\end{proof}
We briefly recall a well-known fact from differential geometry. 

Let $\emptyset\neq\Theta\subset\R^d$ be an orientable $C^1$-hypersurface of positive reach and  let  $\normal\colon \Theta\to\R^d$ be a normal vector along $\Theta$.
For $s\in\{+,-\}$ and $\eps\in (0,\reach(\Theta))$  put
\begin{equation}\label{u00}
Q_{\eps,s} = \{x+s\lambda\normal(x)\mid x\in\Theta, \lambda\in (0,\eps) \}.
\end{equation}
Since $\Theta$ is an orientable $C^1$-hypersurface of positive reach it follows that $Q_{\eps,+}$ and $Q_{\eps,-}$ are open and disjoint with 
\begin{equation}\label{u01}
	\Theta^\eps \setminus \Theta = Q_{\eps,+}\cup Q_{\eps,-},
\end{equation}
see Lemma~\ref{sides0} in the appendix. 

Using~\eqref{u01} we can prove the existence of the limit on the right hand side in condition (A)(iii), see Remark~\ref{remcond}.

\begin{Lem}\label{exlim} Let $\emptyset\neq\Theta\subset\R^d$ be an orientable $C^1$-hypersurface of positive reach   and assume that $\mu$  satisfies (A)(v). Let  $\normal\colon \Theta\to\R^d$ be a normal vector along $\Theta$. Then for every $x\in\Theta$ and every $s\in\{+,-\}$, the limit 
	\[
\lim_{h\downarrow 0} \mu (x+sh\normal(x))
	\]
	exists in $\R^d$.
\end{Lem}

\begin{proof}
	Let $x\in\Theta$, $s\in\{+,-\}$, $\eps\in (0,\reach(\Theta))$ and put 
	\[
	A=  \overline{x,x+s(\eps/2) \normal(x)}\setminus\{x\}. 
	\]
	By~\eqref{u01} we get
	\[
    A\subset \Theta^\eps\setminus \Theta\subset \R^d\setminus  \Theta. 
	\] 
	By (A)(v) we thus obtain  that $\mu$ is intrinsic Lipschitz continuous on  $A$. Since $A$ is convex, we  conclude by Lemma~\ref{lipimpl0}(ii) in the appendix that $\mu$ is Lipschitz continuous on $A$. 
	
	Let $(\lambda_n)_{n\in\N}$ be a  sequence in
	$(0,\varepsilon/2)$ 
	with $\lim_{n\to\infty} \lambda_n = 0$.  Then $x+s\lambda_n\normal(x)\in A$ for all $n\in \N$, and by the Lipschitz continuity of $\mu$ on $A$ we obtain that $(\mu(x+s\lambda_n\normal(x))_{n\in\N}$ is a Cauchy-sequence and hence has a limit $z\in\R^d$. If $(\tilde \lambda_n)_{n\in\N}$ is a further sequence in  $(0,\varepsilon /2)$ with $\lim_{n\to\infty} \tilde \lambda_n = 0$, then $\lim_{n\to\infty} (\lambda_n -\tilde\lambda_n) = 0$, and by the Lipschitz continuity of $\mu$ on $A$ we obtain that $\lim_{n\to\infty} (\mu(x+s\lambda_n\normal(x))-\mu(x+s\tilde \lambda_n\normal(x))) = 0$. Thus, the sequence $(\mu(x+s\tilde \lambda_n\normal(x)))_{n\in\N}$ converges to $z$ as well.
 \end{proof}

\subsection{The transformation $G$}\label{trans}

In this section we construct the bijection $G\colon\R^d\to\R^d$ that is used to transform  the SDE~\eqref{SDE} into an SDE with Lipschitz continuous coefficients and we provide its crucial properties.
We essentially  follow the construction in the corrected version~\cite{LScorr1} of~\cite{LS17}. Since the assumptions used in~\cite{LScorr1} 
differ from 
the conditions (A) and (B), 
see Remark \ref{compLS},
we provide a full proof of Proposition~\ref{G}.

\begin{Prop} \label{G}
	Let $\emptyset\neq\Theta\subset \R^d$ be an orientable $C^4$-hypersurface of positive reach and assume that $\mu$ and $\sigma$ satisfy (A) and (B). Then there exists a function $G\colon\R^d\to\R^d$ with the following properties.
	\begin{itemize}
		\item[(i)] $ G$ is a $C^{1}$-diffeomorphism.
		\item[(ii)] $G$, $G^{-1}$, $G'$, $(G^{-1})' $ are Lipschitz continuous and $G'$, $(G^{-1})' $ are bounded.
		\item[(iii)] $G=(G_1,\dots,G_d)^\top$ and $G^{-1}=(G^{-1}_1,\dots,G^{-1}_d)^\top$ are $C^2$-functions on $\R^{d} \setminus \Theta$ 
and for every $i\in\{1,\dots,d\}$, the funtions $G''_i, (G^{-1}_i)''\colon \R^{d} \setminus \Theta \to \R^{d\times d}$ are bounded and intrinsic Lipschitz continuous.
		\item[(iv)] The function 
		\[
		\sigma_{G} =(G'\,\sigma)\circ G^{-1} \colon \R^{d} \to \R^{d\times d}
		\]
		is Lipschitz continuous with $	\sigma_G(x) = \sigma(x)$ for every $x\in\Theta$ and it holds
		\[
		\sigma =((G^{-1})'\,\sigma_G)\circ G.
		\]
		\item[(v)] For every $i\in\{1,\dots,d\}$, the second derivatives $G''_i, (G^{-1}_i)''\colon \R^{d} \setminus \Theta \to \R^{d\times d}$ of $G_{i}$ and $G^{-1}_i$ on $\R^d\setminus \Theta$ can be extended to bounded mappings $ R_i\colon \R^d \to \R^{d\times d}$ and $ S_i\colon \R^d \to \R^{d\times d}$, respectively, such that the function
		\[
	    	\mu_G= \Bigl(G'\,\mu + \frac{1}{2}    	
	    	\bigl( \tr\bigl(R_1\,\sigma\sigma^{\top}\bigr),\dots,\tr\bigl(R_d\,\sigma\sigma^{\top}\bigr)\bigr)^\top      
	    	\Bigr)\circ G^{-1}\colon\R^d\to\R^d
		\]
		is Lipschitz continuous and it holds
		\[
		\mu = \Bigl((G^{-1})'\,\mu_G + \frac{1}{2} \bigl(\tr\bigl(S_1\,\sigma_G\sigma_G^{\top}\bigr),\dots, \tr\bigl(S_d\,\sigma_G\sigma_G^{\top}\bigr)\bigr)^\top\Bigr)\circ G.
		\]
	\end{itemize}
\end{Prop}

For the proof of Proposition~\ref{G}  we assume throughout the following that $\mu$ and $\sigma$ satisfy (A) and (B), we fix 
a $C^4$-hypersurface $\emptyset\neq\Theta\subset \R^d$ of positive reach 
and a normal vector $\nor$ along $\Theta$
according to (A), an open neighbourhood $U$ of $\Theta$ according to 
(A)(i),(iii) 
and $\eps^*\in (0,\reach(\Theta))$ such that 
(A)(iv) and \eqref{remcond0}
hold with $\eps = \eps^*$.

 First, we provide useful properties of the functions $\alpha$, $\pr_{\Theta}$, $\normal\circ\pr_{\Theta}$ and $\alpha\circ\pr_{\Theta}$. 

\begin{Lem} \label{propconc} The function $\alpha\colon U\to\R^d$ is bounded on $\Theta$. Moreover, there exists  $\tilde \eps \in (0,\reach(\Theta))$ such that the 
	functions $\pr_{\Theta}, \normal\circ \pr_{\Theta}, \alpha\circ \pr_{\Theta}\colon \text{unp}(\Theta)\to \R^d$ are $C^{3}$-functions on $\Theta^{\tilde \eps}\subset\text{unp}(\Theta)$ with 
	\begin{equation}\label{est00}
   \sup_{x\in \Theta^{\tilde \eps}}	\|f^{(\ell)}(x)\|_\ell < \infty
	\end{equation}
	for every $f\in \{\pr_{\Theta}, \normal\circ \pr_{\Theta}, \alpha\circ \pr_{\Theta}\}$ and every $\ell\in\{1,2,3\}$.
\end{Lem}
\begin{proof}
	Put $c_1= \inf_{x\in \Theta}\|\normal(x)^\top \sigma(x)\|$ and $c_2=\sup_{x\in \Theta^{\eps^*}}\|\mu(x)\|$. By (A)(ii) and (A)(iv) we have $c_1\in (0,\infty)$ and $c_2\in [0,\infty)$. Let $x\in\Theta$. For all $h\in (0,\eps^*)$ we have $x+h\normal(x), x-h\normal(x)\in \Theta^{\eps^*}$, and therefore
	\[
	\frac{\bigl\|\mu(x-h\normal(x))- \mu(x+h\normal(x)))\bigr\|}{2 \|\sigma(x)^\top\normal(x)\|^2}\le \frac{c_2}{c_1^2},
	\] 
	which implies $\|\alpha(x)\|\le c_2/c_1^2$.
	
	Since $\Theta$ is a $C^4$-hypersurface of positive reach, we get by Lemma~\ref{projdist}(i) in the appendix that $\pr_{\Theta}$ is a $C^3$-function on $\Theta^\delta$ for all  $\delta \in (0,\reach(\Theta))$. Since $\normal$ and $\alpha$ are $C^3$-functions
	on $U$, 
	we conclude that $\normal\circ\pr_{\Theta}$ and $\alpha\circ\pr_{\Theta}$ are $C^3$-functions on $\Theta^\delta$ for all  $\delta \in (0,\reach(\Theta))$ as well. Using the property (A)(i) of  $\nor\colon U\to\R^d$ we obtain by Lemma~\ref{pdiffbd0} in the appendix the existence of $\tilde \eps \in (0,\reach(\Theta))$ such that 
		\begin{equation}\label{est01}
	\sup_{x\in \Theta^{\tilde \eps}}	\|\pr_{\Theta}^{(\ell)}(x)\|_\ell < \infty
	\end{equation}
	for every $\ell\in\{1,2,3\}$. Moreover, by (A)(i) and (A)(iii) we have
		\begin{equation}\label{est02}
	\sup_{x\in \Theta}	\|f^{(\ell)}(x)\|_\ell < \infty
	\end{equation}
	for every $f\in \{ \normal, \alpha\}$ and every $\ell\in\{1,2,3\}$.
	Using~\eqref{est01} and~\eqref{est02} we obtain~\eqref{est00} for every $f\in \{ \normal\circ \pr_{\Theta}, \alpha\circ \pr_{\Theta}\}$ and every $\ell\in\{1,2,3\}$ by the
chain rule for derivatives, which completes the proof of the lemma.
\end{proof} 

We turn to the construction of the transformation $G$.

Choose $\tilde \eps$ according to Lemma~\ref{propconc} and put
\[
\gamma = \min(\eps^*,\tilde \eps).
\]
For all $ \eps \in(0,\gamma)$ we define 
\[ 
G_{\eps} = (G_{\eps,1},\dots, G_{\eps,d})^\top \colon \R^{d} \to \R^{d}, \,\,x \mapsto 
\begin{cases}	x + \Phi_{\eps}(x)\alpha(\pr_{\Theta}(x)), &\text{if }x\in \Theta^\eps,\\ x , &\text{if }x\in \R^d\setminus \Theta^\eps,\end{cases}
\] 
where
\[ 
\Phi_{\eps} \colon \Theta^{\gamma} \to \R, \,\,x \mapsto \normal(\pr_{\Theta}(x))^{\top}(x-\pr_{\Theta}(x))\|x-\pr_{\Theta}(x)\|\phi\left(\frac{\|x-\pr_{\Theta}(x)\|}{\eps}\right)
\]
and
\begin{equation}\label{fctphi}
\phi \colon \R \to \R, \,\,x \mapsto 
\begin{cases}
	(1-x^2)^{4}, &\text{ if } |x|\leq 1, \\
	0, &\text{ otherwise}.
\end{cases} 
\end{equation}

We will show below that there exists $\delta\in (0,\gamma)$ 
such that for all $\eps\in (0,\delta)$, the function $G=G_\eps$ satisfies the conditions (i) to (v)  in Proposition~\ref{G}.

For this purpose, we first study the functions $\phi$ and $\Phi_\eps$. 

\begin{Lem} \label{diff2}
	The function $\phi$	is a $ C^{3}$-function. 	For every $ \eps \in(0,\gamma)$, the function 	
$\Phi_{\eps}$ has the following properties. 
	\begin{itemize}
			\item[(i)] For every $s\in\{+,-\}$ and $x\in  Q_{\eps,s} $ we have
		\[ 
	 \Phi_{\eps}(x)= s \norm{x-\pr_{\Theta}(x)}^{2} \phi\Bigl(\frac{\|x-\pr_{\Theta}(x)\|}{\eps}\Bigr).
		\]
				\item[(ii)] $\sup_{x \in \Theta^{\gamma}}|\Phi_{\eps}(x)| \le \eps^{2}$. 
				\item[(iii)] $ \Phi_{\eps}$ is a $C^{1}$-function with $\Phi_\eps'(x)=0$ for every $x\in\Theta$ and  there exists $K \in (0,\infty)$, which does not depend on $\eps$, such that 
			\[ 
			\sup_{x \in \Theta^{\gamma}} \|\Phi'_{\eps}(x)\| \leq K \eps. 
			\]
					\item[(iv)]  $\Phi_{\eps}$ is a $C^3$-function on the open set $\Theta^{\gamma}\setminus \Theta $ and  $\Phi_\eps$ as well as all partial derivatives of $\Phi_\eps$ up to order $3$ vanish on $\Theta^{\gamma}\setminus \Theta^\eps$. Moreover,
		\[
		\sup_{x \in \Theta^{\gamma}\setminus \Theta} \|\Phi_{\eps}^{(2)}(x)\|_2 + \sup_{x \in \Theta^{\gamma} \setminus \Theta} \|\Phi_{\eps}^{(3)}(x)\|_3 < \infty. 
		\]		
	
	\end{itemize} 
\end{Lem}

\begin{proof} The proof of the statement on the function $\phi$ is straightforward.
	
	Let $ \eps \in(0,\gamma)$. 
	We turn to the proof of the properties (i) to (iv) of the function $\Phi_\eps$. 
	
	Let $y\in \Theta$, $\lambda\in (0,\eps$), $s\in\{+,-\}$ and put $x= y+ s\lambda \normal(y)$. Since $\normal(y)$ is orthogonal to the tangent space of $\Theta$ at $y$, we have $ \pr_{\Theta}(x)= y $ by Lemma~\ref{fed0} in the appendix. Hence $\|x-\pr_{\Theta}(x)\| = \|s\lambda \normal(y)\| = \lambda$ and we conclude that
		\[
		\normal(\pr_{\Theta}(x))^{\top}(x-\pr_{\Theta}(x)) = \normal(y)^\top s\lambda \normal(y) = s\lambda = s\|x-\pr_{\Theta}(x)\|,
			\]
			which finishes the proof of the property (i).
	
	   For $x\in \Theta^\gamma \setminus \Theta^\eps$ we have $\|x-\pr_{\Theta}(x)\| \ge \eps$. Hence $\phi(\|x-\pr_{\Theta}(x)\| /\eps) =0$, which implies $\Phi_\eps(x) = 0$. Next, let $x\in  \Theta^\eps$. Then $\|x-\pr_{\Theta}(x)\| < \eps$ and therefore
	   \[
	   |\Phi_{\eps}(x)| \le \|x -\pr_{\Theta}(x)\|^{2}\|\normal(\pr_{\Theta}(x))\|\Bigl(1-\frac{\|x-\pr_{\Theta}(x)\|^2}{\eps^2}\Bigr)^{4}  < \eps^{2}, 
	   \]
	   which finishes the proof of the property (ii).
	   
	   We turn to the proof of the properties (iii) and (iv).
		 By  Lemma~\ref{propconc}, the functions $ \normal \circ \pr_{\Theta}$ and $\pr_{\Theta}$ are $C^{3}$-functions on $\Theta^{\gamma}$. 
		Since $\|\cdot\|$ is a $C^{\infty}$-function on  $\R^{d} \setminus \{0\}$ 
		we obtain that $\|\cdot - \pr_{\Theta}(\cdot)\| $ is a $ C^{3}$-function on $\Theta^{\gamma}\setminus \Theta$. Using the fact that $\phi$ is a $C^3$-function on $\R$ we conclude that $\phi \circ (\|\cdot - \pr_{\Theta}(\cdot)\|/\eps) $ is a $ C^{3}$-function on  $\Theta^{\gamma}\setminus \Theta$ as well. Thus, $\Phi_\eps$ is a $ C^{3}$-function on $\Theta^{\gamma}\setminus \Theta$.
				Furthermore, for $x \in \Theta^{\eps}$ we have $\phi(\|x-\pr_{\Theta}(x)\|/\eps)=(1-\|x - \pr_{\Theta}(x)\|^2/\eps^2)^{4} $. Since $\|\cdot\|^2$ is a $C^{\infty}$-function on  $\R^{d}$, we conclude that $ \phi \circ  (\|\cdot - \pr_{\Theta}(\cdot)\|/\eps) $ is a $ C^{3}$-function on $\Theta^{\eps}$. Since $f\colon \R^{d} \to \R^{d}, \, x \mapsto x \, \norm{x}$ is a $C^{1}$-function, we obtain that $\Phi_\eps $ is a $ C^{1}$-function on   $\Theta^{\eps}$. Since $\Theta^{\eps} \cup( \Theta^{\gamma}\setminus \Theta)$ = $\Theta^\gamma$ we conclude that  $\Phi_\eps$ is a $C^1$-function on $\Theta^\gamma$.
				
	Clearly, $\phi(\|x-\pr_{\Theta}(x)\|/\eps) = 0$ for all $x\in \Theta^\gamma \setminus \Theta^\eps$, which implies in particular that $\Phi_\eps$ vanishes on the open set $\{x\in\R^d\mid \eps < d(x,\Theta) <\gamma\} \subset \Theta^\gamma \setminus \Theta^\eps\subset \Theta^\gamma\setminus \Theta$. As a consequence, all partial derivatives of $\Phi_\eps$ up to order $3$ vanish on  $\{x\in\R^d\mid \eps < d(x,\Theta) <\gamma\}$ as well. Since $\Phi_\eps$ is a $C^3$-function on $\Theta^{\gamma}\setminus \Theta$ we conclude that $\Phi_\eps$ and all partial derivatives of $\Phi_\eps$ up to order $3$ also vanish on    $\Theta^{\gamma}\setminus \Theta^\eps=\{x\in\R^d\mid \eps\le d(x,\Theta) <\gamma\} $. 
	
	It remains to prove the estimates  in (iii) and (iv) and the fact that $\Phi_\eps'$ vanishes on $\Theta$. Let $s\in\{+,-\}$. By the property (i) we have
	\[
	\Phi_{\eps}(x)= s f_{\eps}(\norm{x-\pr_{\Theta}(x)}^{2}), \quad x\in Q_{\eps,s},
	\]
    with 	$f_{\eps} \colon \R \to \R,\,  x \mapsto x(1-x/\eps^{2})^{4} $ . Clearly, $ f_{\eps}$ is a $C^{\infty}$-function and straightforward calculations yield  that for all $x \in \R$,
		\[ 
		f_{\eps}'(x)= 1- \frac{8}{\eps^{2}} x + \frac{18}{\eps^{4}}x^{2} - \frac{16}{\eps^{6}} x^{3} + \frac{5}{\eps^{8}}x^{4}. 
		\]
		For $x \in(-\eps^{2},\eps^{2})$ we thus have $ \abs{f_{\eps}'(x)} \leq 1 + 8 + 18 + 16 + 5 = 48 $.
		Since $Q_{\eps,s}\subset \Theta^\eps$ we obtain by the chain rule and Lemma~\ref{projdist}(iii) in the appendix that for every $ x \in Q_{\eps,s}$,
		\begin{equation}\label{phiref} 
			\begin{aligned}
				\Phi_{\eps}'(x)  &= s f_{\eps}'(\|x-\pr_{\Theta}(x)\|^{2})2(x-\pr_{\Theta}(x))^{\top}(I_{d} - \pr_{\Theta}'(x))\\
				& = s f_{\eps}'(\|x-\pr_{\Theta}(x)\|^{2})2(x-\pr_{\Theta}(x))^{\top},
				\end{aligned}
		\end{equation} 
	and therefore
	\[
\|	\Phi_{\eps}'(x)\| \le 96\eps
	\]
	for all  $ x \in Q_{\eps,s}$. Hence, by \eqref{u01},
		\begin{equation}\label{zzt1}
		\sup_{x \in \Theta^\eps\setminus\Theta}\|\Phi_{\eps}'(x)\| \le 96\, \eps. 
		\end{equation}
	
		Let $ x \in \Theta$. Clearly, $\lim_{n\to\infty} x+n^{-1}\normal(x) = x$ and $x+n^{-1}\normal(x)\in Q_{\eps,+} $ for $n> 1/\eps$. Moreover, $\pr_{\Theta}(x+n^{-1}\normal(x)) = x$ 
		for $n>1/\reach(\Theta)$.
		Since $ \Phi_\eps'$ is  continuous we thus obtain by~\eqref{phiref} that
		\begin{align*}
		\Phi_{\eps}'(x)  = \lim_{n \to \infty}\Phi_{\eps}'(x+n^{-1}\normal(x)) = \lim_{n \to \infty} f_{\eps}'(\|n^{-1}\normal(x)\|^{2}) 2n^{-1}\normal(x)^{\top}  = \lim_{n \to \infty}  f_\eps'(n^{-2})2n^{-1} \normal(x)^{\top} = 0,
		\end{align*}
	which jointly with~\eqref{zzt1} and the fact  that $\Phi_{\eps}'$ vanishes on $ \Theta^{\gamma} \setminus \Theta^{\eps}$ completes the proof of the property (iii).
			
        Finally, we prove the estimate in the property (iv). Recall that all partial derivatives of $\Phi_{\eps}$ up to order $3$ vanish on $ \Theta^{\gamma} \setminus \Theta^{\eps}$. Observing~\eqref{u01} it thus remains to show that for $s\in\{+,-\}$,
        \begin{equation}\label{u0}
        	\sup_{x \in Q_{\eps,s}} \|\Phi_{\eps}^{(2)}(x)\|_2 + \sup_{x \in Q_{\eps,s} } \|\Phi_{\eps}^{(3)}(x)\|_3 < \infty. 
        \end{equation}
       
	    Fix $s\in\{-,+\}$. By Lemma~\ref{propconc} we have
		\begin{equation}\label{u2}
	\max_{\ell \in \{1,2,3\}}	\sup_{x \in \Theta^{\gamma}}  \|\pr_{\Theta}^{(\ell)}(x)\|_\ell < \infty, 
		\end{equation}
	which implies
		\begin{equation}\label{u4}
	 \max_{\ell \in \{0,1,2\}} 	\sup_{x \in  \Theta^{\gamma}}\|(\cdot-\pr_{\Theta}(\cdot))^{(\ell)}(x)\|_\ell < \infty. 
	\end{equation}
    	Clearly, 
    		\[ 
    \max_{\ell \in \{0,1,2\}}	
    \sup_{x \in B_{\gamma}(0)}  \bigl\|(\|\cdot\|^{2})^{(\ell)}(x)\bigr\|_\ell < \infty, 
    	\] 
    	which jointly with~\eqref{u4} implies
    	    		\begin{equation}\label{u5}
    	\max_{\ell \in \{0,1,2\}} 		
    	\sup_{x \in  \Theta^{\gamma}}\bigl \|\bigl(\|\cdot-\pr_{\Theta}(\cdot)\|^2\bigr)^{(\ell)}(x)\bigr\| _\ell< \infty.
    \end{equation}
	   Obviously we have
			\[ 
	\max_{\ell \in \{1,2,3\}} 	\sup_{x \in [0,\gamma^2]}  |f_{\eps}^{(\ell)}(x)| < \infty, 
		\]
		which jointly with~\eqref{u5} yields
		\begin{equation}\label{u6}
	\max_{\ell \in \{0,1,2\}} 	\sup_{x \in \Theta^{\gamma}} \|(f_{\eps}' \circ \|\cdot-\pr_{\Theta}(\cdot)\|^{2})^{(\ell)}(x)\|_\ell < \infty. 
\end{equation}
	Employing \eqref{phiref} as well as 
	\eqref{u4} and~\eqref{u6} yields~\eqref{u0} and hereby completes the proof of the lemma.
\end{proof}

Now, we turn to the analysis of the transformation $G_\eps$.

\begin{Lem}\label{Gdiff}
	For every $\eps \in(0,\gamma)$, the function  $ G_{\eps}$ has the following properties.
	\begin{itemize}
		\item [(i)] $ G_{\eps}$ is a $C^{1}$-function with bounded derivative $G_\eps'$ that satisfies $G_\eps'(x)= I_d$ for every $x\in\Theta$ and every $x\in\R^d\setminus \Theta^\eps$.
		\item[(ii)] $G_{\eps}$ is a $C^{3}$-function on  $\R^{d} \setminus \Theta $ with 
		\[
		 \sup_{x \in \R^{d}\setminus \Theta}  \|G_{\eps}^{(2)}(x)\|_2 +  \sup_{x \in \R^{d}\setminus \Theta}  \|G_{\eps}^{(3)}(x)\|_3 < \infty.
		 \]
	\end{itemize}
\end{Lem}

\begin{proof}
Let $\eps \in(0,\gamma)$. By Lemma~\ref{propconc} we know that $\alpha \circ \pr_{\Theta}$ is a  $C^{3}$-function on $\Theta^\gamma$. Using Lemma~\ref{diff2}(iii) and (iv) we conclude that $G_\eps$ is a $C^1$-function on  $\Theta^\gamma$ and a $C^3$-function on  $\Theta^\gamma\setminus\Theta$, respectively. Since $G_\eps(x) = x$ for all $x\in \R^d\setminus  \Theta^\eps$, we obtain that $G_\eps$ is a $C^\infty$-function on the open set $ \R^d\setminus \cl(\Theta^\eps)$. Note that $ \cl( \Theta^\eps) =\{x\in\R^d\mid d(x,\Theta) \le \eps\} \subset \Theta^\gamma$. Hence $G_\eps$ is a $C^1$-function on $\R^d =\Theta^\gamma \cup (\R^d\setminus  \cl(\Theta^\eps))$ and a $C^3$-function on  $\R^d\setminus  \Theta = (\Theta^\gamma\setminus\Theta) \cup (\R^d\setminus \cl( \Theta^\eps))$, respectively.

By Lemma~\ref{propconc},
\[ 
\max_{\ell\in \{0,\dots,3\}} \sup_{x \in \Theta^{\gamma}} \|(\alpha\circ \pr_{\Theta})^{(\ell)}(x)\|_\ell< \infty. 
\]
Combining the latter fact with Lemma~\ref{diff2}(ii), (iii) and (iv) we obtain by the product rule for derivatives,
\[
\sup_{x\in \Theta^\gamma} \|G_{\eps}^{(1)}(x)\|_1 + \sup_{x \in \Theta^{\gamma}\setminus \Theta} \max \{\|G_{\eps}^{(2)}(x)\|_2,\|G_{\eps}^{(3)}(x)\|_3 \} < \infty. 
\]
Since $ G_{\eps}(x)= x $  for all $ x \in \R^{d}\setminus \Theta^{\eps}$ we furthermore have
 \[
 \max_{\ell\in\{1,2,3\}}  \sup_{x \in \R^{d}\setminus \cl(\Theta^{\eps})} \|G_{\eps}^{(\ell)}(x)\|_\ell < \infty. 
  \]
 
 It remains to prove that $G_\eps'(x)=I_d$ for every $x\in\Theta$ and every $x\in\R^d\setminus \Theta^\eps$. 
  Since $G_\eps(x) = x$ for every $x\in \R^d\setminus \Theta^\eps$ and $G'_\eps$ is continuous we have
  \begin{equation} \label{Gdiffref} 
  	G_{\eps}'(x)= \begin{cases}
  		I_{d} +(\alpha\circ \pr_{\Theta})(x) \Phi_{\eps}'(x) +  \Phi_{\eps}(x)(\alpha \circ \pr_{\Theta})'(x), & \text{ if } x \in \Theta^{\eps},\\
  		I_{d}, & \text{ if } x \in \R^{d} \setminus \Theta^{\eps}
  	\end{cases} 
  \end{equation}
by the product rule for derivatives. Let $x\in\Theta$. Then $\Phi_\eps(x)=0$ by the definition of $\Phi_\eps$ and we have $\Phi'_\eps(x) =0$   by Lemma~\ref{diff2}(iii). Thus $G_\eps'(x)=I_d$,  which finishes the proof of the lemma. 
\end{proof}

Next, we show that $\eps$ can be chosen in such a way that $G_\eps$ is a diffeomorphism.

\begin{Lem} \label{diffeo}
	There exists $ \delta \in(0,\gamma)$ such that for all $ \eps \in(0,\delta)$ the function $ G_{\eps}$ 
is a diffeomorphism with $\sup_{x \in \R^{d}} \|(G_{\eps}^{-1})'(x)\| < \infty $.	
\end{Lem}

\begin{proof}
	We first recall that by  Lemma~\ref{propconc} and 
	Lemma~\ref{diff2}(ii), (iii)
	 there exist $c_1,c_2\in (0,\infty)$ such that
\begin{equation}\label{r2}
	\sup_{x\in \Theta^\gamma}\max\bigl( \|(\alpha\circ \pr_{\Theta})(x)\|, \|(\alpha\circ \pr_{\Theta})'(x)\|\bigr) \le c_1
\end{equation}	
and for all $\eps \in (0,\gamma)$,
\begin{equation}\label{r3}
	\sup_{x \in \Theta^{\eps}} \max\bigl(|\Phi_{\eps}(x)|, \|\Phi_{\eps}'(x)\|\bigr) \leq c_2 \eps, 
\end{equation}	
respectively.

	Let $ \eps \in(0,\gamma)$. By Lemma~\ref{Gdiff}(i) we know that $G_\eps$ is a $C^1$-function. Thus, by Hadamard's global inverse function theorem, $G_\eps$ is a diffeomorphism if and only if
	\begin{itemize}
		\item [(a)] $G_\eps'(x)$ is invertible for every $x\in\R^d$, and
		\item[(b)] $\lim_{\|x\| \to \infty} \|G_{\eps}(x)\| = \infty $.
	\end{itemize}

Since $G_\eps(x) = x$ for $x\in \R^d\setminus \Theta^\eps$ we have for all $x\in\R^d$,
\begin{equation}\label{r1}
\|G_\eps(x)\| \ge \|x\| - \sup_{y\in \Theta^\eps} |\Phi_\eps(y)|\cdot\|\alpha(\pr_{\Theta}(y))\|.
\end{equation}
Combining~\eqref{r1} with~\eqref{r2} and~\eqref{r3} yields (b).

Put
\[
\delta = \min \bigl((2c_1c_2 + 1)^{-1},\gamma\bigr).
\]
We show that (a) is satisfied for every $\eps \in (0,\delta)$.

 Let $x\in \Theta^\eps$, recall~\eqref{Gdiffref}  and put 
\[
	\Gamma_{x,\eps} = G_\eps'(x) - I_d = (\alpha\circ \pr_{\Theta})(x) \Phi_{\eps}'(x)+\Phi_{\eps}(x)( \alpha \circ \pr_{\Theta})'(x). 
\]
Let $|	\Gamma_{x,\eps}|_2$ denote the spectral norm of $\Gamma_{x,\eps}$. By ~\eqref{r2} and~\eqref{r3} we obtain that
\begin{equation}\label{rrr1}
	\begin{aligned}
	|	\Gamma_{x,\eps}|_2 & \le	\|\Gamma_{x,\eps}\| \le  \|\Phi_{\eps}'(x)\|\|(\alpha\circ \pr_{\Theta})(x)\|+|\Phi_{\eps}(x)|\|( \alpha \circ \pr_{\Theta})'(x)\| \\
	& \le 2c_1c_2\eps <2c_1c_2\delta <1. 
	\end{aligned}
	\end{equation}
By well-known facts on Neumann series we conclude from~\eqref{rrr1} that 
$G_{\eps}'(x)= I_{d} + \Gamma_{x,\eps} $ is invertible and
\begin{equation}\label{rr2}
 |	(I_{d} + \Gamma_{x,\eps})^{-1}|_2  \le (1-|\Gamma_{x,\eps}|_2)^{-1}\le  (1-\|\Gamma_{x,\eps}\|)^{-1}.
	\end{equation}
Since 	$G_{\eps}'(x)= 	I_{d}$ for all $x\in   \R^d\setminus \Theta^\eps$ we thus obtain that (a) is satisfied as well.

Finally, we prove that $\sup_{x \in \R^{d}} \|(G_{\eps}^{-1})'(x)\| <\infty$. For all $x\in\R^d$ we have $(G_\eps^{-1})'(x)= (G_\eps'(G^{-1}_\eps(x)))^{-1}$. In the case $G^{-1}_\eps(x)\in \R^d\setminus\Theta^\eps$ we thus obtain by~\eqref{Gdiffref} that $ \|(G_\eps^{-1})'(x)\| = \|I_d\|=d^{1/2}$. In the case $G^{-1}_\eps(x)\in \Theta^\eps$ we get  by~\eqref{rrr1} and~\eqref{rr2} that 
\[
\|(G_\eps^{-1})'(x)\| \le  c|(G_\eps^{-1})'(x)|_2\le c(1-\|\Gamma_{G_\eps^{-1}(x),\eps}\|)^{-1} \le 
c(1-2c_1c_2\delta)^{-1},
\]
where $c\in(0,\infty)$ depends neither on  $x$ nor on $\eps$.
This finishes the proof of the lemma.
\end{proof}

In the sequel we fix 
\[
\delta\in (0,\gamma)
\]
 according to Lemma~\ref{diffeo}.

\begin{Lem} \label{Lip}
	For every  $\eps \in(0,\delta)$, the diffeomorphism $G_\eps$ has the following properties. 
	\begin{itemize}
		\item[(i)]  The functions $G_{\eps}$ and $G_{\eps}'$ are Lipschitz continuous.
	\item[(ii)] For every $i\in\{1,\dots,d\}$, the function $G_{\eps,i}''\colon \R^{d} \setminus \Theta \to \R^{d\times d}$ is  intrinsic Lipschitz continuous.
			\item[(iii)] The functions $G_{\eps}^{-1}= (G_{\eps,1}^{-1},\dots, G_{\eps,d}^{-1})^\top$ and $(G_\eps^{-1})'$ are Lipschitz continuous.
		\item[(iv)] $G_{\eps}^{-1} $ is a $ C^{2}$-function on $\R^{d}\setminus \Theta$ and for every $i\in\{1,\dots,d\}$, the function $(G_{\eps,i}^{-1})''\colon \R^{d} \setminus \Theta \to \R^{d\times d}$ is bounded and intrinsic Lipschitz continuous.
	\end{itemize}
\end{Lem}

\begin{proof}
	The Lipschitz continuity of $G_{\eps}$  is a consequence of the boundedness of the derivative $G_{\eps}'$, see Lemma~\ref{Gdiff}(i). 
	
	Next, let $i\in\{1,\dots,d\}$. By Lemma~\ref{Gdiff} we know that $G_{\eps,i}$ has bounded partial derivatives up to order $3$ on $\R^d\setminus\Theta$. We may thus apply Lemma~\ref{diffintr0} in the appendix to obtain that all partial derivatives of $(G_{\eps,i})_{|\R^d\setminus\Theta}$ of order $1$ and $2$ are  intrinsic Lipschitz continuous. This yields part (ii) of the lemma and the fact that $G_\eps'$ is intrinsic Lipschitz continuous on $\R^d\setminus\Theta$. 
	Since $G_{\eps}'$ is continuous and $\Theta$ is a $C^{4}$-hypersurface of positive reach, the Lipschitz continuity of $G_{\eps}'$ now follows from Lemmas~\ref{condfullf0} and~\ref{piecewtolip0} in the appendix. This completes the proof of part (i)  of the lemma.
	
	We turn to the proof of part (iii) and part (iv) of the lemma. Clearly, the Lipschitz continuity of  $G_{\eps}^{-1}$ is a consequence of the boundedness of the derivative  $(G_{\eps}^{-1})' $, see Lemma~\ref{diffeo}.

	We next prove the desired regularity of  $(G_{\eps}^{-1})'$. 
For all $x\in \R^{d}$ we have
	\[ 
	(G_{\eps}^{-1})'(x)=(G_{\eps}'(G_{\eps}^{-1}(x)))^{-1},
	\] 
	and hence
	\begin{equation} \label{ginvdiffrep}
		\text{vec}(( G_{\eps}^{-1})'(x))= \text{vec}((G_{\eps}'(G_{\eps}^{-1}(x)))^{-1})=f(\text{vec}\circ G_{\eps}'(G_{\eps}^{-1}(x))),
	\end{equation}
where 
\[
 f \colon O \to \R^{d^{2}}, \quad x \mapsto \text{vec}(\text{mat}(x)^{-1})
\]
and 
\[
 O = \{ x \in \R^{d^{2}} \mid \text{det}(\text{mat}(x))\neq 0 \}.
\]

	By Lemmas \ref{Gdiff} and \ref{diffeo}  the functions $\text{vec}\circ G_{\eps}'$ and $G_{\eps}^{-1}$ are continuously differentiable on $\R^{d} \setminus \Theta$ and $\R^{d}$, respectively. Moreover, $f$ is continuously differentiable on the open set $O$ and  for all $x \in O$ and all $y \in \R^{d^{2}}$,
	\[
	f'(x)y = -\text{vec}(\text{mat}(x)^{-1} \; \text{mat}(y)\; \text{mat}(x)^{-1}),
	\]	
	see, e.g. ~\cite[Chapter 8, Theorem 4.3]{MatCalc}.  Thus, 
	by \eqref{ginvdiffrep}, 
	the function $\text{vec} \circ (G_{\eps}^{-1})'$ 
	 is continuously differentiable on $(G_{\eps}^{-1})^{-1}(\R^{d} \setminus \Theta)=  \R^{d} \setminus \Theta$ and   for all $ x \in \R^{d}\setminus \Theta$ and all $j\in \{1, \ldots, d\}$,
\begin{align} \label{Gdiffrep}
		\begin{aligned}
			&((\text{vec}\circ ( G_{\eps}^{-1})')'(x))_j\\
&\quad= f'(\text{vec}(G_{\eps}'(G_{\eps}^{-1}(x))))\cdot (\text{vec}\circ G_{\eps}')'(G_{\eps}^{-1}(x))\cdot((G_{\eps}^{-1})'(x))_j\\
			&\quad=  -\text{vec}\bigl((G_{\eps}'(G_{\eps}^{-1}(x)))^{-1} \cdot \text{mat}\big((\text{vec}\circ G_{\eps}')'(G_{\eps}^{-1}(x))\cdot((G_{\eps}^{-1})'(x))_j\big)\cdot (G_{\eps}'(G_{\eps}^{-1}(x)))^{-1}\bigr)\\
			&\quad=  -\text{vec}\bigl((G_{\eps}^{-1})'(x)\cdot \text{mat}\big((\text{vec}\circ G_{\eps}')'(G_{\eps}^{-1}(x))\cdot((G_{\eps}^{-1})'(x))_j\big)\cdot (G_{\eps}^{-1})'(x)\bigr).
		\end{aligned}
	\end{align}
	 Using Lemmas \ref{Gdiff} and \ref{diffeo} we conclude from \eqref{Gdiffrep} that for all $ i \in \{1,\dots,d \}$,
	\begin{equation}\label{ll1}
	 \sup_{x \in \R^{d} \setminus \Theta} \norm{(G_{\eps, i}^{-1})''(x)} < \infty.
	 \end{equation}

	  Let $ i \in \{1,\dots,d \}$. Lemma~\ref{diffintr0} in the appendix and \eqref{ll1} yield that  $(G_{\eps, i}^{-1})' $ is intrinsic Lipschitz continuous on $\R^{d}\setminus \Theta$. Since $(G_{\eps, i}^{-1})' $ is continuous, we obtain from Lemmas~\ref{condfullf0} and ~\ref{piecewtolip0} in the appendix that $(G_{\eps,i}^{-1})' $ is Lipschitz continuous. 
Thus, $(G_{\eps}^{-1})' $ is Lipschitz continuous.
	
Let $ j \in \{1,\dots,d \}$.  It follows from ~\eqref{Gdiffrep} 
and Lemma \ref{Gdiff}(ii)
that $( (G_{\eps, i}^{-1})'')_j$ is continuously differentiable on $\R^{d} \setminus \Theta$. Moreover, applying the product rule to ~\eqref{Gdiffrep} and using Lemmas \ref{Gdiff} and \ref{diffeo} as well as \eqref{ll1} it is straightforward to show that 
\[
 \sup_{x \in \R^{d} \setminus \Theta} \norm{( (G_{\eps, i}^{-1})'')_j'(x)} < \infty.
 \]
Lemma~\ref{diffintr0} in the appendix  implies that $((G_{\eps, i}^{-1})'')_j $ is intrinsic Lipschitz continuous. Hence, $(G_{\eps,i}^{-1})'' $ is intrinsic Lipschitz continuous. This completes the proof of the lemma.

\end{proof}

\begin{Lem} \label{Lip1} For every $\eps\in (0,\delta)$, 
the mapping
	\[
	\nu_{\eps}\colon \R^{d} \setminus \Theta \to \R^{d}, \,\, x \mapsto \Bigl(G_{\eps}'\mu + \frac{1}{2} \bigl(\tr( G_{\eps,i}''\sigma\sigma^{\top})\bigr)_{1\le i \le d}\Bigr)(x)
	\]
	is intrinsic Lipschitz continuous.
\end{Lem}
\begin{proof}
	First we prove that the function $G_{\eps}'\mu$ is intrinsic Lipschitz continuous on $\R^d\setminus \Theta$. By Assumption (A)(v), the function $\mu$ is intrinsic Lipschitz continuous on $\R^d\setminus\Theta$. Since $\eps < \eps^*$ we furthermore get by Assumption (A)(iv) that $\mu$ is bounded on $\Theta^\eps$. Lemma~\ref{Gdiff}(i) and Lemma~\ref{Lip}(i) obviously imply that the function  $G_{\eps}'$ is intrinsic Lipschitz continuous on $\R^d\setminus\Theta$ and bounded on $\Theta^\eps$. Moreover, we have $G_\eps'(x)= I_{d} $ for all $ x \in \R^{d} \setminus \Theta^{\eps}$, see 
	Lemma~\ref{Gdiff}(i), 
	which implies that $G_\eps'$ is constant on $\R^{d}\setminus \Theta^{\eps} $. Applying Lemma~\ref{productnew0} in the appendix with $A=C=\R^d\setminus\Theta$, $B=\Theta^\eps\setminus\Theta$, $f=G'_\eps$ and $g=\mu$ we conclude that $G_{\eps}'\mu$ is intrinsic Lipschitz continuous on $\R^d\setminus\Theta$.

It remains to prove that for every  $i \in \{1,\dots,d \}$, the function 
    $\tr\bigl(G_{\eps,i}''(\sigma\sigma^{\top})_{|\R^{d} \setminus \Theta} \bigr)$ 
   is intrinsic Lipschitz continuous. 
     By Assumptions (A)(iv) and (B) we 
   obtain that the mappings $\sigma,\sigma^\top \colon\R^d\to \R^{d\times d}$ are  bounded on $\Theta^\eps$  and intrinsic Lipschitz continuous on $\R^d\setminus\Theta$. By Lemma~\ref{Gdiff}(ii) and Lemma~\ref{Lip}(ii) we have that the mapping  $G_{\eps,i}'' \colon\R^{d} \setminus \Theta\to \R^{d \times d} $ is 
     intrinsic Lipschitz continuous and bounded. Using~\eqref{Gdiffref} we furthermore get that $G_{\eps,i}''(x)= 0 $  for all $ x \in \R^{d} \setminus \Theta^{\epsilon}$. Applying Lemma~\ref{productnew0} in the appendix first with $A=C=\R^d\setminus\Theta$, $B=\Theta^\eps\setminus\Theta$, $f=G''_{\eps,i}$ and $g=\sigma_{|\R^d\setminus\Theta}$ and then with $A=C=\R^d\setminus\Theta$, $B=\Theta^\eps\setminus\Theta$, $f=G''_{\eps,i}\sigma_{|\R^d\setminus\Theta}$ and $g=\sigma^\top_{|\R^d\setminus\Theta}$ we  conclude that $G_{\eps,i}''(\sigma\sigma^{\top})_{| \R^{d} \setminus \Theta    }$ is intrinsic Lipschitz continuous. Finally, since $\tr\colon\R^d\to\R^{d\times d}$ is a Lipschitz continuous mapping, we obtain by Lemma~\ref{comp0} in the appendix that $\tr\circ G_{\eps,i}''(\sigma\sigma^{\top})_{|\R^{d}\setminus \Theta}$ is intrinsic Lipschitz continuous.	
   \end{proof}

\begin{Lem} \label{exten}
		For every $\eps\in (0,\delta)$, the mappings  $G_{\eps,i}''\colon \R^{d} \setminus \Theta\to \R^{d\times d} $, $i \in \{1,\dots,d \}$, can be extended to bounded mappings $R_{\eps,i}\colon \R^{d}  \to  \R^{d\times d}$, $i \in \{1,\dots,d \}$, respectively, such that the function 
	\[
	\bar\nu_{\eps} = G_{\eps}'\mu + \frac{1}{2} \bigl(\tr(R_{\eps,i}\,\sigma\sigma^{\top})\bigr)_{1\le i \le d}\colon \R^d\to \R^d
	\] 
	is a Lipschitz continuous extension of $\nu_{\eps}$ 
	to $\R^{d}$.
\end{Lem}
\begin{proof}
	Let $\eps\in (0,\delta)$ and recall that  $\Theta^\eps\setminus\Theta = Q_{\eps,+}\cup Q_{\eps,-}$ where the sets 
	$Q_{\eps,s}=\{x+s\lambda\normal(x)\colon x\in \Theta,\lambda\in (0,\eps)\}$, $s\in \{+,-\}$, are open and disjoint, see~\eqref{u00} and~\eqref{u01}.
	For $k,m\in\N$, a function $g\colon \R^d\setminus\Theta\to\R^{k\times m}$ and $x\in\Theta$ we put
	\[
	g(x+)=\lim_{h\downarrow 0} g(x+h\normal(x)),\,\, 	g(x-)=\lim_{h\downarrow 0} g(x-h\normal(x))
	\]
	if these limits exist in $\R^{k\times m}$.
	
	Let $i\in\{1,\dots,d\}$. Below we show that there exists a  mapping $H\colon \Theta^\eps\to\R^{d\times d}$ such that for all $x\in\Theta$ and $s\in\{+,-\}$ we have
	\begin{equation}\label{u000}
	G_{\eps,i}''(xs) = 	H(x)+ s2\alpha_{i}(x)\normal(x)\normal(x)^{\top}.
	\end{equation}	
 By Lemma~\ref{exlim} we know that the limit $\mu_i(xs)$ exists in $\R$ for every $x\in\Theta$ and $s\in\{+,-\}$.
Now, we define  
\[
R_{\eps,i}\colon\R^d\to \R^{d\times d},\,\, x\mapsto 
\begin{cases}
G_{\eps,i}''(x), & \text{if } x\in\R^d\setminus\Theta,\\ G_{\eps,i}''(x+)+ \frac{2(\mu_i(x+)- \mu_i(x))}{\|\sigma(x)^\top\normal(x)\|^2}\normal(x)\normal(x)^{\top}, & \text{if } x\in\Theta.
\end{cases}
\]	
By Lemma~\ref{Gdiff}(ii) we have $\sup_{x\in\R^d\setminus\Theta} \|G_{\eps,i}''(x)\| <\infty$, which implies  $\sup_{x\in\Theta} \|G_{\eps,i}''(x+)\| <\infty$.  By condition (A)(iv) and the fact that $\eps <\eps^*$ we obtain that $\sup_{x\in\Theta^\eps} \|\mu_{i}(x)\| <\infty$, which  implies $\sup_{x\in\Theta} \|\mu_i(x+)\| <\infty$.  Furthermore, note that $\inf_{x\in\Theta} \|\sigma(x)^\top\normal(x)\| >0$, due to condition (A)(ii), and that $\|\normal(x)\normal(x)^\top\| =1$ for every $x\in\Theta$. Combining the latter facts yields the boundedness of $R_{\eps,i}$.

By Lemma~\ref{Lip1}, the mapping $\bar\nu_{\eps,i}$ is intrinsic Lipschitz continuous on $\R^d\setminus \Theta$. Next, we show that for all $x\in\Theta$ and $s\in\{+,-\}$,
\begin{equation}\label{uvw2}
	\lim_{h\downarrow 0} \bar\nu_{\eps,i}(x+sh\normal(x)) = \bar\nu_{\eps,i}(x)
\end{equation} 
By Lemma~\ref{a0} in the appendix we may then conclude that $\bar \nu_{\eps,i}$ is continuous, and since $\Theta$ is a 
 $C^4$-hypersurface 
of positive reach we now obtain the Lipschitz continuity of $\bar \nu_{\eps,i}$ by using Lemmas~\ref{condfullf0} and~\ref{piecewtolip0} in the appendix.

For the proof of~\eqref{uvw2} we first note that by Lemma~\ref{exlim}, the continuity of $G_{\eps,i}'$ and $\sigma$ and Lemma~\ref{Gdiff}(i)  we get
\begin{equation}\label{uvw3}
	\begin{aligned}
	\lim_{h\downarrow 0} \bar\nu_{\eps,i}(x+sh\normal(x)) &  =\lim_{h\downarrow 0} \nu_{\eps,i}(x+sh\normal(x)) 
	 = \mu_i(xs) + \frac{1}{2} \tr(G_{\eps,i}''(xs)\sigma(x)\sigma(x)^\top).
\end{aligned}
\end{equation}
Using Lemma~\ref{Gdiff}(i) and
~\eqref{u000} 
we obtain
\begin{equation}\label{uvw4}
	\begin{aligned}
&	\mu_i(x+) + \frac{1}{2} \tr(G_{\eps,i}''(x+)\sigma(x)\sigma(x)^\top)\\
 & \qquad \qquad = 	\mu_i(x) + \frac{1}{2} \tr(R_{\eps,i}(x)\sigma(x)\sigma(x)^\top)\\
	& \qquad \qquad \qquad+ \mu_i(x+) - \mu_i(x)- \frac{\mu_i(x+)- \mu_i(x)}{\|\sigma(x)^\top\normal(x)\|^2} \tr(\normal(x)\normal(x)^{\top}\sigma(x)\sigma(x)^\top)\\
	& \qquad \qquad =(G_{\eps}'(x)\mu(x))_i + \frac{1}{2} \tr(R_{\eps,i}(x)\sigma(x)\sigma(x)^\top) = \bar\nu_{\eps,i}(x)  
	\end{aligned}
\end{equation}
as well as
\begin{equation}\label{uvw5}
	\begin{aligned}
		&	\mu_i(x-) + \frac{1}{2} \tr\bigl(G_{\eps,i}''(x-)\sigma(x)\sigma(x)^\top)\\
		& \qquad \qquad = 	\mu_i(x) + \frac{1}{2} \tr\bigl(R_{\eps,i}(x)\sigma(x)\sigma(x)^\top)  + \mu_i(x-) - \mu_i(x) \\
		& \qquad \qquad \qquad+ \frac{1}{2}\tr\Bigl(\bigl(G_{\eps,i}''(x-) - G_{\eps,i}''(x+)- 2\frac{\mu_i(x+)- \mu_i(x)}{\|\sigma(x)^\top\normal(x)\|^2} \normal(x)\normal(x)^{\top}\bigr)\sigma(x)\sigma(x)^\top\Bigr)\\
		& \qquad \qquad =(G_{\eps}'(x)\mu(x))_i + \frac{1}{2} \tr\bigl(R_{\eps,i}(x)\sigma(x)\sigma(x)^\top) + \mu_i(x-) - \mu_i(x) \\
		& \qquad \qquad \qquad- \frac{1}{2}\tr\Bigl(\Bigl( 4\alpha_i(x) + 2\frac{\mu_i(x+)- \mu_i(x)}{\|\sigma(x)^\top\normal(x)\|^2}\Bigr) \normal(x)\normal(x)^{\top}\sigma(x)\sigma(x)^\top\Bigr)\\	
		&  \qquad \qquad= \bar\nu_{\eps,i}(x)  + \mu_i(x-) - \mu_i(x) - \frac{\mu_i(x-) -\mu_i(x)}{\|\sigma(x)^\top\normal(x)\|^2}\tr\Bigl( \normal(x)\normal(x)^{\top}\sigma(x)\sigma(x)^\top\Bigr)\\ & \qquad \qquad= \bar\nu_{\eps,i}(x).
	\end{aligned}
\end{equation}
Combining~\eqref{uvw3} with~\eqref{uvw4} and~\eqref{uvw5} yields~\eqref{uvw2}.

It remains to prove~\eqref{u000}. 
Let $z\in\R^d$.
Using Lemma~\ref{diff2}(iv) as well as~\eqref{phiref} and Lemma~\ref{projdist}(iii) in the appendix we obtain that for every $y\in Q_{\eps,s}$ and $s\in\{+,-\}$, 
\begin{align*}
		z^\top\Phi_{\eps}''(y) & 
			= (\Phi_\eps' z)'(y)
     		 = s 2\bigl(f_{\eps}'(\|\cdot-\pr_{\Theta}(\cdot)\|^{2})(\cdot-\pr_{\Theta}(\cdot))^{\top}z\bigr)' (y)\\
		& = s4f_{\eps}''(\| y-\pr_{\Theta}(y) \|^{2})(y-\pr_{\Theta}(y))^{\top}((y-\pr_{\Theta}(y))^{\top}z) \\
		& \qquad \qquad + s2f_{\eps}'(\|y-\pr_{\Theta}(y)\|^{2})z^\top(I_d - \pr_{\Theta}'(y)).
\end{align*}
Let $x\in \Theta$ and $h\in (0,\eps)$. Then $x+sh\normal(x)\in Q_{\eps,s}$ and we have $\pr_{\Theta}(x+sh\normal(x)) = x$. Hence
\begin{align*}
	z^\top\Phi_{\eps}''(x+sh\normal(x)) & = s4f_{\eps}''(\| sh\normal(x) \|^{2})(sh\normal(x))^{\top}((sh\normal(x))^{\top}z) \\
	& \qquad\qquad  + s2f_{\eps}'(\|sh\normal(x)\|^{2})z^\top(I_d - \pr_{\Theta}'(x+sh\normal(x)))\\
	& = s4h^2 f_{\eps}''( h^{2})\normal(x)^{\top}(\normal(x)^{\top}z) + s2f_{\eps}'(h^{2})z^\top(I_d - \pr_{\Theta}'(x+sh\normal(x))).
\end{align*}
By the continuity of $f_\eps'$, $f_\eps''$ and $\pr'_{\Theta}$ and by the fact that $\pr'_\Theta(x) = I_d- \normal(x)\normal(x)^\top$, see Lemma~\ref{pdiff01} in the appendix, we conclude that
\[
\lim_{h\downarrow 0} 	z^{T}\Phi_{\eps}''(x+sh\normal(x))  = s2f_\eps'(0) z^\top 
(I_d-\pr'_{\Theta}(x))
 = s2z^\top \normal(x)\normal(x)^\top,
 \]
 which yields 
 \begin{equation}\label{uvw1}
 \Phi_{\eps}''(xs) =  s2 \normal(x)\normal(x)^\top.
 \end{equation}

    Recall from  Lemma~\ref{propconc} 
    that  $\alpha \circ \pr_{\Theta}$ is a $C^{3}$-function 
    on $\Theta^\eps$. By \eqref{Gdiffref} we  have for all $y \in \Theta^{\epsilon}$,
	\[
		G_{\eps,i}'(y)= e_i^\top +(\alpha_i \circ \pr_{\Theta})(y)\Phi_{\eps}'(y)+ \Phi_{\eps}(y)  (\alpha_i \circ \pr_{\Theta})'(y).
	\]
By the product rule for derivatives we conclude that for all $y\in Q_{\eps,s}$ and $s\in\{+,-\}$, 
\begin{align*}
	G_{\eps,i}''(y) &= ((G_{\eps,i}')^\top)'(y)=
 	\bigr((\alpha_i \circ \pr_{\Theta})'(y)\bigl)^\top\Phi_{\eps}'(y)  +(\alpha_i \circ \pr_{\Theta})(y)\Phi_{\eps}''(y)   \\
	    & \qquad  + (\Phi_{\eps}'(y))^\top  (\alpha_i \circ \pr_{\Theta})'(y) +  \Phi_{\eps}(y) (\alpha_i \circ \pr_{\Theta})''(y),
	\end{align*}
which jointly with~\eqref{uvw1} implies that for all $x\in \Theta$ and $s\in\{+,-\}$,
\begin{align*}
	\lim_{h\downarrow 0} G_{\eps,i}''(x+sh\normal(x) ) &= \bigr((\alpha_i \circ \pr_{\Theta})'(x)\bigl)^\top\Phi_{\eps}'(x)  +(\alpha_i \circ \pr_{\Theta})(x)\Phi_{\eps}''(xs)   \\ 	& \qquad  + (\Phi_{\eps}'(x))^\top  (\alpha_i \circ \pr_{\Theta})'(x) +  \Phi_{\eps}(x) (\alpha_i \circ \pr_{\Theta})''(x) \\
	& = H(x) + s2\alpha_i(x) \normal(x)\normal(x)^\top,
\end{align*}
where $H(x) =  \bigr((\alpha_i \circ \pr_{\Theta})'(x)\bigl)^\top\Phi_{\eps}'(x) + (\Phi_{\eps}'(x))^\top  (\alpha_i \circ \pr_{\Theta})'(x) +  \Phi_{\eps}(x) (\alpha_i \circ \pr_{\Theta})''(x) $.

This completes the proof of the lemma.
\end{proof}

\begin{Lem}\label{transfinal}
Let $\eps\in (0,\delta)$ and  choose $\bar\nu_\eps\colon \R^d\to\R^d$ according to Lemma~\ref{exten}.
\begin{itemize}
	\item[(i)] The mapping $\mu_\eps\colon \R^d\to\R^d,\,\, x\mapsto \bar \nu_\eps\circ G_\eps^{-1}$ is Lipschitz continuous.
	\item[(ii)] The mapping $\sigma_\eps\colon\R^d\to\R^{d\times d},\,\, x\mapsto (G'_\eps\sigma)\circ G_\eps^{-1}$ is Lipschitz continuous with $\sigma_\eps(x) = \sigma(x)$ for every $x\in\Theta$ and satisfies $	\sigma =((G_\eps^{-1})'\,\sigma_\eps)\circ G_\eps$.
\end{itemize} 
	\end{Lem}

\begin{proof}
Part (i) is an immediate consequence of Lemma~\ref{exten} and Lemma~\ref{Lip}(iii). 

For the proof of part (ii) we first note that, by Lemmas~\ref{Lip}(i) and~\ref{Gdiff}(i), the mapping $G_\eps'$ is bounded and Lipschitz continuous on $\R^d$ as well as constant on $\R^d\setminus \Theta^\eps$. Moreover, by condition (A)(iv) and condition (B), the mapping $\sigma$ is Lipschitz continuous on $\R^d$ as well as bounded on $\Theta^{\eps^*}\supset \Theta^\eps$. We may thus apply Lemma~\ref{productnew0} in the appendix with $A=C=\R^d$, $B=\Theta^\eps$, $f=G_\eps'$ and $g=\sigma$ to obtain that the mapping $G'_\eps \sigma\colon \R^d\to\R^d$ is intrinsic Lipschitz continuous. Since $\R^d$ is convex we have Lipschitz continuity of $G'_\eps \sigma$, see Lemma~\ref{lipimpl0} (ii) in the appendix. The latter fact and Lemma~\ref{Lip}(iii) imply the Lipschitz continuity of $\sigma_\eps$.

Next, let $x\in\Theta$. We have $G_\eps'(x) = I_d$, see Lemma~\ref{Gdiff}(i), and $G_\eps(x)=x$ by the definition of $G_\eps$. The latter fact implies $G_\eps^{-1}(x) = x$. Hence $\sigma_\eps(x) = \sigma(x)$.

Finally, we have $(G_\eps^{-1})' = (G_\eps')^{-1}\circ G_\eps^{-1}$, which yields $((G_\eps^{-1})'\,\sigma_\eps)\circ G_\eps = (G_\eps')^{-1} G'_\eps \sigma = \sigma$ and hereby finishes the proof of part (ii) of the lemma.
	\end{proof}

\begin{Lem}\label{transfinal2}
	Let $\eps\in (0,\delta)$, choose $\bar\nu_\eps\colon \R^d\to\R^d$ according to Lemma~\ref{exten} and define  $\mu_\eps = \bar \nu_\eps\circ G_\eps^{-1}$ and $\sigma_\eps$ as in Lemma~\ref{transfinal}. For all $i \in \{1,\dots,d \}$, the mapping $(G^{-1}_{\eps,i})''\colon \R^{d} \setminus \Theta\to \R^{d\times d} $ can be extended to a bounded mapping $S_{\eps,i}\colon \R^{d}  \to  \R^{d\times d}$ such that
	\begin{equation}\label{ln}
	\mu = \Bigl(\bigl(G_{\eps}^{-1}\bigr)'\mu_\eps + \frac{1}{2} \bigl(\tr(S_{\eps,i}\,\sigma_\eps\sigma_\eps^{\top})\bigr)_{1\le i \le d}\Bigr)\circ G_\eps.
	\end{equation}
\end{Lem}

\begin{proof} 
Let $i\in\{1, \ldots, d\}$. We define  
\[
S_{\eps,i}\colon\R^d\to \R^{d\times d},\,\, x\mapsto 
\begin{cases}
(G^{-1}_{\eps,i})''(x), & \text{if } x\in\R^d\setminus\Theta,\\ -\frac{\tr(R_{\eps,i}(x)\,\sigma(x)\sigma(x)^{\top})}{\|\normal(x)^\top\sigma(x)\|^2} \normal(x)\normal(x)^\top, & \text{if } x\in\Theta.
\end{cases}
\]	
By Lemma~\ref{Lip}(iv) we have $\sup_{x\in\R^d\setminus\Theta} \|(G^{-1}_{\eps,i})''(x)\| <\infty$.  
 Furthermore, for all $x\in\Theta$ we have $\inf_{x\in\Theta} \|\normal(x)^\top\sigma(x)\| >0$ due to condition (A)(ii) and $\|\normal(x)\normal(x)^\top\| =1$.  Moreover,
\[
\sup_{x\in\Theta}|\tr(R_{\eps,i}(x)\,\sigma(x)\sigma(x)^{\top})|<\infty
\]
 due to boundedness of $R_{\eps,i}$, see Lemma~\ref{exten}, and condition (A)(iv).
Thus,  $S_{\eps,i}$ is bounded.

We next show  \eqref{ln}. First, let $x\in\Theta$. Since $G_\eps(x)=x$ by definition of $G_\eps$ and $G_\eps'(x)= I_d$ by Lemma \ref{Gdiff}    we obtain
\[
\bigl(G_{\eps}^{-1}\bigr)'(x)=(G_{\eps}'(G_{\eps}^{-1}(x)))^{-1}= I_d
\]
and
\[
\mu_{\eps}(x)=\mu(x) + \frac{1}{2} \bigl(\tr(R_{\eps,i}(x)\,\sigma(x)\sigma(x)^{\top})\bigr)_{1\le i \le d}.
\]
 Using the fact that $\sigma_\eps(x)=\sigma(x)$, see  Lemma \ref{transfinal}, we conclude that
\begin{align*}
&\Bigl(\bigl(G_{\eps}^{-1}\bigr)'\mu_\eps + \frac{1}{2} \bigl(\tr(S_{\eps,i}\,\sigma_\eps\sigma_\eps^{\top})\bigr)_{1\le i \le d}\Bigr)\circ G_\eps(x)\\
&\qquad\qquad=\mu_{\eps}(x) + \frac{1}{2}\bigl( \tr(S_{\eps,i}(x)\,\sigma(x)\sigma(x)^{\top})\bigr)_{1\le i \le d}\\
&\qquad\qquad=\mu(x)+ \frac{1}{2}\bigl(\tr(R_{\eps,i}(x)\,\sigma(x)\sigma(x)^{\top})+\tr(S_{\eps,i}(x)\,\sigma(x)\sigma(x)^{\top})\bigr)_{1\le i \le d}\\
&\qquad\qquad=\mu(x),
\end{align*}
where the last equality follows from
\[
\tr(\normal(x)\normal(x)^\top\sigma(x)\sigma(x)^\top)=\tr(\sigma(x)^\top\normal(x)\normal(x)^\top\sigma(x))
=\|\normal(x)^\top\sigma(x)\|^2.
\]

For all $x\in\R^d\setminus \Theta$ we  have  $R_{\eps,i}(x)=G_{\eps,i}''(x)$, $i\in\{1, \ldots, d\}$, and
\[
\bigl(G_{\eps}^{-1}\bigr)'\circ G_\eps(x)=\bigl(G_{\eps}'(x)\bigr)^{-1}.
\]
Using Lemma~\ref{exten} and Lemma \ref{transfinal}
we therefore obtain that for all $x\in\R^d\setminus \Theta$, 
\begin{align*}
&\Bigl(\bigl(G_{\eps}^{-1}\bigr)'\mu_\eps + \frac{1}{2} \bigl(\tr(S_{\eps,i}\,\sigma_\eps\sigma_\eps^{\top})\bigr)_{1\le i \le d}\Bigr)\circ G_\eps(x)\\
&\qquad =\bigl(G_{\eps}'(x)\bigr)^{-1}\bar\nu_{\eps}(x) + \frac{1}{2} \bigl(\tr(S_{\eps,i}\,\sigma_\eps\sigma_\eps^{\top})\bigr)_{1\le i \le d}\circ G_\eps(x)\\
&\qquad=\mu(x) + \frac{1}{2} \Bigl(((G_{\eps}^{-1})'\circ G_\eps)  \bigl(\tr\bigl(G_{\eps,i}''\,\sigma\sigma^{\top}\bigr)\bigr)_{1\le i \le d}+\bigl(\tr\bigl(((G^{-1}_{\eps,i})''\circ G_\eps)\,G_\eps'\sigma\sigma^{\top}(G_\eps')^{\top}\bigr)\bigr)_{1\le i \le d}\Bigr)(x).
\end{align*}
For convenience of writing we define $f, g\colon \R^d\setminus \Theta\to \R^d\setminus \Theta$ by
\[
f(x)=G^{-1}_{\eps}(x), \quad g(x)=G_{\eps}(x).
\]
 It thus remains to show that for all $k\in\{1, \ldots, d\}$,
\begin{equation}\label{l5}
(f_k'\circ g)\bigl(\tr\bigl(g_i''\,\sigma\sigma^{\top}\bigr)\bigr)_{1\le i \le d}+\tr\bigl((f_k''\circ g)\,g'\sigma\sigma^{\top}(g')^{\top}\bigr)=0.
\end{equation}
Let $k\in\{1, \ldots, d\}$. The chain rule and the fact that $f\circ g(x)=x$ for all $x\in \R^d\setminus \Theta$ yield
\begin{align*}
&(f_k'\circ g)\bigl(\tr\bigl(g_i''\,\sigma\sigma^{\top}\bigr)\bigr)_{1\le i \le d}+\tr\bigl((f_k''\circ g)\,g'\sigma\sigma^{\top}(g')^{\top}\bigr)\\
&\qquad=\sum_{i=1}^d  \Bigl(\frac{\partial f_k}{\partial x_i }\circ g\Bigr) \tr\bigl(g_i''\,\sigma\sigma^{\top}\bigr)+\sum_{i,\ell=1}^d (f_k''\circ g)_{i,\ell}\,(g'\sigma\sigma^{\top}(g')^{\top})_{\ell,i}\\
&\qquad =\sum_{h,j=1}^d \Bigl(\sum_{i=1}^d \Bigl(\frac{\partial f_k}{\partial x_i}\circ g\Bigr)\frac{\partial^2 g_i}{\partial x_h\partial x_j}+\sum_{i, \ell=1}^d \Bigl(\frac{\partial^2 f_k}{\partial x_i\partial x_\ell}\circ g\Bigr)\frac{\partial g_i}{\partial x_h}\frac{\partial g_\ell}{\partial x_j}\Bigr)(\sigma\sigma^{\top})_{j,h}\\
&\qquad= \sum_{h,j=1}^d\frac{\partial^2(f\circ g)_k}{\partial x_h \partial x_j} (\sigma\sigma^{\top})_{j,h}\\
&\qquad =0,
\end{align*}
which yields \eqref{l5} and finishes the proof of the lemma.
	
	\end{proof}

\subsubsection*{Proof of Propostion~\ref{G}}
Choose $\tilde \eps\in (0,\reach(\Theta))$ according to Lemma~\ref{propconc}, let $\gamma=\min(\tilde \eps, \eps^*)$, choose $\delta\in (0,\gamma)$ according to Lemma~\ref{diffeo}, let $\eps\in (0,\delta)$ and put $G=G_\eps$. Then part (i) of Proposition~\ref{G} is a consequence of Lemma~\ref{diffeo}. Part (ii) follows from Lemma~\ref{Gdiff}(i), Lemma~\ref{diffeo} and Lemma~\ref{Lip}(i),(iii). Part (iii) of the proposition follows from Lemma~\ref{Gdiff}(ii) and Lemma~\ref{Lip}(ii),(iv). Part (iv) of the proposition follows from Lemma~\ref{transfinal}(ii). Part (v) is a consequence of Lemma~\ref{exten}, Lemma~\ref{transfinal}(i) and Lemma~\ref{transfinal2}. \qed
\bigskip

\subsection{An It\^{o} formula}\label{ito}

In this section we provide in Theorem~\ref{Ito_new} an It\^o formula that can be applied with the transformation $G$ and its inverse $G^{-1}$ from Proposition~\ref{G} in Section~\ref{trans} and enables us to 
prove the existence and uniqueness of a strong solution of the SDE~\eqref{SDE} under the conditions (A) and (B), see Section~\ref{Exist}, and is also employed to obtain an $L_p$-estimate  for the distance in supremum norm of the Euler-Maruyama scheme  transformed by $G$, i.e. $G\circ \widehat X_n$, and the  Euler-Maruyama scheme $\widehat Y_n$ for the transformed solution $Y=G\circ X$, see~\eqref{prox1} in Section~\ref{ProofThm1}.

Recall that for any Lipschitz continuous function $\psi\colon\R^d\to\R$ there exists  
a Borel
 set $A\subset\R^d$ with $\lambda_d(\R^d\setminus A)=0$ such that $\psi$ is differentiable in every $x\in A$ and $\sup_{x\in A} \|\psi'(x)\| <\infty$.
 Moreover, any measurable extension of $\psi'\colon A\to\R^d$ to $\R^d$ is a weak 
derivative 
of $\psi$. See, e.g.~\cite[Section 5.8]{Evans2010} for these facts. In the sequel, we use $\psi'= (\partial \psi/\partial x_1,\dots, \partial \psi/\partial x_d )\colon \R^d \to 
\R^{1\times d}$ 
to denote any weak derivative of $\psi$. Clearly, $\psi'$ can always be chosen to be bounded on $\R^d$.

 \begin{Thm} \label{Ito_new}
	Let $\alpha = (\alpha_t)_{t\in [0,1]}$ be an $\R^d$-valued, measurable, 
adapted 
	stochastic
	process with  $\PP$-a.s.
	\begin{equation}\label{a1}
	\int_0^1 \|\alpha_t\|\, dt <\infty,
	\end{equation}
	  and let $\beta=(\beta_t)_{t\in [0,1]}$ be an $\R^{d\times d}$-valued, measurable, adapted 
	stochastic
	process with
	\begin{equation}\label{a2}
	\int_0^1 \EE \bigl[\|\beta_t\|^2\bigr]\, dt  <\infty.
	\end{equation}
	Let $y_0\in \R^d$
	and  let $Y=(Y_t)_{t\in [0,1]}$ be the  continuous semi-martingale given by 
	\[
	Y_t = y_0 + \int_0^t \alpha_s\, ds + \int_0^t \beta_s\, dW_s,\quad t\in [0,1].
	\]
	Furthermore, let $f\colon\R^d\to\R$ be a $C^1$-function with bounded, Lipschitz continuous derivative $f'\colon\R^d\to \R^{1\times d}$ and let $f''\colon\R^d\to\R^{d\times d}$ be a bounded weak derivative of $f'$. Let $M\subset \R^d$ be closed and
assume that 	$f$ is a
$C^2$-function on  $\R^d\setminus M$.
 Finally, {let $\gamma = (\gamma_{t})_{t \in [0,1]}$ be an $\R^{d\times d}$-valued, measurable stochastic process and} assume that there exist $\delta\in (0,\infty)$ and a $C^2$-function $g\colon\R^d\to \R$  with $M\subset \{g=0\}$, {$\|g'\|_{\infty} < \infty$} and $\PP$-a.s.
	\[
	\inf_{t\in [0,1]} ( \|g'(Y_t)\gamma_t\| - \delta)\one_{\{Y_t\in M\}} \ge 0.
	\]
	Then, {for every $r \in (0,\infty)$,} $\PP$-a.s., 
	\begin{equation*}
		\begin{aligned}
		&	\sup_{t \in [0,1]} \Big| f(Y_t) - f(y_{0})   - \int_{0}^{t} \bigl(f'(Y_s) \alpha_{s} 
			+ \frac{1}{2} \tr( f''(Y_s)\beta_s\beta_s^\top)\bigr) \, ds 
			 - \int_{0}^{t} f'(Y_s) \beta_s\; dW_{s}\Big|\\
			&\leq {  \| f'' \|_{\infty} \biggl(\frac{2\|g'\|_{\infty}^{2}}{\delta^{2}}\biggr)^{(r-2)/r} \biggl(\int_0^1  \|\beta_s\|^r \, ds \biggr)^{2/r}   } \, \biggl( \int_0^1 \bigl \|{ \beta_t-\gamma_t}\bigr\|^2\, dt\biggr)^{{(r-2)/r}},
		\end{aligned}
	\end{equation*}
	{where $0 \cdot \infty := 0$.}
\end{Thm}

{
\begin{Rem}
	We briefly comment on the role of the function $g$ in Theorem \ref{Ito_new}. In the proof of Theorem \ref{Ito_new}, local time estimates are used to find a suitable bound on the total time the process $Y$ spends in $M$. However, results on local times are only available in the one-dimensional case. The function $g$ allows to switch to the one-dimensional process $g(Y)$ since $M \subset \{g = 0 \}$ and hence the total  time $g(Y)$ spends in zero is an upper bound for the total time $Y$ spends in $M$. This idea was introduced in \cite{LS18}. The specific construction of $g$ in Proposition \ref{grep} is inspired by the  construction used in \cite[Proof of Theorem 2.7]{LS18}.
\end{Rem}
}

The proof of Theorem~\ref{Ito_new} will be based on a mollification argument.
{See \cite[Appendix A]{LSTcorr} for a similar approach.}
We first provide a number of technical results on convolution and smoothing. 

Recall that for
a locally integrable function $\psi\colon\R^d\to\R$ and a $C^{\infty}$-function $\varphi\colon\R^d\to\R$ with compact support, the convolution of $\psi$ and $\varphi$ is given by
\[
\psi\ast \varphi\colon \R^d\to\R,\,\, x\mapsto \int_{\R^d} \psi(y) \varphi(x-y)\, dy.
\]

Furthermore, recall that any Lipschitz continuous function $\psi\colon\R^d\to\R$ is of at most linear growth and is therefore locally integrable. Moreover,  weak 
partial
derivatives $\partial \psi /\partial x_i$ of $\psi$ are locally integrable and convolutions with $\partial \psi/ \partial x_i$  do not depend on the particular version of $\partial \psi /\partial x_i$.

See~\cite[Theorem 6.30]{Rudin1991} for a proof of the following result.

\begin{Lem} \label{diffswap}
Let 
$\psi\colon\R^d\to\R$ be Lipschitz continuous and let $\varphi\colon \R^{d} \to \R$ be a  $C^{\infty}$-function with compact support. 	Then $\psi \ast \varphi\colon\R^d\to \R$ is a $C^{\infty}$-function and for every $i \in \{1,\dots, d \}$  we have
	\[ 
	\frac{\partial (\psi \ast \varphi)}{\partial x_{i}} = \Bigl(\frac{\partial \psi}{\partial x_{i}}\Bigr) \ast \varphi = \psi \ast \Bigl(\frac{\partial \varphi}{\partial x_{i}}\Bigr). 
	\] 
\end{Lem}

We make use of a standard mollifier given by
\[ 
\eta\colon \R^{d} \to \R,\,\, x \mapsto c^{-1} \exp\bigl((\|x\|^2-1)^{-1}\bigr) \one_{B_1(0)}(x), 
\] 
where $c = \int_{B_1(0)} \exp\bigl((\|x\|^2-1)^{-1}\bigr)\, dx$.  For every $n\in\N$ we define
\begin{equation}\label{moll}
		\eta_{n}\colon \R^{d} \to \R, \,\,x \mapsto n^{d} \eta(n x).
\end{equation}
Then $\eta_n$ is a $C^\infty$-function with $\supp(\eta_n) = \overline B_{1/n}(0)$ and  $\int_{\R^d} \eta_n(x)\, dx=1$.

\begin{Lem} \label{convres}
	Let $\psi\colon \R^{d} \to \R$ be locally integrable. Then for $\lambda_d$-almost all $ x \in \R^{d}$,
	\[ 
	\lim_{n \to \infty} (\psi \ast \eta_{n})(x) = \psi(x).  
	\] 
	If $U \subset \R^{d}$
	 is open and $\psi$ is continuous on $U$ then for every compact set $K \subset U$,
	\[ 
	\lim_{n \to \infty} \|\psi\ast \eta_{n}- \psi\|_{\infty,K} = 0. 
	\] 
\end{Lem}

\begin{proof}
	See 
		\cite[Theorem C.5.7(ii)]{Evans2010} 
   for a proof of the first statement. 
	
		To prove the second statement, let $ \rho = \psi_{|U}$, put 
		$U_n=\{x\in U\mid d(x,\partial U) > 1/n\}$ 
		and let
			\[
		\rho_n\colon U_n\to\R,\,\, x\mapsto \int_{U} \eta_n(x-y)\rho(y)\, dy
		\] 
		 for every $n\in\N$. Since $U$ is open we have $U_n+B_{1/n}(0)\subset U$ for every $n\in\N$ and
		since $K$ is compact and $U$ is open there exists $n_0\in\N$ such that $K\subset U_n$ for every $n\ge n_0$. 
		Thus, for all   $x\in K$ and all $n\ge n_0$,
		\begin{align*}
			\psi \ast \eta_{n}(x) & = \int_{\R^d} \eta_n(y) \psi(x-y)\, dy = \int_{B_{1/n}(0)} \eta_n(y) \psi(x-y)\, dy \\
			& =  \int_{B_{1/n}(0)} \eta_n(y)\rho(x-y)\, dy = \rho_n(x).
		\end{align*}
	By
	\cite[Theorem C.5.7(iii)]{Evans2010} 
	we have $\lim_{n \to \infty} \|\rho_n- \rho\|_{\infty,K} = 0$, which finishes the proof of the lemma.
	\end{proof}

We turn to a result on approximating the components $f_i$ of $f$ in Theorem~\ref{Ito_new} by mollification with $\eta_n$.
\begin{Lem} \label{approxi}
	Let  $h\colon \R^{d} \to \R$ be a  $C^{1}$-function with bounded, Lipschitz continuous derivative. Let $M \subset \R^{d}$ be closed  and assume that $h$ is a $C^{2}$-function on $\R^d\setminus M$. Then, for every $n\in\N$, the function $\phi_{n} = h \ast 	\eta_{n}\colon\R^d\to\R$ is a $C^{\infty}$-function. Moreover, for every compact set $K \subset \R^{d}$ we have\\[-.4cm] 
	\begin{itemize}
		\item[(i)] $\lim_{n \to \infty} \|h - \phi_{n}\|_{\infty,K} = 0$,\\[-.3cm]
		\item[(ii)] $\lim_{n \to \infty} \|h' - \phi_{n}'\|_{\infty,K} = 0$,\\[-.3cm]
	\end{itemize}
	Furthermore, for all $x \in \R^{d} \setminus M$ we have
	\begin{itemize}
		\item[(iii)] $\lim_{n \to \infty} \|h''(x) - \phi_{n}''(x)\| = 0$.\\[-.3cm]
	\end{itemize}
Finally, for any weak derivative $h''$ of $h'$ we have $\sup_{n \in \N} \norm{\phi_{n}''}_{\infty} \le \|h''\|_\infty$. 	
\end{Lem}

\begin{proof}
 Part (i) of the lemma  follows by applying Lemma~\ref{convres} with $\psi=h$ and $U=\R^d$.
 
  Since $h$ is Lipschitz continuous we may  apply Lemma~\ref{diffswap} with $\psi= h$ and $\varphi = \eta_n$ to obtain that for all $n\in\N$, $\phi_n$ is a $C^\infty$-function with
\begin{equation}\label{diffs}
\frac{\partial \phi_n}{\partial x_i} = \frac{\partial h}{\partial x_i} \ast \eta_n
\end{equation}
for all $i\in\{1,\dots,d\}$.
Since $\frac{\partial h}{\partial x_i}$ is continuous we can now apply Lemma~\ref{convres} with $\psi=\frac{\partial h}{\partial x_i}$ and $U=\R^d$ to obtain part (ii) of the lemma.

Next, let $i,j \in \{1,\dots,d \}$ and note that, by assumption, $\rho=\frac{\partial h}{\partial x_{i}}$ is Lipschitz continuous. Using Lemma \ref{diffswap} with  $\psi=\rho$ and $\varphi = \eta_n$ as well as   \eqref{diffs} yields for every $n\in\N$,
\begin{equation} \label{secder} 
	\frac{\partial \rho}{\partial x_j} \ast \eta_{n} =  	\frac{\partial ( \rho \ast \eta_n)}{\partial x_j}  = \frac{\partial^{2} \phi_n}{\partial x_{i} \partial x_{j}}.
\end{equation}
Clearly,  we may assume that  $	\partial \rho /\partial x_j$ is bounded and coincides with $ \partial^{2} h/ \partial x_{i} \partial x_{j}$ on the open set $\R^d\setminus M$, see the remarks on weak derivatives before 
Theorem \ref{Ito_new},
and, in particular,  is continuous on  $\R^d\setminus M$. Using Lemma~\ref{convres} with $U=\R^d\setminus M$ and $\psi = 	\frac{\partial \rho}{\partial x_j}$ as well as \eqref{secder} we conclude that for every compact set $K\subset \R^d\setminus M$,
\[
\lim_{n\to\infty}\Bigl \| \frac{\partial^{2} h}{\partial x_{i} \partial x_{j}} -\frac{\partial^{2} \phi_n}{\partial x_{i} \partial x_{j}} \Bigr\|_{\infty,K} = 0
\]
and, furthermore, for every $x\in \R^d$,
\[
\Bigl|\frac{\partial^{2} \phi_n}{\partial x_{i} \partial x_{j}} (x) \Bigr| = \Bigl|	\Bigl(\frac{\partial \rho}{\partial x_j} \ast \eta_{n} \Bigr)(x)\Bigr| \le \| \partial \rho/ \partial x_j\|_\infty \int_{\R^d} \eta_n(y)\, dy =  \| \partial \rho/ \partial x_j\|_\infty,
\]
which finishes the proof of part (iii)  and the final estimate of the lemma. 
\end{proof}

\subsubsection*{Proof of Theorem~\ref{Ito_new}}

{Put
\[
  C = \| f'' \|_{\infty} \biggl(\frac{2\|g'\|_{\infty}^{2}}{\delta^{2}}\biggr)^{(r-2)/r} \biggl(\int_0^1  \|\beta_s\|^r \, ds \biggr)^{2/r}    \, \biggl( \int_0^1 \bigl \| \beta_t-\gamma_t\bigr\|^2\, dt\biggr)^{(r-2)/r}.
\]}
Let $K \subset \R^{d}$  be a compact neighborhood of $y_{0}$ and consider the stopping time  $\tau_K = 1\wedge \inf\{t \in [0,1] \colon Y_{t} \not\in K \}$. 
Below we show that $\PP$-a.s.,
\begin{equation}\label{stopping}
	\begin{aligned}
		& \sup_{t\in [0,1]}	\Bigl| f(Y_{t\wedge \tau_K}) - f(y_{0})   - \int_{0}^{t\wedge \tau_K} \bigl(f'(Y_s) \alpha_{s} 
	+ \frac{1}{2} \tr(f''(Y_s)\beta_s\beta_s^\top)\bigr) \, ds 
	- \int_{0}^{t\wedge \tau_K} f'(Y_s) \beta_s\, dW_{s}\Bigr|\\ 
	&\qquad\qquad\qquad\qquad 
	\leq  {C}.	
	\end{aligned}
\end{equation}
For  $K_m = \overline B_m(y_0)$, $m\in \N$, we have
\[
\forall \omega\in\Omega\quad \exists m_0\in\N\quad \forall m\geq m_0:\quad \tau_{K_m}(\omega) = 1,
\]
 which jointly with~\eqref{stopping} yields the statement of Theorem~\ref{Ito_new}.

We turn to the proof of~\eqref{stopping}. Let $\phi_n = f\ast \eta_n$ for all $n\in\N$. Using  Lemma~\ref{approxi} with $h=f$ we obtain in particular that $\phi_n$ is a $C^2$-function for all $n\in\N$. Applying the It\^o formula we therefore conclude that $\PP$-a.s.  for all $n\in\N$ and all $t\in [0,1]$,
	\begin{align*}
			\phi_n(Y_{t\wedge \tau_K}) & = \phi_n(y_{0})   + \int_{0}^{t\wedge \tau_K} \bigl(\phi_n'(Y_s) \alpha_{s} 
		+ \frac{1}{2} \tr(\phi_n''(Y_s)\beta_s\beta_s^\top)\bigr) \, ds  + \int_{0}^{t\wedge \tau_K} \phi_n'(Y_s) \beta_s\, dW_{s}.	
	\end{align*}
 Hence, $\PP$-a.s.  for all $n\in\N$ and all $t\in [0,1]$,
\begin{equation}\label{stopping2}
	\begin{aligned}
		&	\Bigl| f(Y_{t\wedge \tau_K}) - f(y_{0})   - \int_{0}^{t\wedge \tau_K} \bigl(f'(Y_s) \alpha_{s} 
		+ \frac{1}{2} \tr( f''(Y_s)\beta_s\beta_s^\top)\bigr) \, ds  - \int_{0}^{t\wedge \tau_K} f'(Y_s) \beta_s\, dW_{s}\Bigr|\\
	& \qquad \le \bigl|f(Y_{t\wedge \tau_K})-\phi_n(Y_{t\wedge \tau_K})\bigr| + |f(y_0) - \phi_n(y_0)| +\int_{0}^{t\wedge \tau_K} \|f'(Y_s) - \phi'_n(Y_s)\| \|\alpha_{s}\|\, ds  \\
	&\qquad\qquad + \frac{1}{2}  \int_{0}^{t\wedge \tau_K} 
	\bigl|\tr(( f''(Y_s)-\phi_n''(Y_s))\beta_s\beta_s^\top)\bigr| \, ds + \biggl|  \int_{0}^{t\wedge \tau_K} \bigl(f'(Y_s)-\phi_n'(Y_s)\bigr) \beta_s\, dW_{s}\biggr|\\
	& \qquad \le  2	\|f - \phi_{n}\|_{\infty,K} + 	\|f' - \phi'_{n}\|_{\infty,K}\int_{0}^{1}  \|\alpha_{s}\|\, ds  + \frac{1}{2 } \int_0^1\| f''(Y_s)-\phi_n''(Y_s)\| \| {\beta}_s\|^2\, ds\\
	&\qquad \qquad  +\sup_{u\in [0,1]}\biggl|  \int_{0}^{u\wedge \tau_K} \bigl(f'(Y_s)-\phi_n'(Y_s)\bigr) \beta_s\, dW_{s}\biggr|.
	\end{aligned}
\end{equation} 	
Using Lemma~\ref{approxi}(i),(ii) with $h=f$ as well as~\eqref{a1} we obtain that  $\PP$-a.s.,
\begin{equation}\label{estt1}
\lim_{n\to\infty}\biggl(	\|f - \phi_{n}\|_{\infty,K} + 	\|f' - \phi'_{n}\|_{\infty,K}\int_{0}^{1}  \|\alpha_{s}\|\, ds\biggr) =0.
\end{equation}
By the last estimate in Lemma~\ref{approxi} we obtain that
\begin{equation}\label{estt2}
\sup_{k\in\N} \| f'' - \phi''_{k}\|_{\infty} \le 2\|f''\|_{\infty}.
\end{equation}

By the Burkholder-Davis-Gundy inequality we get the existence of $c_1, c_2\in (0,\infty)$ such that for all $n\in\N$,
\begin{align*}
\EE\biggl[\sup_{u\in [0,1]}\biggl|  \int_{0}^{u\wedge \tau_K} \bigl(f'(Y_s)-\phi_n'(Y_s)\bigr) \beta_s\, dW_{s}\biggr| \biggr]& \le c_1 \EE\biggl[ \Bigl(\int_0^{1\wedge \tau_K} \|(f'(Y_s)-\phi_n'(Y_s))\beta_s\|^2\, ds\Bigr)^{1/2}\biggr]\\
& \le c_2	\|f' - \phi'_{n}\|_{\infty,K}  \Bigl(\int_0^{1}\EE \bigl[\| \beta_s\|^2\bigr]\, ds\Bigr)^{1/2},
\end{align*}
which jointly with~\eqref{a2} and Lemma~\ref{approxi}(ii) yields
\[
\lim_{n\to\infty} \EE\biggl[\sup_{u\in [0,1]}\biggl|  \int_{0}^{u\wedge \tau_K} \bigl(f'(Y_s)-\phi_n'(Y_s)\bigr) \beta_s\, dW_{s}\biggr| \biggr] = 0.
\]
As a consequence, there exists a 
strictly
increasing sequence $(n_k)_{k\in\N}$ in $\N$ such that $\PP$-a.s.,
\begin{equation}\label{estt3}
\lim_{k\to\infty} \sup_{u\in [0,1]}\biggl|  \int_{0}^{u\wedge \tau_K} \bigl(f'(Y_s)-\phi_{n_k}'(Y_s)\bigr) \beta_s\, dW_{s}\biggr|  = 0.
\end{equation}
Next, we obtain by the properties of $g$ { and \eqref{estt2}} that $\PP$-a.s. for all $n\in\N$,
\begin{equation}\label{estt3b}
\begin{aligned}
& \int_0^1 \| f''(Y_s)-\phi_n''(Y_s)\|  \|{\beta}_s\|^2\, ds\\
&\qquad \le \sup_{k\in\N} \| f'' - \phi''_{k}\|_{\infty} \biggl(\int_0^1  \|{\beta}_s\|^r \, ds \biggr)^{2/r} \biggl( \int_0^1 
\one_{ \{Y_s\in M\}}
 \, ds\biggr)^{(r-2)/r} \\ &\qquad\qquad+ \int_0^1\| f''(Y_s)-\phi_n''(Y_s)\| \|{\beta}_s\|^2  \one_{\{Y_s\in \R^d\setminus M\}}\, ds\\
&\qquad \le { 2 \, \| f''\|_{\infty}} \biggl(\int_0^1  \|{\beta}_s\|^r \, ds \biggr)^{2/r} \delta^{-2(r-2)/r}\biggl( \int_0^1 \one_{ \{g(Y_s)=0\}} \|g'(Y_s)\gamma_s\|^2 \, ds\biggr)^{(r-2)/r} \\ &\qquad\qquad+ \int_0^1\| f''(Y_s)-\phi_n''(Y_s)\| \|{\beta}_s\|^2  \one_{\{Y_s\in \R^d\setminus M\}}\, ds.
	\end{aligned}
\end{equation}
Using Lemma~\ref{approxi}(iii) we obtain $\lim_{n\to\infty}\| f''(Y_s)-\phi_n''(Y_s)\| \|{\beta}_s\|^2  \one_{\{Y_s\in \R^d\setminus M\}} =0$
for all $s\in[0,1]$.
 Observing  { \eqref{estt2} we get}
 $\sup_{k\in\N} \| f''(Y_s)-\phi_k''(Y_s)\| \|{\beta}_s\|^2  \one_{\{Y_s\in \R^d\setminus M\}} \le 2 \|f''\|_{\infty} \|{\beta}_s\|^2 $. { By} the boundedness of $f''$ as well as~{\eqref{a2}} we may therefore conclude by the dominated convergence theorem that $\PP$-a.s.,
\begin{equation}\label{estt4}
	\lim_{n\to\infty} \int_0^1\| f''(Y_s)-\phi_n''(Y_s)\| \|{\beta}_s\|^2  \one_{\{Y_s\in \R^d\setminus M\}}\, ds = 0.
\end{equation}
Since $g$ is a $C^2$-function we may apply the It\^o formula to obtain that $g\circ Y$ is a continuous semi-martingale with quadratic variation 
\[
 \langle g\circ Y\rangle_t = \int_0^t \|(g'(Y_s){\beta}_s\|^2\, ds.
 \]
 Thus, by the occupation time formula,
\[
 \int_0^1 \one_{ \{g(Y_s)=0\}} \|g'(Y_s){\beta}_s\|^2 \, ds = 
\int_{\R} \one_{\{0 \}}(a) L_{1}^{a}(g\circ Y) \, da = 0, 
\]
where $L^{a}(g\circ Y) = (L^{a}_t(g\circ Y))_{t\in[0,1]}$ denotes the local time of $g\circ Y$ at the point $a\in\R$. 
This yields
\begin{equation}\label{estt5}
	  \begin{aligned}
	  \int_0^1 \one_{ \{g(Y_s)=0\}} \|g'(Y_s) {\gamma}_s\|^2 \, ds 
	 &\leq 2 \,  \int_0^1 \one_{ \{g(Y_s)=0\}} \|g'(Y_s){\beta}_s\|^2 \, ds + 2 \| g'\|_{\infty}^{{{2}}} \int_{0}^{1} \| \gamma_{s} - \beta_{s} \|^{2} ds \\
	 & = 2 \, \| g'\|_{\infty}^{{2}} \int_{0}^{1} \| \gamma_{s} - \beta_{s} \|^{2} ds.
	 \end{aligned}
\end{equation}
Combining~\eqref{estt3b} with~\eqref{estt4} and~\eqref{estt5} yields that $\PP$-a.s., 
\begin{equation}\label{estt6}
	\limsup_{n \to \infty} \int_0^1 \| f''(Y_s)-\phi_n''(Y_s)\|  \|\beta_s\|^2\, ds { \,\leq\, 2C.} 
\end{equation}
Finally, combining~\eqref{stopping2} with~\eqref{estt1}, \eqref{estt3} and~\eqref{estt6} we obtain that  $\PP$-a.s., 
	\begin{align*}
	& \sup_{t\in [0,1]}	\Bigl| f(Y_{t\wedge \tau_K}) - f(y_{0})   - \int_{0}^{t\wedge \tau_K} \bigl(f'(Y_s) \alpha_{s} 
	+ \frac{1}{2} \tr( f''(Y_s)\beta_s\beta_s^\top)\bigr) \, ds  - \int_{0}^{t\wedge \tau_K} f'(Y_s) \beta_s\, dW_{s}\Bigr|\\
	& \qquad \le \limsup_{k\to\infty}\biggl( 
	2	\|f - \phi_{n_k}\|_{\infty,K} + 	\|f' - \phi'_{n_k}\|_{\infty,K}\int_{0}^{1}  \|\alpha_{s}\|\, ds \\
	& \qquad \qquad  + \frac{1}{2 } \int_0^1\| f''(Y_s)-\phi_{n_k}''(Y_s)\| \|{\beta}_s\|^2\, ds  +\sup_{u\in [0,1]}\biggl|  \int_{0}^{u\wedge \tau_K} \bigl(f'(Y_s)-\phi_{n_k}'(Y_s)\bigr) \beta_s\, dW_{s}\biggr| 	\biggr)\\
	& \qquad {\le   C,}
\end{align*}
which yields~\eqref{stopping} and hereby completes the proof of Theorem~\ref{Ito_new}. \qed

Finally, we provide {a construction of the function $g$ that is needed to assure that the It\^o formula Theorem~\ref{Ito_new} may be applied with the transformation $G$ and its inverse $G^{-1}$.

We recall that a normal vector along a $C^2$-hypersurface is a $C^1$-function, see Lemma~\ref{normalreg0} in the appendix.

\begin{Prop} \label{grep}
	Let $\emptyset\neq\Theta\subset\R^d$ be an orientable $C^2$-hypersurface of positive reach, let $\nor\colon\Theta\to\R^d$ be a normal vector along $\Theta$, assume that there exists an open neighborhood   $U\subset \R^d$  of $\Theta$ such   that $\nor$ can be extended to a $C^1$-function $\nor\colon U\to\R^d$ with bounded derivative on $\Theta$,  and assume that $\sigma$
	and $\nor$ satisfy
	(A)(ii). Then there exist $\eps\in (0,\reach(\Theta))$ and a $C^2$-function $g\colon\R^d\to\R$ with the following properties. \\[-.3cm]
	\begin{itemize}
		\item[(i)] $\|g\|_\infty + 	\|g'\|_\infty + \|g''\|_\infty < \infty$. \\[-.3cm]
		\item[(ii)] For all $x\in \Theta^\eps$ we have $|g(x)| \le d(x,\Theta)$. \\[-.3cm]
		\item[(iii)] $\inf_{x\in \Theta^\eps} \|g'(x)\sigma(x)\| > 0$.  
	\end{itemize}
\end{Prop}

\begin{proof}
	Choose $\eps^* \in (0,\reach(\Theta))$ such that $\sigma$ satisfies
	\eqref{remcond0}
	with $\eps = \eps^*$. Let 
	\[
	\eps \in (0,\eps^*/5).
	\]
	We first define the function $g$. 
	Let
	\[
	\widetilde{\lambda}\colon \R\to\R, \, x\mapsto 
	\begin{cases}  -\frac{8\eps}{15}, & \text{ if } x < -\eps,\\ x - \frac{2}{3\eps^2} x^3 + \frac{1}{5\eps^4} x^5, & \text{ if }|x|\le \eps,\\
		\frac{8\eps}{15}, & \text{ if } x >\eps, \end{cases}
	\]
	and define  
	\[
	\lambda\colon \R\to\R, \, x\mapsto \widetilde{\lambda}(x) - \widetilde{\lambda}(x+2\eps) +  \frac{8\eps}{15} = 
	\begin{cases} - \widetilde{\lambda}(x+2\eps), & \text{ if } x < -\eps,\\   \widetilde{\lambda}(x) & \text{ if } x \ge -\eps. \end{cases}
	\]
	Let
	\[
	f\colon \R^d\to\R,\, x\mapsto \begin{cases}
		\nor(\pr_{\Theta}(x))^\top (x-\pr_{\Theta}(x)), & \text{ if } x\in \Theta^{4\eps},\\ 4\eps, & \text{ otherwise},
	\end{cases}
	\]
	and put
	\[
	g = \lambda \circ f.
	\]
	
	Next,
	we provide properties of the function $\lambda$. It is easy to check that $\widetilde \lambda$ is a $C^2$-function with
	\[
	\widetilde \lambda'(x) = \Bigl(1-\frac{2}{\eps^2} x^2 + \frac{1}{\eps^4}x^4\Bigr) \one_{[-\eps,\eps]}(x), \,\, \widetilde \lambda''(x) = \Bigl(-\frac{4}{\eps^2} x + \frac{4}{\eps^4}x^3\Bigr) \one_{[-\eps,\eps]}(x)
	\]
	for all $x\in\R$. Hence
	$\lambda$ is a $C^2$-function and
	\begin{equation}\label{pgrep1}
		\|\lambda'\|_\infty + \|\lambda''\|_\infty < \infty.
	\end{equation}
	Furthermore,  $\widetilde \lambda'(x)=0$ iff $|x|\ge \eps$, which yields $\|\widetilde{\lambda}\|_\infty = 8\eps/15$, and therefore
	\begin{equation}\label{pgrep2}
		\|\lambda\|_\infty \le \|\widetilde{\lambda}\|_\infty \le  \eps.
	\end{equation}
	Moreover, for all $x\in (-\eps/2,\eps/2)$ we have $\lambda(x) = \widetilde{\lambda}(x)$, which implies that for all $x\in (-\eps/2,\eps/2)$, 
	\begin{equation}\label{pgrep3}
		\lambda'(x) = 1-\frac{2}{\eps^2} x^2 + \frac{1}{\eps^4}x^4 \ge 1-\frac{2}{\eps^2} x^2 \ge \frac{1}{2}.
	\end{equation}
	Next, it is easy to see that 
	$0\le 2y^2/3 - y^4/5 \le 7/15$
	for all $y\in [-1,1]$, which implies that for all $x\in [-\eps,\eps]$,
	\begin{equation}\label{pgrep4}
		|\lambda(x)| = |\widetilde{\lambda}(x)| = \Bigl| x - \frac{2}{3\eps^2} x^3 + \frac{1}{5\eps^4} x^5\Bigr| = |x| \bigl| 1-\bigl(2 (x/\eps)^2/3 -(x/\eps)^4/5\bigr)\bigr| \le |x|.
	\end{equation}
	Finally, note that for all $x\in \R$ with $|x| \ge 3\eps$ we have
	\begin{equation}\label{pgrep4a}
		\lambda(x) = 8\eps/15
	\end{equation}
	by the definition of $\lambda$.

        Next, we show that for all $x\in \Theta^{4\eps}$, 
	\begin{equation}\label{pgrep5}
		|f(x)| = d(x,\Theta).
		\end{equation}
         If $x\in\Theta$ then $f(x)=0=d(x,\Theta)$. Let $x\in \Theta^{4\eps}\setminus \Theta$
         and note that $4\eps < \text{reach}(\Theta)$.
          Then by \eqref{u01} there exist $y\in\Theta, s\in\{+,-\}$ and 
        $ \eta \in(0,          4\eps)$
          such that $x=y+s\eta\normal(y)$. Since $\normal(y)$ is orthogonal to the tangent space of $\Theta$ at $y$, we have $ \pr_{\Theta}(x)= y $ by Lemma~\ref{fed0} in the appendix. Hence,
\[
|f(x)| = |\nor(y)^\top s\eta\normal(y)|=\eta=d(x,\Theta).
\]

          Clearly, \eqref{pgrep5} implies that $|f(x)|\ge 3\eps$ for all $x\in \Theta^{4\eps}\setminus \Theta^{3\eps}$, and by the definition of $f$ we thus obtain that $|f(x)|\ge 3\eps$ for all $x\in\R^d\setminus \Theta^{3\eps}$. Using~\eqref{pgrep4a} we get that for all $x\in\R^d\setminus \Theta^{3\eps}$,
	\begin{equation}\label{pgrep6}
		g(x) = 8\eps/15,
	\end{equation}
	and, in particular, $g$ is a $C^\infty$-mapping on the open set $\R^d\setminus \cl(\Theta^{3\eps})$.
	
	Recall
	that $4\eps < \reach(\Theta)$. By Lemma~\ref{normalreg0}  and Lemma~\ref{projdist}(i) in the appendix we thus obtain that $\nor\circ\pr_{\Theta}$ and $\pr_{\Theta}$ are $C^1$-functions on $\Theta^{4\eps}$. Hence, $f$ is a $C^1$-function on $\Theta^{4\eps}$ and since $\lambda$ is a $C^2$-function we obtain that  $g$ is a $C^1$-function on $\Theta^{4\eps}$ as well. Since $\Theta^{4\eps}\cup(\R^d\setminus \cl(\Theta^{3\eps})) = \R^d $ we conclude that $g$ is a $C^1$-function.
	
	By 
	Lemma–\ref{ort0}, 
	Lemma~\ref{normalreg0} and Lemma~\ref{projdist}(iii) in the appendix we obtain for every $x\in \Theta^{4\eps}$,
	\begin{equation}\label{pgrep7}
		\begin{aligned}
			f'(x) & = (x-\pr_{\Theta}(x))^\top (\nor\circ\pr_{\Theta} )'(x)+ (\nor\circ\pr_{\Theta}(x))^\top (\cdot - \pr_{\Theta}(\cdot))'(x)\\
			& = \nor(\pr_{\Theta}(x))^\top (I_d - \pr_{\Theta}'(x)) =  \nor(\pr_{\Theta}(x))^\top,
		\end{aligned}
	\end{equation}
	which jointly with ~\eqref{pgrep6} yields
	\begin{equation}\label{pgrep8}
		g'(x) = 	(\lambda\circ f)' (x) = \begin{cases}
			\lambda'(f(x)) \nor(\pr_{\Theta}(x))^\top, & \text{ if } x\in\Theta^{4\eps},\\
			0, & \text{ if }x\in\R^d\setminus\cl(\Theta^{3\eps}).
		\end{cases}
	\end{equation}
	Using~\eqref{pgrep8} and again the fact that $\lambda$  is a $C^2$-function and $\nor\circ\pr_{\Theta}$ is a $C^1$-function on $\Theta^{4\eps}$ as well as the fact that $f$ is a $C^1$-function on $\Theta^{4\eps}$ we can now conclude that $g$ is a $C^2$-function with
	\begin{equation}\label{pgrep9}
		g''(x)  = \begin{cases}
			\lambda'(f(x)) (\nor\circ\pr_{\Theta})'(x)^\top + \lambda''(f(x)) \nor(\pr_{\Theta}(x))\nor(\pr_{\Theta}(x))^\top , & \text{ if } x\in\Theta^{4\eps},\\
			0, & \text{ if }x\in\R^d\setminus\cl(\Theta^{3\eps}).
		\end{cases}
	\end{equation}
	By~\eqref{pgrep2} and the latter two equations  we immediately get 
	\begin{equation}\label{pgrep10}
		\|g\|_\infty \le \eps, \, \|g'\|_\infty \le \|\lambda'\|_\infty, \,\|g''\|_\infty  \le \|\lambda'\|_\infty \|\nor'\|_{\Theta,\infty}  \| \pr'_{\Theta}\|_{\Theta^{4 \eps},\infty} +  \|\lambda''\|_\infty.
	\end{equation}
	Since $\nor'$ is bounded on $\Theta$ we may apply Lemma~\ref{pdiffbd0} in the appendix with $M=\Theta$ and $k=2$
	to obtain 
	$\| \pr'_{\Theta}\|_{\Theta^{4\eps},\infty} <\infty$. 
 Using the latter fact as well as  $\|\nor'\|_{\Theta,\infty}<\infty$ and~\eqref{pgrep1}, we obtain part (i) of the proposition from~\eqref{pgrep10}.
	
	Next, let $x\in \Theta^\eps$. By~\eqref{pgrep5} we then have $|f(x)| =d(x,\Theta) <\eps$. Hence, by~\eqref{pgrep4}, 
	\[
	|g(x)| = |\lambda(f(x))| \le |f(x)| = d(x,\Theta),
	\]
	which proves part (ii) of the proposition.
	
	Finally, let $x\in \Theta^{\eps/2}$. By~\eqref{pgrep5} we then have $|f(x)| <\eps/2$ and therefore $\lambda'(f(x)) \ge 1/2$, due to~\eqref{pgrep3}. Using  
	\eqref{remcond0}
	and~\eqref{pgrep8} we thus obtain
	that there exists $c\in (0,\infty)$ such that for every $x\in \Theta^{\eps/2}$,
	\[
	\|g'(x)\sigma(x)\| = \lambda'(f(x)) \|\nor(\pr_{\Theta}(x))^\top \sigma(x)\|\ge c/2.
	\]
	This proves part (iii) of the proposition with $\eps/2$ in place of $\eps$. Clearly, part (ii) holds for $\eps/2$ in place of $\eps$ as well, and therefore the proof of Proposition~\ref{grep} is complete.
	\end{proof}

\subsection{Proof of Theorem~\ref{exist}}\label{Exist}

	Choose an orientable $C^4$-hypersurface $\emptyset\neq\Theta\subset\R^d$ of positive reach according to (A), a function $G\colon\R^d\to\R^d$ according to Proposition~\ref{G}
	and for every $i\in\{1,\dots,d\}$  bounded extensions $R_i,S_i\colon \R^d \to \R^{d\times d}$ of the second derivatives of $G_{i}$ and $G^{-1}_i$ on $\R^d\setminus \Theta$, respectively, according to Proposition~\ref{G}(v). Define $\sigma_G\colon\R^d\to\R^{d\times d}$ and $\mu_G\colon \R^d\to\R^d$ as in Proposition~\ref{G}(iv) and (v), respectively. Then the latter two functions are Lipschitz continuous and therefore the SDE
	\begin{equation}\label{sdetrans}
		\begin{aligned}
				dY_{t} & = \mu_G(Y_{t})\,dt + \sigma_G(Y_{t})\,dW_{t}, \quad t \in [0,1], \\
			Y_{0} &= G(x_{0})
		\end{aligned}
	\end{equation}
has a unique strong solution $Y=(Y_t)_{t\in[0,1]}$, which satisfies for all 
$q\in [0,\infty)$,
\begin{equation}\label{momenttrans}
	\EE[\|Y\|_\infty^q] < \infty,
\end{equation}
see, e.g.~\cite{Mao08}. 

Choose $\eps\in(0,\reach(\Theta))$ and a $C^2$-function $g\colon\R^d\to\R$ according to Proposition~\ref{grep} and put
\begin{equation}\label{llbb0}
\delta = \inf_{x\in\Theta} \|g'(x) \sigma(x)\|.
\end{equation}
Using Proposition~\ref{G}(iv) and Proposition~\ref{grep}(iii) we then have 
\begin{equation}\label{llbb1}
\inf_{x\in\Theta} \|g'(x) \sigma_G(x)\| = \delta >0,
\end{equation}
 which implies that for all $t\in [0,1]$,
\[
 \|g'(Y_t)\sigma_G(Y_t)\| \one_{\{Y_t\in \Theta\}} \ge  \delta\one_{\{Y_t\in \Theta\}}.
\]
By Proposition~\ref{grep}(ii) we furthermore have $\Theta\subset \{g=0\}$.
Moreover, $\Theta$ is closed
since $\Theta$ is of positive reach.
Since $\lambda_d(\Theta)=0$, see Lemma \ref{leb0} in the appendix, we conclude  that
 $S_i$ is a bounded weak derivative of 
 $(G_i^{-1})'$
 for every $i\in\{1,\dots,d\}$.
Using the latter facts as well as~\eqref{momenttrans}, the continuity of the stochastic process $Y$ and the function $\mu_G$ and
the Lipschitz continuity of the function
 $\sigma_G$ we  conclude that for every $i\in\{1,\dots,d\}$ we may apply Theorem~\ref{Ito_new} with  $y_0=G(x_0)$, $\alpha = \mu_G\circ Y$, $\beta=\gamma=\sigma_G\circ Y$, $M=\Theta$, $f= G^{-1}_i $, $ f'' = S_i$ {and any $r \in (2,\infty)$} 
to obtain that $\PP$-a.s. for all  $t\in[0,1]$,
\begin{align*}
G_i^{-1}(Y_t) & = G_i^{-1}(y_{0})   + \int_{0}^{t} \bigl((G^{-1}_i)'(Y_s) \mu_{G}(Y_s)
+ \frac{1}{2} \tr\bigl(S_i(Y_s)\sigma_G(Y_s)\sigma_G(Y_s)^\top\bigr)\bigr) \, ds \\
& \qquad\qquad + \int_{0}^{t} (G^{-1}_i)'(Y_s)  \sigma_G(Y_s)\, dW_{s}.
\end{align*}
Using Proposition~\ref{G}(iv),(v) we thus conclude that $\PP$-a.s. for all $t\in[0,1]$,
\begin{align*}
 G^{-1}(Y_t) & = G^{-1}(y_{0})   + \int_{0}^{t} \Bigl((G^{-1})'(Y_s) \mu_{G}(Y_s) \\
 & \qquad\qquad\qquad 
	+ \frac{1}{2} \bigl( \tr\bigl(S_1(Y_s)\sigma_G(Y_s)\sigma_G(Y_s)^\top\bigr),\dots, \tr\bigl(S_d(Y_s)\sigma_G(Y_s)\sigma_G(Y_s)^\top\bigr) \bigr)^\top\Bigr) \, ds \\
	& \qquad\qquad + \int_{0}^{t} (G^{-1})'(Y_s)  \sigma_G(Y_s)\, dW_{s}\\
	& = x_0 + \int_0^t \mu (G^{-1}(Y_s))\, ds + \int_0^t \sigma  (G^{-1}(Y_s))\, dW_s, 
\end{align*}
which shows that the stochastic process $G^{-1}\circ Y $ is a strong solution of the SDE ~\eqref{SDE}.

Next assume that there exists a further strong solution $\widetilde X$ of the SDE ~\eqref{SDE} and put $\alpha=\mu\circ \widetilde{X}$ and $\beta=\gamma=\sigma\circ\widetilde{X}$. Since $\mu$ and $\sigma$ are of at most linear growth, see Lemma~\ref{lingrowth}, we obtain that $\EE\bigl[\|\widetilde{X}\|_\infty^q\bigr] < \infty$ for every $q\in (0,\infty)$, see e.g.~\cite{Mao08}. Moreover, there exists $c\in(0,\infty)$ such that $\|\alpha\|_\infty+ \|\beta\|_\infty \le c(1+\|\widetilde{X}\|_\infty)$. Hence $\beta$ satisfies the condition~\eqref{a2} in Theorem~\ref{Ito_new}, and using the continuity of $\widetilde X$ we see that $\alpha$ satisfies the condition~\eqref{a1} in Theorem~\ref{Ito_new}. We furthermore have for all $t\in [0,1]$,
\[
\|g'(\widetilde{X}_t)\sigma(\widetilde{X}_t)\| \one_{\{\widetilde{X}_t\in \Theta\}} \ge  \delta\one_{\{\widetilde{X}_t\in \Theta\}}
\]
with $\delta\in (0,\infty)$ given by~\eqref{llbb0}. For every $i\in\{1,\dots,d\}$ we may thus apply Theorem~\ref{Ito_new} with $y_0 = x_0$,  $\alpha = \mu\circ \widetilde{X}$, $\beta=\gamma=\sigma\circ\widetilde{X}$, $M=\Theta$, $f= G_i$ and 
$ f'' = R_i$ { and any $r \in (2,\infty)$}
 to obtain that $\PP$-a.s. for all $t\in[0,1]$,
\begin{align*}
	G_i(\widetilde{X}_t) & = G_i(x_{0})   + \int_{0}^{t} \bigl(G_i'(\widetilde{X}_s) \mu(\widetilde{X}_s)
	+ \frac{1}{2} \tr\bigl(R_i(\widetilde{X}_s)\sigma(\widetilde{X}_s)\sigma(\widetilde{X}_s)^\top\bigr)\bigr) \, ds \\
	& \qquad\qquad + \int_{0}^{t} G_i'(\widetilde{X}_s)  \sigma(\widetilde{X}_s)\, dW_{s}.
\end{align*}
Hence, $\PP$-a.s. for all $t\in[0,1]$,
\begin{align*}
	G(\widetilde{X}_t) & = G(x_{0})   + \int_{0}^{t} \Bigl(G'(\widetilde{X}_s) \mu(\widetilde{X}_s)\\
	&\qquad\qquad\qquad 
	+ \frac{1}{2} \bigl(\tr\bigl(R_1(\widetilde{X}_s)\sigma(\widetilde{X}_s)\sigma(\widetilde{X}_s)^\top\bigr),\dots, \tr\bigl(R_d(\widetilde{X}_s)\sigma(\widetilde{X}_s)\sigma(\widetilde{X}_s)^\top\bigr)\bigr)^\top\Bigr) \, ds \\
	& \qquad\qquad + \int_{0}^{t} G'(\widetilde{X}_s)  \sigma(\widetilde{X}_s)\, dW_{s}\\
&	= y_0 + \int_0^t \mu_G( G(\widetilde{X}_s))\, ds + \int_0^t \sigma_G(G(\widetilde{X}_s))\, dW_s.
\end{align*}
Thus, $G\circ \widetilde{X}$ is a strong solution of the SDE~\eqref{sdetrans}, which implies $Y=G\circ\widetilde{X}$ $\PP$-a.s. 
We conclude that $G^{-1}\circ Y = \widetilde{X}$ $\PP$-a.s., which finishes the proof of Theorem~\ref{exist}.

\subsection{Moment estimates and occupation time estimates for the Euler-Maruyama scheme}\label{MomOcc}

In this section we provide moment estimates and occupation time estimates for the time-continuous Euler-Maruyama scheme associated to the SDE~\eqref{SDE} that are needed for the proof of Theorems \ref{Thm1} and \ref{Thm2}. For technical reasons we provide these estimates  dependent on the initial value $x_0$. To be formally precise, for every $x\in\R^d$, we consider the SDE
\begin{equation}\label{sde00}
	\begin{aligned}
		dX^x_t & = \mu(X^x_t) \, dt + \sigma(X^x_t) \, dW_t, \quad t\in [0,1],\\
		X^x_0 & = x,
	\end{aligned}
\end{equation}
and for all  $n\in\N$ we use $\eul_{n}^x=(\eul_{n,t}^x)_{t\in[0,1]}$ to denote the time-continuous Euler-Maruyama scheme with step-size $1/n$ associated to the SDE \eqref{sde00}, i.e. $\eul_{n,0}^x=x$ and
\[
\eul_{n,t}^x=\eul_{n,\utn}^x+\mu(\eul_{n,\utn}^x)\cdot (t-\utn)+\sigma(\eul_{n,\utn}^x)\cdot (W_t-W_{\utn})
\]
for every $t\in [0,1]$, where $\utn = \lfloor nt\rfloor /n$ for every $t\in[0,1]$.
In particular, 
$\widehat X_n = \eul_{n}^{x_0}$ for every $n\in\N$. Furthermore,   the integral representation
\begin{equation}\label{intrep}
	\eul_{n,t}^x=x+\int_0^t \mu(\eul_{n,\usn}^x)\, ds+\int_0^t\sigma(\eul_{n,\usn}^x)\,dW_s
\end{equation}
holds for every $n\in\N$ and $t\in[0,1]$.

We have the following uniform $L_p$-estimates for $\eul_{n}^x$, $n\in\N$, which follow  from \eqref{intrep} and the linear growth property of $\mu$ and $\sigma$, see Lemma~\ref{lingrowth}, by using standard arguments.

\begin{Lem}\label{eulprop} Let $\emptyset\neq\Theta\subset\R^d$ be a
$C^1$-hypersurface of positive reach and assume that $\mu$ and $\sigma$ satisfy (A)(iv),(v) and (B).
	Then for all $p\in[1, \infty)$ there exists  $c\in(0, \infty)$ such that for all $x\in\R^d$, all $n\in\N$, all $\delta\in[0,1]$ and all $t\in[0, 1-\delta]$,
	\[
	\bigl(\EE\bigl[\sup_{s\in[t, t+\delta]} \|\eul_{n,s}^x-\eul_{n,t}^x\|^p\bigr]\bigr)^{1/p}\leq c\cdot (1+\|x\|)\cdot \sqrt{\delta}.
	\]
	In particular,
	\[
	\sup_{n\in\N} \bigl(\EE\bigl[\|\eul_{n}^x\|_\infty^p\bigr]\bigr)^{1/p}\leq c\cdot (1+\|x\|).
	\]
\end{Lem}

Next, we provide a Markov
type
 property of the time-continuous Euler-Maruyama 
scheme $\eul^x_n$ relative to the gridpoints $1/n,2/n,\ldots,1$, which is an immediate consequence of the definition of $\eul^x_n$, see also~\cite[Lemma 3]{MGY20}.

\begin{Lem}\label{markov}
	Assume that  $\mu$ and $\sigma$ are measurable. Let $x\in\R^d$, $n\in\N$, $j\in\{0, \ldots, n-1\}$ and $f\colon C([j/n,1];\R^d)\to \R$ be measurable and bounded. Then
	\[
	\EE\bigl[f\bigl((\eul_{n,t}^x)_{t\in [j/n, 1]}\bigr)\bigr | \mathcal F_{j/n}\bigr] =\EE\bigl[f\bigl((\eul_{n,t}^x)_{t\in [j/n, 1]}\bigr)\bigr | \eul_{n,j/n}^x\bigr]\quad \PP\text{-a.s.},
	\]
	and for $\PP ^{\eul_{n,j/n}^x} $-almost all $y\in\R^d$,
	\[
	\EE\bigl[f \bigl( (\eul_{n,t}^x)_{t\in [j/n, 1]} \bigr) \bigr |\eul_{n,j/n}^x=y\bigr] =\EE\bigl[f\bigl( (\eul^y_{n,t-j/n})_{t\in [j/n,1]}\bigr)\bigr].
	\]
\end{Lem}

We proceed with
an estimate for the expected occupation time  of a neighborhood of the hypersurface $\Theta$ by the time-continuous Euler-Maruyama scheme $\eul_{n}^x$.
The following result is a generalization of  \cite[Lemma 4]{MGY20}, where the case $d=1$ and $\Theta=\{\xi\}$ with $\xi\in\R$ is studied.
\begin{Lem} \label{occtime}
Let $\emptyset\neq\Theta\subset\R^d$ be an orientable $C^2$-hypersurface of positive reach, let $\nor\colon\Theta\to\R^d$ be a normal vector along $\Theta$, assume that there exists an open neighborhood   $U\subset \R^d$  of $\Theta$ such   that $\nor$ can be extended to a $C^1$-function $\nor\colon U\to\R^d$ with bounded derivative on $\Theta$, and assume that  $\mu$, $\sigma$
and $\nor$
 satisfy (A)(ii),(iv),(v) and (B).
	Then there exists $c \in (0,\infty)$ such that for all $x \in \R^{d}$, all $n \in \N$ and all $\eps \in [0,\infty)$, 
	\begin{equation*}
		\int_{0}^{1} \PP(\{\widehat{X}_{n,t}^{x} \in \Theta^{\eps} \}) dt \leq c(1+\|x\|^{2})\Bigl(\eps + \frac{1}{{n}}\Bigr).
	\end{equation*}
\end{Lem}

\begin{proof}
 Let $x\in\R^d$ and $n\in\N$ and note that by~\eqref{intrep}, Lemma~\ref{lingrowth} and Lemma~\ref{eulprop}, the process $\eul_n^{x}$ is a continuous semi-martingale. Choose $\eps\in (0,\reach(\Theta))$ and a $C^2$-function $g\colon\R^d\to\R$
according to Proposition~\ref{grep}, put
 \[
 Y^x_n = g\circ \eul_n^{x}, \,\,\kappa = \inf_{y\in\Theta^\eps}\|g'(y) \sigma (y)\|
 \] 
 and note that $\kappa >0$ due to Proposition~\ref{grep}(iii).
 By the It\^{o} formula we obtain that $Y^x_n$ is a continuous semi-martingale such that almost surely, for all $t\in[0,1]$,
 \[
 Y^x_{n,t} = g(x) + \int_0^t U_s\, ds + \int_0^t V_s\, dW_s,
 \] 
 where 
 \[
 U_s = g'(\eul^x_{n,s})\mu(\eul^x_{n,\usn}) + \frac{1}{2}\tr \bigl(g''(\eul^x_{n,s})\sigma(\eul^x_{n,\usn}) \sigma(\eul^x_{n,\usn})^\top\bigr),\,\, V_s = g'(\eul^x_{n,s})\sigma(\eul^x_{n,\usn})
 \]
 for all $s\in [0,1]$. Moreover, $Y^x_n$ has quadratic variation
 \[
 \langle Y_n^x\rangle_t = \int_0^t { \|V_{s}\|^{2}}\, ds.
 \]
 For $a\in\R$ let $L^a(Y_n^x) = (L^a_t(Y_n^x))_{t\in[0,1]}$ denote the local time of $Y^x_n$ at the point $a$. By  the Tanaka formula, see e.g.~\cite[Chap. VI]{RevuzYor2005}, we have for all $a\in\R$ and all $t\in[0,1]$,
 \[
 L^a_t(Y_n^x) = |Y^x_{n,t} -a|
 -|g(x)-a| 
- \int_0^t \sgn(Y_{n,s}^x-a)\, U_s \, ds - \int_0^t \sgn(Y_{n,s}^x-a)\, V_s \, dW_s 
 \]
 and therefore
 \begin{equation}\label{oct0}
 	L^a_t(Y_n^x) \le |Y^x_{n,t} -g(x) | + \int_0^t |U_s| \, ds + \biggl| \int_0^t \sgn(Y_{n,s}^x-a)\, V_s \, dW_s \biggr|. 
 \end{equation}
 Using Lemma~\ref{lingrowth} and Proposition~\ref{grep}(i) we get that there exist $c_1,c_2\in (0,\infty)$ such that for all $s\in [0,1]$,
 \begin{equation}\label{oct1}
 	|U_s| \le c_1\bigl( \|g'\|_\infty (1+\|\eul_n^x\|_\infty) + \|g''\|_\infty (1+\|\eul_n^x\|_\infty)^2 \bigr) \le c_2 (1+\|\eul_n^x\|^2_\infty)
 \end{equation}
  and
  \begin{equation}\label{oct2}
 	{ \|V_{s}\|^{2}} \le c_1 \|g'\|^2_\infty (1+\|\eul_n^x\|_\infty)^2  \le c_2 (1+\|\eul_n^x\|^2_\infty).
 \end{equation} 

Using 
the H\"older inequality, the Burkholder-Davis-Gundy inequality, the estimates~\eqref{oct1} and~\eqref{oct2}  and the second estimate in Lemma~\ref{eulprop} we conclude  that there exist $c_1,c_2\in (0,\infty)$ such that for all $x\in \R^d$, all $n\in\N$, all $a\in\R$ and   all $t\in [0,1]$,
\begin{equation}\label{oct3}
	\begin{aligned}
	\EE\bigl[L^a_t(Y_n^x)\bigr] & \le 
	c_1 \int_0^1 \EE [|U_s|] \, ds + c_1 \biggl( \int_0^1 \EE [ { \|V_{s}\|^{2}}] \, ds \biggr)^{1/2}\\ & \le c_2 \bigl(1+\EE\bigl[\|\eul_n^x\|^2_\infty\bigr]\bigr)\le c_3 (1+\|x\|^2). 
	\end{aligned}
\end{equation}
Let $\tilde\eps\in [0,\infty)$.
 By  the occupation time formula, see e.g. \cite[Chap. VI]{RevuzYor2005}, and~\eqref{oct3} we conclude that there exists $c\in (0,\infty)$ such that for all $x\in \R^d$, all $n\in\N$ and all $a\in\R$,
\begin{equation}\label{oct4}
	\begin{aligned}
		\EE\biggl[\int_0^1 \one_{[-\tilde\eps,\tilde\eps]} (Y_{n,t}^x)\, { \|V_{s}\|^{2}} \, dt\biggr] = \int_{\R} \one_{[-\tilde\eps,\tilde\eps]} (a) \EE\bigl[L^a_t(Y_n^x)\bigr]\, da \le c(1+\|x\|^2)\tilde\eps. 
	\end{aligned}
\end{equation}	
   Put 
   \[
   R_t = g'(\eul^x_{n,t})\sigma(\eul^x_{n,t})
   \]
    for every $t\in[0,1]$. 
	Using Proposition~\ref{grep}(i), the Lipschitz continuity of $\sigma$ and  Lemma~\ref{eulprop} we obtain that there exist $c_1,c_2 \in (0,\infty)$ such that for all $x\in \R^d$ and all $n\in\N$,
	\begin{equation} \label{diffbd}
		\begin{aligned}
		&	\int_{0}^{1}\EE\bigl[ { \|R_{t} - V_{t}\|^{2}}\bigr]\, dt\\
			&\qquad = \int_{0}^{1} \EE\bigl[ {\| g'(\eul_{n,t}^{x}) ( \sigma (\eul_{n,t}^{x}) - \sigma (\eul_{n,\utn}^{x})) \|^{2}}\bigr]\, dt  \\
			&\qquad \le { \|g'\|^{2}_{\infty}} \int_{0}^{1}\EE\bigl[ \|\sigma (\eul_{n,t}^{x}) - \sigma (\eul_{n,\utn}^{x}) \|^{{2}}\bigr] \, dt\\
						&\qquad\leq { c_{1}\|g'\|_{\infty}^{2}}\,  \int_{0}^{1} \EE\bigl[ \|\eul_{n,t}^{x} - \eul_{n,\utn}^{x}\|^{2}\bigr] \, dt  \le c_{{2}} (1+\|x\|^{2}) \frac{1}{{n}}.
		\end{aligned}
	\end{equation}
Without loss of generality we may assume that $\tilde\eps \le \eps$. Employing Proposition~\ref{grep}(ii) as well as~\eqref{oct4} and~\eqref{diffbd} we conclude that there exists $c\in (0,\infty)$ such that for all $x\in \R^d$ and all $n\in\N$, 
\begin{equation}\label{oct5}
	\begin{aligned}
&\int_{0}^{1} \PP(\{\widehat{X}_{n,t}^{x} \in \Theta^{\tilde\eps} \}) dt=
\EE\biggl[	\int_{0}^{1} \one_{\{\widehat{X}_{n,t}^{x} \in \Theta^{\tilde\eps} \}}\, dt\biggr]\\ & 
\qquad =
 \frac{1}{\kappa^2} 	\,	\EE\biggl[	\int_{0}^{1} \kappa^2\one_{\{\widehat{X}_{n,t}^{x} \in \Theta^{\tilde\eps} \} } \one_{\{ |g(\eul_{n,t}^x)| < \tilde \eps  \}}   \, dt\biggr] \le \frac{1}{\kappa^2} \,\EE\biggl[\int_0^1 \one_{[-\tilde\eps,\tilde\eps]} (Y_{n,t}^x) { \|R_{t}\|^{2}} \, dt\biggr]	 \\ & \qquad
		 {\le \frac{2}{\kappa^{2}} \EE\biggl[ \int_{0}^{1} \one_{[-\tilde{\eps},\tilde{\eps}]}(Y_{n,t}^{x})\|V_{t} \|^{2} dt \biggr] + \frac{2}{\kappa^{2}} \EE\biggl[\int_{0}^{1} \| R_{t} - V_{t} \|^{2} dt\biggr]}  
		 \le c(1+\|x\|^2)\Bigl(\tilde \eps + \frac{1}{{ n}}\Bigr), 
	\end{aligned}
\end{equation}
which finishes the proof of Lemma~\ref{occtime}.
	\end{proof}

\begin{Rem} \label{remoccup}
In \cite[Theorem 2.7]{LS18} 
and its corrected version~\cite[Theorem 2.7]{LScorr2} 
and in \cite[Theorem 2.8]{NSS19} estimates for the  expected occupation time  of a neighborhood of  a 
 hypersurface  $\Theta$ of positive reach by an It\^{o} process are proven. These estimates are then applied to 
the time-continuous Euler-Maruyama scheme $\eul_n$ and a time-continuous adaptive Euler-Maruyama scheme to prove $L_2$-error rates of at least $1/4-$ and $1/2-$, respectively. See the proof of
 \cite[Theorem 3.1]{LS18, LScorr2} 
and \cite[Lemma 3.3]{NSS19}. Note, however, that these estimates  in fact  cannot be applied in  any of these cases, since neither of the two schemes satisfies the 
respective
conditions on the It\^{o} process  of \cite[Theorem 2.7]{LS18, LScorr2} and \cite[Theorem 2.8]{NSS19}
under the respective assumptions on the coefficients $\mu$ and $\sigma$ in
the corrected version~\cite{LScorr2} of
 \cite{LS18} and in \cite{ NSS19}. 
These estimates can also not be applied to the Euler-Maruyama scheme $\eul_n$ under the assumptions
(A) and (B) of the actual paper.

Indeed, the conditions of \cite[Theorem 2.7]{LS18, LScorr2} (they coincide with the conditions of \cite[Theorem 2.8]{NSS19}) applied to the Euler-Maruyama scheme $\eul_n$, are as follows: 
\begin{itemize}
\item[(1)] there exist $\varepsilon_1\in(0, \reach(\Theta) )$ and $c_1>0$ such that for $\PP$-almost all $\omega\in\Omega$ and all $t\in[0,1]$,
\[
\eul_{n,t}(\omega)\in\Theta^{\varepsilon_1}\quad \Rightarrow\quad \max\bigl(\|\mu(\eul_{n,\utn}(\omega))\|, \|\sigma(\eul_{n,\utn}(\omega))\|\bigr)\leq c_1,
\]
\item[(2)]  there exist $\varepsilon_2\in(0, \reach(\Theta) )$ and  $c_2>0$ such that for $\PP$-almost all $\omega\in\Omega$ and all $t\in[0,1]$,
\begin{equation}\label{cond}
\eul_{n,t}(\omega)\in\Theta^{\varepsilon_2}\quad \Rightarrow\quad \|\normal(\pr_\Theta( \eul_{n,t}(\omega)))^\top \sigma( \eul_{n,\utn}(\omega))\|  \geq c_2.
\end{equation}
\end{itemize}
Due to the boundedness of the coefficients $\mu$ and $\sigma$ in \cite{LS18, LScorr2}, condition (1) is clearly fulfilled. Due to  Assumption 2.1.4 in \cite{LS18, LScorr2} (assumption (A)(ii) in the present paper), \eqref{cond} in condition (2) is fulfilled if $t=\utn$, see Remark \ref{remcond00}. However,
\eqref{cond} in condition (2) does not have to be
fulfilled for all $t\in[0,1]$ in general.

Indeed, consider the SDE \eqref{SDE} with $d=1$, $x_0=1$, $\mu=0$
and $\sigma=x\cdot\one_{[0,1]}(x)+\one_{(1, \infty)}(x)$.  Then $\mu$ and $\sigma$ satisfy  the Assumption 2.1 
in the corrected version~\cite{LScorr2} of
\cite{LS18} 
(and also assumptions
(A) and (B) of the actual paper)
with $\Theta=\{1/2\}$. Moreover, for all $n\in\N$,
	\begin{align*}
		\eul_{n,0}&=1, \quad \eul_{n, 1/n}=1+W_{1/n},\\
		\eul_{n,t}&=1+W_{1/n}+\sigma(1+W_{1/n})\cdot (W_t-W_{1/n}), \quad t\in(1/n, 2/n].
	\end{align*}
Condition (2) implies that 
there exist $\varepsilon_2>0$ and  $c_2>0$ such that for $\PP$-almost all $\omega\in\Omega$ and all $t\in[3/(2n), 2/n]$,
\begin{equation}\label{contr1}
\eul_{n,t}(\omega)\in(0.5-\varepsilon_2, 0.5+\varepsilon_2)\quad \Rightarrow\quad |\sigma( \eul_{n,1/n}(\omega))|\geq c_2.
\end{equation}
This, however, does not hold. Indeed, 
let $\varepsilon_2>0$ and  $c_2>0$,
put $\tilde \varepsilon_2 = \min (1/4,\varepsilon_2)$ and choose $k\in\N$ such that 
$k>\max(\tfrac{1}{c_2}, \tfrac{1}{4\tilde \varepsilon_2}-\tfrac{1}{2})$. Note that $0 <(1/2-\tilde\varepsilon_2)(k+1) < (1/2+\tilde\varepsilon_2)k$ and put
\[
A=\{1+W_{1/n}\in[\tfrac{1}{k+1}, \tfrac{1}{k}]\}\cap \{\forall t\in [3/(2n), 2/n]\colon  1+W_t-W_{1/n}\in ((1/2-\tilde \varepsilon_2)(k+1), (1/2+\tilde \varepsilon_2)k)\}.
\]
We then have $\PP(A)>0$ and for all $\omega\in A$,
\[
\sigma(1+W_{1/n}(\omega)) = 1+W_{1/n}(\omega).
\]
Moreover, for all $\omega\in A$ and all $t\in [3/(2n), 2/n]$,
\[
\eul_{n,t}(\omega) = (1+W_{1/n}(\omega)) (1+W_t(\omega)-W_{1/n}(\omega))   \in(1/2-\tilde \varepsilon_2, 1/2+\tilde \varepsilon_2)\subset (1/2- \varepsilon_2, 1/2+\varepsilon_2),
\]
but
\[
 |\sigma( \eul_{n,1/n}(\omega))|=|1+W_{1/n}(\omega)  | \le 1/k < c_2, 
\]
which contradicts
\eqref{contr1}.

The coefficients  $\mu$ and $\sigma$  satisfy
 Assumption 2.1 in \cite{NSS19} as well.
Moreover, for the SDE under consideration, the adaptive Euler-Maruyama scheme from \cite{NSS19} with the parameter $\delta=1/n$  coincides with the Euler-Maruyama scheme $\eul_n$ on the set $A$ for $t\in[0, 2/n]$ if $n$ is large enough.  Thus, condition (2) is not fulfilled in general also for the adaptive Euler-Maruyama scheme. 

Moreover, also condition (1) is not fulfilled for the adaptive Euler-Maruyama scheme in general under  Assumption 2.1 from \cite{NSS19}. Indeed, consider the SDE \eqref{SDE} with $d=1$, $x_0=1$, $\mu=0$
and $\sigma(x)=x$, $x\in\R$. Then $\mu$ and $\sigma$ satisfy  Assumption 2.1 in \cite{NSS19}
(and also assumptions
(A) and (B) of the actual paper)
 with $\Theta=\{1/2\}$.  Furthermore, for all $n\in\N$,
\begin{align*}
	\eul_{n,0}&=1, \quad \eul_{n, 1/n}=1+W_{1/n},\\
	\eul_{n,t}&=(1+W_{1/n})\cdot (1+W_t-W_{1/n}), \quad t\in(1/n, 2/n].
\end{align*}
Condition (1) implies that 
there exist $\varepsilon_1>0$ and  $c_1>0$ such that for $\PP$-almost all $\omega\in\Omega$ and all $t\in[3/(2n), 2/n]$,
\begin{equation}\label{contr2}
\eul_{n,t}(\omega)\in(0.5-\varepsilon_1, 0.5+\varepsilon_1)\quad \Rightarrow\quad | \eul_{n,1/n}(\omega)|\leq c_1.
\end{equation}
This, however, does not hold. Indeed,
let $\varepsilon_1>0$ and  $c_1>0$,
put $\tilde \varepsilon_1 = \min (1/4,\varepsilon_1)$ and choose $k\in\N$ such that 
$k>\max(c_1, \tfrac{1}{4\tilde \varepsilon_1}-\tfrac{1}{2})$. Note that $0 <(1/2-\tilde\varepsilon_1)/k < (1/2+\tilde\varepsilon_1)/(k+1)$ and put
\[
B=\{1+W_{1/n}\in[k,k+1]\}\cap \{\forall t\in [3/(2n), 2/n]\colon  1+W_t-W_{1/n}\in ((1/2-\tilde \varepsilon_1)/k, (1/2+\tilde \varepsilon_1)/(k+1))\}.
\]
We then have $\PP(B)>0$ and for all $\omega\in B$ and all $t\in [3/(2n), 2/n]$,
\[
\eul_{n,t}(\omega)\in(1/2-\tilde \varepsilon_1, 1/2+\tilde \varepsilon_1) \subset (1/2-\varepsilon_1, 1/2+\varepsilon_1),
\]
but
\[
 |\eul_{n,1/n}(\omega)|=|1+W_{1/n}(\omega)  | \ge k  > c_1,
\]
which  contradicts 
\eqref{contr2}.
Finally observe again that the adaptive Euler-Maruyama scheme  from \cite{NSS19}  with the parameter $\delta=1/n$  coincides with the Euler-Maruyama scheme $\eul_n$ on the set $B$ for $t\in[0, 2/n]$ if $n$ is large enough.
\end{Rem}

 We turn to the main result in this section, which provides an $L_p$-estimate of the total amount of times $t$ 
 that
 $\eul_{n,t}$ and $\eul_{n,\utn}$ stay on 'different sides' of the hypersurface $\Theta$ of potential discontinuities of the drift coefficient $\mu$. 
 
 \begin{Prop} \label{occtime2}
Let $\emptyset\neq\Theta\subset\R^d$ be an orientable $C^2$-hypersurface of positive reach, let $\nor\colon\Theta\to\R^d$ be a normal vector along $\Theta$, assume that there exists an open neighborhood   $U\subset \R^d$  of $\Theta$ such   that $\nor$ can be extended to a $C^1$-function $\nor\colon U\to\R^d$ with bounded derivative on $\Theta$, and assume that  $\mu$, $\sigma$ 
and $\nor$
satisfy (A)(ii),(iv),(v) and (B).
 Then for all $p\in [1,\infty)$ and all $\delta\in (0,1/2)$ there  exists $c \in (0,\infty)$ such that for  all $n \in \N$,
 	\[
 	\EE\biggl[\Bigl|	\int_{0}^{1} \one_{\{d(\eul_{n,\utn}, \Theta)\le\|\eul_{n,t} - \eul_{n,\utn}\| \}}\, dt\Bigr|^p\biggr]^{1/p} \leq \frac{c}{n^{1/2-\delta}}.
 	\]
 \end{Prop}

For the proof of Proposition~\ref{occtime2} we first establish the following auxiliary estimate. For all $t \in [0,1]$ and all $n \in \N$ we put	
\begin{equation}\label{setA}
A_{n,t} = \{ d(\eul_{n,\utn}, \Theta)\le\|\eul_{n,t} - \eul_{n,\utn}\| \}. 
\end{equation}

\begin{Lem}\label{normbd}
Let $\emptyset\neq\Theta\subset\R^d$ be an orientable $C^1$-hypersurface of positive reach
and assume that $\mu$ and $\sigma$ satisfy (A)(iv),(v) and (B).
	Then for all $\delta \in (0,1/2)$ and all $\rho \in (0,1)$ there exists $c \in (0,\infty)$ such that for all $n \in \N$, all $t\in [0,1]$  and all $A \in \mathcal F$,  
		\[
			\PP(A \cap A_{n,t}) \leq \frac{c}{n}\,\PP(A)^{\rho} +  \PP \Bigl(A \cap \Bigl\{d(\eul_{n,t},\Theta) \leq \frac{2}{n^{1/2 - \delta}} \Bigr\}\Bigr). 
	\]
\end{Lem}

\begin{proof}

Fix  $\delta \in (0,\frac{1}{2})$ and $\rho \in (0,1)$. 
	First, note that for all $x,y\in\R^d$ with $d(y,\Theta) \le \|x-y\|$ we have $	d(x,\Theta) \le \|x-y\| + d(y,\Theta) \le 2\|x-y\|$, which implies that for all $n \in \N$ and all $t\in [0,1]$,
	\[
A_{n,t}\subset \{d(\eul_{n,t},\Theta) \leq 2 \|\eul_{n,t}-\eul_{n,\utn} \| \}.
	\]
	Hence,  for all $n \in \N$, all $t\in [0,1]$  and all $A \in \mathcal F$,
	\begin{equation} \label{newbd}
			\PP\Bigl(A \cap A_{n,t} \cap \Bigl\{ \| \eul_{n,t}-\eul_{n,\utn}\|     \le \frac{1}{n^{1/2- \delta}}\Bigr\} \Bigr)  \leq \PP\Bigl(A \cap \Bigl\{d(\eul_{n,t},\Theta) \leq \frac{2}{n^{1/2 - \delta}} \Bigr\}\Bigr).
\end{equation}

		By Lemma \ref{eulprop} and the Markov inequality we obtain that for all $p \in [1,\infty)$ there exists  $c \in (0,\infty)$ such that for all $n \in \N$ and $t \in [0,1]$, 
	\begin{equation}\label{newbd2}
		\PP\Bigl(    \|\eul_{n,t}- \eul_{n,\utn}  \|      >  \frac{1}{n^{1/2- \delta}}  \Bigr)  \leq  \EE \bigl[ \|  \eul_{n,t}- \eul_{n,\utn}  \|^p \bigr] \, n^{ (1/2 - \delta)p }  
		 \leq  \frac{c}{n^{\delta p}}.
	\end{equation}
Employing~\eqref{newbd2} with $p=(\delta(1-\rho))^{-1}$ we conclude that there exists  $c \in (0,\infty)$ such that for all $n \in \N$, all $t \in [0,1]$ and all $A \in \mathcal F$,
	\begin{equation} \label{newbd3}
	\begin{aligned}
		&\PP\Bigl(A \cap A_{n,t}\cap \Bigl\{ \| \eul_{n,t}-\eul_{n,\utn} \|     > \frac{1}{n^{1/2- \delta}}\Bigr\} \Bigr) \\
		&\qquad \qquad \le \PP\Bigl(A \cap \Bigl\{ \|\eul_{n,t}-\eul_{n,\utn}\|     > \frac{1}{n^{1/2- \delta}}\Bigr\} \Bigr)\\
		&\qquad \qquad \le  \PP(A)^{\rho} \, \PP\Bigl( \| \eul_{n,t}-\eul_{n,\utn} \|     > \frac{1}{n^{1/2- \delta}} \Bigr)^{1-\rho} \leq  \PP(A)^{\rho} \, \frac{c}{n}.
	\end{aligned}
\end{equation}
	Combining \eqref{newbd} and \eqref{newbd3}  completes the proof of Lemma~\ref{normbd}.
	\end{proof}

Based on Lemma~\ref{normbd} we establish the following estimate, which in particular yields Proposition~\ref{occtime2} in the case $p=1$. 
\begin{Lem}\label{normbd2}
Let $\emptyset\neq\Theta\subset\R^d$ be an orientable $C^2$-hypersurface of positive reach, let $\nor\colon\Theta\to\R^d$ be a normal vector along $\Theta$, assume that there exists an open neighborhood   $U\subset \R^d$  of $\Theta$ such   that $\nor$ can be extended to a $C^1$-function $\nor\colon U\to\R^d$ with bounded derivative on $\Theta$, and assume that  $\mu$, $\sigma$ 
and $\nor$
satisfy (A)(ii),(iv),(v) and (B).
	Then for all $\delta \in (0,1/2)$ and $\rho \in (0,1)$ there exists $c \in (0,\infty)$ such that for all $n \in \N$, all $s\in [0,1]$  and all $A \in \mathcal F_s$,
	\[
	\int_s^1 \PP(A \cap A_{n,t}) \, dt \leq \frac{c}{n^{1/2-\delta}} \,\PP(A)^{\rho}. 
	\]
\end{Lem}

 \begin{proof}

 The 
 inequality
 trivially holds for $s=1$. Fix  $\delta \in (0,1/2)$ and $\rho \in (0,1)$. By Lemma~\ref{normbd} we obtain that there exists $c\in (0,\infty)$ such that for for all $n \in \N$, all $s \in [0,1)$ and all $A \in \mathcal F_s$,
                \begin{equation} \label{bd1}
                	\begin{aligned}
                	                               \int_{s}^{1} \PP(A \cap A_{n,t}) dt & \leq \frac{\PP(A)}{n}  +  \int_{\underline{s}_{n} + \frac{1}{n}}^{1} \PP(A \cap A_{n,t}) dt\\
                	                            &\le   \frac{c}{n} \PP(A)^{\rho} + \int_{\underline{s}_{n} + \frac{1}{n} }^{1} \PP\Bigl( A\cap \Bigl\{d(\eul_{n, t},\Theta)  \leq \frac{2}{n^{1/2 - \delta}} \Bigr\}\Bigr)\, dt. 
                	                            \end{aligned}
                \end{equation}
                Using Lemma~\ref{markov} we get that for all $n \in \N$, all $s \in [0,1)$ and all $A \in \mathcal F_s$,
                \begin{equation} \label{betbd}
                               \begin{aligned}
                                               &\int_{\underline{s}_{n} + \frac{1}{n}}^{1}  \PP\Bigl( A\cap \Bigl\{d(\eul_{n, t},\Theta)  \leq \frac{2}{n^{1/2 - \delta}} \Bigr\}\Bigr)\, dt \\
                                               &\qquad\qquad = \EE\biggl[ \EE\biggl[ \one_{A} \int_{\underline{s}_{n} + \frac{1}{n}}^{1} \one_{\bigl\{ d(\eul_{n,t},\Theta) \leq \frac{2}{n^{1/2 - \delta}}\bigr\}}\, dt\, \biggr| \mathcal F_{\underline{s}_{n} + \frac{1}{n}}\biggr]\biggr] \\
                                               &\qquad\qquad =\EE\biggl[ \one_{A} \EE\biggl[  \int_{\underline{s}_{n} + \frac{1}{n}}^{1} \one_{\bigl\{ d(\eul_{n,t},\Theta) \leq \frac{2}{n^{1/2 - \delta}}\bigr\}}\, dt \, \biggr| \eul_{n,\underline{s}_{n} + \frac{1}{n}}\biggr]\biggr].
                               \end{aligned}
                \end{equation}
                By Lemma~\ref{markov} and Lemma~\ref{occtime} we furthermore derive that there exists $c\in (0,\infty)$ such that  for all $n \in \N$, all $s \in [0,1)$ and  $\PP^{\eul_{n,\underline{s}_{n} + \frac{1}{n}}}$-almost all $x\in\R^d$,
                \begin{equation}\label{nextx1}
                               \begin{aligned}
                                               \EE\biggl[  \int_{\underline{s}_{n} + \frac{1}{n}}^{1} \one_{\bigl\{ d(\eul_{n,t},\Theta) \leq \frac{2}{n^{1/2 - \delta}}\bigr\}}\, dt \, \biggr| \eul_{n,\underline{s}_{n} + \frac{1}{n}} = x\biggr] & =
                                               \EE  \biggl[  \int_{0}^{1-(\underline{s}_{n} + \frac{1}{n})} \one_{\bigl\{ d(\eul_{n,t}^x,\Theta) \leq \frac{2}{n^{1/2 - \delta}}\bigr\}}\, dt\biggr]  \\
                                               &\leq c (1+\|x\|^{2}) \frac{1}{n^{1/2-\delta}}.
                               \end{aligned}
                \end{equation}
                Inserting~\eqref{nextx1} into~\eqref{betbd} and employing Lemma~\ref{eulprop} we conclude that there exist $c_1,c_2,
                c_3
                \in (0,\infty)$ such that for all $n \in \N$, all $s \in [0,1)$ and all $A \in \mathcal F_s$,
                \begin{equation} \label{betbd3}
                               \begin{aligned}
                                               \int_{\underline{s}_{n} + \frac{1}{n}}^{1}  \PP\Bigl( A\cap \Bigl\{d(\eul_{n, t},\Theta)
                                                 \leq \frac{2}{n^{1/2 - \delta}} \Bigr\}\Bigr)\, dt
                                               &            \le  \frac{c_1}{n^{1/2-\delta}} \EE\bigl[ \one_A\, (1+\|\eul_{n,\underline{s}_{n} + \frac{1}{n}}\|^{2})\bigr]\\
                                                &  \le   \frac{c_2}{n^{1/2-\delta}} \PP(A)^\rho\EE\bigl[  1+\|\eul_{n}\|_\infty^{2/(1-\rho)}\bigr]^{1-\rho} \\
                                                  & \le  \frac{c_3}{n^{1/2-\delta}} \PP(A)^\rho.
                               \end{aligned}
                \end{equation}
                Combining~\eqref{bd1} with~\eqref{betbd3} completes the proof of Lemma~\ref{normbd2}.

\end{proof}

We turn to the proof of Proposition~\ref{occtime2}.

\subsubsection*{Proof of Proposition~\ref{occtime2}} Let $\delta \in (0,1/2)$. Clearly, we may assume that $p\in\N$ and $p\ge 2$. Then, for all $n\in\N$,
\[ 
\EE\biggl[ \Bigl(\int_{0}^{1} \one_{A_{n,t}} \, dt\Bigr)^{p}\biggr]= p! \int_{0}^{1} \int_{t_{1}}^{1}\dots \int_{t_{p-1}}^{1} \PP(A_{n,t_{1}} \cap \dots \cap A_{n,t_{p}}) \, dt_{p} \dots dt_2\, dt_{1}.  
\] 
Let $\tilde\delta \in(0,  \delta)$ 
and $\rho\in (0,1)$. 
Iteratively applying Lemma~\ref{normbd2} $(p-1)$-times with $\tilde\delta$ in place of $\delta$, $s = t_{k}$ and $A= A_{n,t_{1}} \cap \dots \cap A_{n,t_{k}}\in \mathcal F_{t_k}$   for $k = p-1,\dots, 1$ and finally applying Lemma~\ref{normbd2} with $\tilde\delta$ in place of $\delta$, $A=\Omega$ and $s=0$ we conclude that 
 there exist  $c_{1},\dots,c_{p} \in (0,\infty)$ depending only on $\tilde\delta$ and $\rho$ such that for all $n \in \N$,
\begin{equation} \label{eqbd}
	\begin{aligned}
		\EE\biggl[\Bigl(\int_{0}^{1} \one_{A_{n,t}} \, dt\Bigr)^{p}\biggr] & \leq p! \frac{c_{1}}{n^{1/2 - \tilde{\delta}}} \Bigl(\int_{0}^{1} \int_{t_{1}}^{1} \dots \int_{t_{p-2}}^{1} \PP( A_{n,t_{1}} \cap \dots \cap A_{n,t_{p-1}}) \, dt_{p-1} \dots dt_2 \,dt_{1}\Bigr)^{\rho} \\
		&\leq p! \,  \frac{c_{1} \cdots c_{p-1}}{n^{1/2-\tilde{\delta}} n^{(1/2- \tilde{\delta})\rho} \cdots n^{(1/2- \tilde{\delta}) \, \rho^{p-2}}} \Bigl(\int_{0}^{1} \PP(A_{n,t_{1}}) \, dt_{1}\Bigr)^{\rho^{p-1}} \\
		&\leq  p! \, \frac{c_{1} \cdots c_{p} }{n^{(1/2 - \tilde{\delta}) \frac{1-\rho^{p}}{1-\rho} } }.
	\end{aligned}
\end{equation} 

Since $\tilde\delta < \delta<1/2$ there exists $\eps \in (0,1)$ such that $p (1/2 - \delta) \leq (p-\eps)(1/2 - \tilde{\delta} )$. Since $\lim_{\rho\to 1} (1-\rho^{p})/(1-\rho) = p$ there exists $\rho \in (0,1)$ such that $(1-\rho^{p})/(1-\rho) \geq p - \eps$. With this choice of $\rho$ 
in \eqref{eqbd} we finally conclude  that there exists $c > 0$ such that for all $n \in \N$,
\[ 
\EE\biggl[ \Bigl(\int_{0}^{1} \one_{A_{n,t}} \, dt\Bigr)^{p}\biggr] \leq \frac{c}{n^{(1/2 - \tilde{\delta}) \frac{1-\rho^{p}}{1-\rho} } } \leq \frac{c}{n^{p(1/2 - \delta)}},
\] 
which completes the proof of Proposition~\ref{occtime2}. \qed

\subsection{Proof of Theorem~\ref{Thm1}}\label{ProofThm1}
Clearly, we may assume that $p\in [2,\infty)$
and $\delta\in (0,1/2)$. 
Choose a $C^4$-hypersurface $\emptyset\neq\Theta\subset\R^d$ of positive reach according to (A), a function $G\colon\R^d\to\R^d$ according to Proposition~\ref{G}
and for every $i\in\{1,\dots,d\}$  bounded extensions $R_i,S_i\colon \R^d \to \R^{d\times d}$ of the second derivatives of $G_{i}$ and $G^{-1}_i$ on $\R^d\setminus \Theta$, respectively, according to Proposition~\ref{G}(v). Moreover, choose $\eps\in(0,\reach(\Theta))$ and a $C^2$-function $g\colon\R^d\to\R$ according to Proposition~\ref{grep}.

Define $\sigma_G\colon\R^d\to\R^{d\times d}$ and $\mu_G\colon \R^d\to\R^d$ as in Proposition~\ref{G}(iv) and (v), respectively,  let $Y=(Y_t)_{t\in[0,1]}$ be a strong solution 
of the corresponding SDE~\eqref{sdetrans} and for every $n\in\N$ let $\euly_n$ denote   the associated time-continuous Euler-Maruyama scheme, i.e. $\euly_{n,0}=G(x_0)$ and
\[
\euly_{n,t}=\euly_{n,i/n}+\mu_G(\euly_{n,i/n})\, (t-i/n)+\sigma_G(\euly_{n,i/n})\, (W_t-W_{i/n})
\]
for $t\in (i/n,(i+1)/n]$ and $i\in\{0,\ldots,n-1\}$. Since $\mu_{G}$ and $\sigma_G$ are Lipschitz continuous,  there exists $c\in (0,\infty)$ such that for all $n\in\N$,
\begin{equation}\label{eultr1}
	\EE\bigl[ \| \euly_{n}\|_\infty^p\bigr]\le c
\end{equation}
and
\begin{equation}\label{eultr2}
	\EE\bigl[ \|Y-\euly_{n}\|_\infty^p\bigr]^{1/p}\leq \frac{c}{\sqrt{n}}.
\end{equation}
Note further that the Lipschitz continuity of $G$ and  Lemma~\ref{eulprop} imply that there exist $c_1,c_2\in (0,\infty)$ such that for all $n\in\N$,
\begin{equation}\label{pro2} 
	\EE\bigl[ 	\|G\circ \eul_{n}\|_\infty^p\bigr] 
	\leq c_1 \bigl(1+\EE\bigl[\|\eul_{n}\|_\infty^p\bigr]\bigr) \le c_2.
\end{equation}

Recall from the proof of Theorem~\ref{exist} that the process $G^{-1}\circ Y$ is a strong solution of the SDE~\eqref{SDE}. Using the Lipschitz continuity of $G^{-1}$ and~\eqref{eultr2} we may thus conclude that  there exist $c_1,c_2\in (0,\infty)$ such that for all $n\in\N$,
\begin{equation}\label{pro1}
\EE\bigl[\| X-\eul_n\|_\infty^p\bigr]^{1/p} \le c_1 \EE\bigl[\| Y-G\circ \eul_n\|_\infty^p\bigr]^{1/p} \le \frac{c_2}{\sqrt{n}} + c_1 \EE\bigl[\| \euly_n-G\circ \eul_n\|_\infty^p\bigr]^{1/p}.
\end{equation}
It therefore remains to analyze the quantity $\EE\bigl[\| \euly_n-G\circ \eul_n\|_\infty^p\bigr]^{1/p}$. 
For every $n\in\N$ we define a function $u_n\colon[0,1]\to [0,\infty) $ by
\[
u_n(t)=\EE\bigl[\sup_{s\in[0, t]} \|\euly_{n,s}-G(\eul_{n,s})\|^p\bigr]
\]
for every $t\in[0,1]$. 
Note that the functions $u_n$, $n\in\N$, are well-defined and bounded
due to \eqref{eultr1} and \eqref{pro2}. 

Below we show that there exists  $c\in(0, \infty)$  such that for all $n\in\N$ 
and all $t\in[0,1]$,
\begin{equation}\label{G7}
	u_n(t)\leq c\biggl(\frac{1}{n^{p(1/2-\delta)}}+  	\EE\biggl[\Bigl|	\int_{0}^{1}
	 \one_{\{d(\eul_{n,\usn}, \Theta)
	 \le\|\eul_{n,s} - \eul_{n,\usn}\| \}}\, ds\Bigr|^{2p}\biggr]^{1/2}+\int_0^t u_n(s)\, ds\biggr).
\end{equation}
Using Proposition~\ref{occtime2} we conclude from~\eqref{G7} that 
there exists $c\in (0,\infty)$ such that for all $n\in\N$ and all $t\in[0,1]$,
\[
u_n(t)\leq c \Bigl( \frac{1}{n^{p(1/2-\delta)}} +\int_0^t u_n(s)\, ds\Bigr).
\]
By Gronwall's inequality it then follows that there exists $c\in (0,\infty)$ such that for all $n\in\N$,
\begin{equation}\label{p1}
\EE\bigl[\| \euly_n-G\circ \eul_n\|_\infty^p\bigr] =	u_n(1) \le \frac{c}{n^{p(1/2-\delta)}},
\end{equation} 
which jointly with~\eqref{pro1} yields the statement of Theorem~\ref{Thm1}.

It remains to prove~\eqref{G7}. 
Recall that
for all $n\in\N$ and all $t\in[0,1]$,
\[
\eul_{n,t}=x_0+\int_0^t \mu(\eul_{n,\usn})\, ds+\int_0^t \sigma(\eul_{n,\usn})\,dW_s.
\]
Note that $\lambda_d(\Theta)=0$, see Lemma \ref{leb0} in the appendix. We conclude by Proposition~\ref{G}
that $R_i$ is a bounded weak derivative of 
$G_i'$
 for every $i\in\{1,\dots,d\}$.
 Note that 
 $\Theta$ is closed, 
 since $\Theta$ is of positive reach.
Observing 
Proposition~\ref{grep}(iii)
as well as Lemma~\ref{lingrowth} and Lemma~\ref{eulprop}  we thus may apply Theorem~\ref{Ito_new} with $y_0=x_0$, $\alpha_t = \mu(\eul_{n,\utn})$, $\beta_t=   \sigma(\eul_{n,\utn})$ and $\gamma_t =  \sigma(\eul_{n,t})$ for $t\in[0,1]$,  $M=\Theta$, $f=G_i$ and $ f'' = R_i$  
for every $i\in\{1,\dots,d\}${ and $r = 4$} to obtain
that there exists $c\in (0,\infty)$ such that $\PP\text{-a.s.}$ for all $n\in\N$,
\begin{equation}\label{prox1}
		\sup_{t\in[0,1]}\|G(\eul_{n,t}) - Z_{n,t}\| \le c {\biggl( \int_{0}^{1} \| \sigma(\eul_{n,\utn})  \|^{4} dt \biggr)^{1/2}}  \biggl(\int_0^1 \bigl\| { \sigma  (\eul_{n,\utn})-\sigma (\eul_{n,t})}\bigr\|^{{2}}\, dt\biggr)^{{ 1/2}}, 
	\end{equation}
where, for all $n\in\N$, the stochastic process $Z_n = (Z_{n,t})_{t\in[0,1]} $ is given by
\begin{align*}
	Z_{n,t} & =G(x_0)+\int_0^t\Bigl(G'(\eul_{n,s})\cdot \mu(\eul_{n,\usn})+\frac{1}{2}  \bigl(\tr \bigl( R_i(\eul_{n,s})\cdot(\sigma\sigma^\top)(\eul_{n,\usn})\bigr)\bigr)_{1\leq i\leq d}\Bigr)\,ds\\
	& \qquad +\int_0^t G'(\eul_{n,s})\cdot\sigma(\eul_{n,\usn})\,dW_s. 
\end{align*}
Clearly, for all $n\in\N$ and all $t\in[0,1]$,
\begin{align*}
	Z_{n,t}
	& = G(x_0)+\int_0^t \mu_G(G(\eul_{n,\usn}))\, ds
	+ \int_0^t \bigl(G'(\eul_{n,s})-G'(\eul_{n,\usn})\bigr)\cdot \mu(\eul_{n,\usn})\,ds \\
	& \qquad +\int_0^t\sigma_G(G(\eul_{n,\usn}))\,dW_s 
	+  \int_0^t \bigl(G'(\eul_{n,s})-G'(\eul_{n,\usn})\bigr)\cdot \sigma(\eul_{n,\usn})\,dW_s\\
	& \qquad + \frac{1}{2}\cdot \int_0^t \bigl( \tr \bigl( (R_i(\eul_{n,s})-R_i(\eul_{n,\usn}))\cdot (\sigma\sigma^\top)(\eul_{n,\usn})\bigr)\bigr)_{1\leq i\leq d}\,ds. 
\end{align*}
It follows that  there exists $c\in (0,\infty)$ such that $\PP\text{-a.s.}$ for all $n\in\N$ and all $t\in[0,1]$,
\begin{align*}
&\|\euly_{n,t}-G(\eul_{n,t})\| \le\|\euly_{n,t}-Z_{n,t}\|+\|Z_{n,t}-G(\eul_{n,t})\|\\
&\qquad \leq  \sum_{i=1}^3 \|V_{n,i, t}\|  + c  {\biggl( \int_{0}^{1} \| \sigma(\eul_{n,\usn})  \|^{4} ds \biggr)^{1/2}} \biggl(\int_0^1 \bigl\| { \sigma (\eul_{n,s})  - \sigma (\eul_{n,\usn})}\bigr\|^{{2}}\, ds\biggr)^{{1/ 2}},
\end{align*}
where 
\begin{align*}
	V_{n,1,t} & = \int_0^t (\mu_G(G(\eul_{n,\usn}))-  \mu_G(\euly_{n,\usn}))\, ds+ \int_0^t(\sigma_G(G(\eul_{n,\usn}))-  \sigma_G(\euly_{n,\usn}))\,dW_s,\quad \\
	V_{n,2,t}&=\int_0^t \bigl(G'(\eul_{n,s})-G'(\eul_{n,\usn})\bigr)\cdot \mu(\eul_{n,\usn})\,ds 
	+  \int_0^t \bigl(G'(\eul_{n,s})-G'(\eul_{n,\usn})\bigr)\cdot \sigma(\eul_{n,\usn})\,dW_s,\\
	V_{n,3,t} & =  \frac{1}{2}\cdot \int_0^t \bigl( \tr \bigl( (R_i(\eul_{n,s})-R_i(\eul_{n,\usn}))\cdot (\sigma\sigma^\top)(\eul_{n,\usn})\bigr)\bigr)_{1\leq i\leq d}\,ds. 
\end{align*}
Hence, there exists $c\in 
(0,\infty)$
such that for all $n\in\N$ and for all $t\in[0,1]$,
\begin{equation}\label{jj}
	\begin{aligned}
	&u_n(t)\leq c \cdot \biggl( \sum_{i=1}^3 \EE\Bigl[\sup_{s\in[0,t]} \|V_{n,i,s}\|^p\Bigr] \\
	&\qquad\qquad + \EE\biggl[{\biggl(\Bigl( \int_{0}^{1} \| \sigma(\eul_{n,\usn})  \|^{4} ds \Bigr)^{1/2}} \Bigr(\int_0^1   \|{ \sigma (\eul_{n,s})  - \sigma (\eul_{n,\usn})}\|^{{ 2}} \, ds\Bigr)^{{ 1/2}}{\biggr)^{p}}\biggr]\biggr).
	\end{aligned}
\end{equation}

We next estimate the single summands on the right hand side of \eqref{jj}. 
Using the H\"older inequality, the Lipschitz continuity of $\sigma$ as well as Lemma~\ref{lingrowth} and  Lemma~\ref{eulprop} we obtain that there exist $c_1,c_2,c_3 \in (0,\infty)$ such that for all $n\in\N$
\begin{equation} \label{prox2}
\begin{aligned}
	&\EE\biggl[ {\biggl(\Bigl( \int_{0}^{1} \| \sigma(\eul_{n,\usn})  \|^{4} ds \Bigr)^{1/2}} \Bigr(\int_0^1   \| {\sigma  (\eul_{n,s})  - \sigma (\eul_{n,\usn})}\|^{{2}} \, ds\Bigr)^{{ 1/2}}{\biggr)^{p}}\biggr] \\
	& \qquad { \leq c_{1} \EE\biggl[ \biggl( \sup_{s \in [0,1]}(1+\| \eul_{n,s} \|^{2})\Bigl( \int_{0}^{1} \| \eul_{n,s} - \eul_{n,\usn} \|^{2} ds \Bigr)^{\frac{1}{2}}\biggr)^{p}\biggr]} \\
	& \qquad{\leq c_{2} \biggl( 1+ \sup_{k \in \N} \EE\Bigl[ \| \eul_{k} \|_{\infty}^{4p}\Bigr]\biggr)^{1/2} \EE\biggl[ \Bigl( \int_{0}^{1} \| \eul_{n,s} - \eul_{n,\usn} \|^{2p} ds \Bigr) \biggr]^{1/2}} \\
	& \qquad { \leq c_{3} \frac{1}{n^{p/2}}.}
\end{aligned}
\end{equation}	
Using the H\"older inequality, the Burkholder-Davis-Gundy inequality and the
Lipschitz continuity of $\mu_G$ and $\sigma_G$
we obtain that there exists
$c\in(0, \infty)$ such that for all $n\in\N$ and all $t\in[0,1]$,
\begin{equation}\label{t3}
	\EE\Bigl[\sup_{s\in[0,t]} \|V_{n,1,s}\|^p\Bigr]\leq c\cdot \int_0^t \EE\bigl[\|G(\eul_{n,\usn}))-\euly_{n,\usn}\|^p\bigr]\, ds\leq c\cdot \int_0^t u_n(s)\, ds.
\end{equation}
Furthermore, using the H\"older inequality, the Burkholder-Davis-Gundy inequality as well as the Lipschitz continuity of $G'$, see Proposition \ref{G}(ii), and employing Lemma~\ref{lingrowth} as well as  Lemma~\ref{eulprop} we conclude that there exist
$c_1, c_2,c_3\in(0, \infty)$ such that for all $n\in\N$ and all $t\in[0,1]$,
\begin{equation}\label{t4}
	\begin{aligned}
	& 	\EE\bigl[\sup_{s\in[0,t]} \|V_{n,2,s}\|^p\bigr]\\
	&\qquad\qquad \leq c_1 \int_0^t \EE\bigl[\|G'(\eul_{n,s}))-G'(\eul_{n,\usn})\|^p\cdot (\|\mu(\eul_{n,\usn})\|^p+\|\sigma(\eul_{n,\usn})\|^p\bigr)\bigr]\, ds\\
	&	\qquad\qquad \leq c_2 \int_0^t \bigl(\EE\bigl[\|\eul_{n,s}-\eul_{n,\usn}\|^{2p}\bigr]\bigr)^{1/2}\cdot\bigl(1+ \EE\bigl[\|\eul_{n,\usn}\|^{2p}\bigr]\bigr)^{1/2}\, ds\leq \frac{c_3}{n^{p/2}}.
	\end{aligned}
\end{equation}
Next, recall the definition~\eqref{setA} of the sets $A_{n,s}$ and note that, by Proposition~\ref{G}(iii),(v), the mappings $R_i\colon\R^d\to \R^{d\times d}$, $i=1,\dots,d$, are Lipschitz continuous on the set $\R^d\setminus \Theta$ w.r.t. the intrinsic metric $\rho$ on $\R^d\setminus \Theta$.  Note further, that for all $x,y\in\R^d$ with $d(x,\Theta)>\|x-y\|$ one has 
$\{x+\lambda (y-x)\mid \lambda\in [0,1]\}\subset \R^d\setminus \Theta$, which in turn implies that $\rho(x,y) = \|x-y\|$. We can thus 
 conclude  that there exists $c\in (0,\infty)$ such that for all $n\in\N$, all $s\in [0,1]$ and all $i\in\{1,\dots,d\}$,
\begin{equation}\label{prox3}
	\begin{aligned}
\|R_i(\eul_{n,s}) - R_i(\eul_{n,\usn})\|
\one_{\Omega\setminus A_{n,s}} 
\leq c \rho(\eul_{n,s}, \,\eul_{n,\usn})  \one_{\Omega\setminus A_{n,s}}  =  c \|\eul_{n,s} - \eul_{n,\usn}\| \one_{\Omega\setminus A_{n,s}}. 
\end{aligned}
\end{equation}
Using~\eqref{prox3},  Lemma \ref{lingrowth} and the boundedness of the functions $R_1,\dots,R_d$, see Proposition~\ref{G}(v), we obtain that there exist $c_1,c_2,
c_3
\in (0,\infty)$ such that for all $n \in \N$ and all $t\in [0,1]$,
\begin{align*}
		\|V_{n,3,t}\| & \le c_1 \, \int_0^t \bigl\| \bigl( \tr \bigl( (R_i(\eul_{n,s})-R_i(\eul_{n,\usn}))\cdot \sigma\sigma^\top(\eul_{n,\usn})\bigr)\bigr)_{1\leq i\leq d}\bigr\| \,ds\\
		& \le c_2 \int_0^1 \bigl\| \bigl( \|R_i(\eul_{n,s})-R_i(\eul_{n,\usn})\| \cdot \|\sigma(\eul_{n,\usn})\|^2\bigr)_{1\leq i\leq d}\bigr\| \,ds\\
		& \le 
		c_3
		 \bigl(1+\|\eul_{n}\|_\infty^2\bigr) \int_0^1 \bigl( \|\eul_{n,s} - \eul_{n,\usn}\| \one_{\Omega\setminus A_{n,s}} + \one_{A_{n,s}}\bigr)\, ds.
\end{align*}
Hence, by Lemma~\ref{eulprop}, there exist $c_1,c_2,c_3\in (0,\infty)$ such that for all $n\in\N$ and all $t\in [0,1]$,
\begin{equation}\label{prox4}
\begin{aligned}
 	\EE\bigl[\sup_{s\in[0,t]} \|V_{n,3,s}\|^p\bigr] &\le c_1\EE\biggl[ \bigl(1+\|\eul_{n}\|_\infty^{2p}\bigr)\biggl(  \int_0^1  \|\eul_{n,s} - \eul_{n,\usn}\|^p\, ds  + \Bigl(\int_0^1 \one_{A_{n,s}}\, ds\Bigr)^p\biggr] \\
 	& \le c_2 \bigl( 1+ \EE[ \| \eul_n\|_\infty^{4p}]^{1/2}\bigr) \biggl( \Bigl( \int_0^1 \EE\bigl[ \|\eul_{n,s} - \eul_{n,\usn}\|^{2p}  \bigr]\, ds \Bigr)^{1/2} \\
 	&\qquad\qquad\qquad \qquad\qquad\qquad\qquad + \EE \Bigl[ \Bigl( \int_0^1 \one_{A_{n,s} }\, ds \Bigr)^{2p}\Bigr]^{1/2}\biggr)\\
 	&\le c_3 \biggl(\frac{1}{n^{p/2}}+ \EE \Bigl[ \Bigl( \int_0^1 \one_{A_{n,s} }\, ds \Bigr)^{2p}\Bigr]^{1/2}\biggr).
\end{aligned}
\end{equation}
Combining~\eqref{jj} with~\eqref{prox2}, ~\eqref{t3}, ~\eqref{t4} and~\eqref{prox4} yields~\eqref{G7} and hereby completes the proof of Theorem~\ref{Thm1}.

\subsection{Proof of Theorem~\ref{Thm2}}\label{ProofThm2}

Theorem~\ref{Thm2} is proven similarly as Theorem 2 in~\cite{MGY20}. For convenience of the reader we present the proof here.

	Let $p\in [1,\infty)$ and $\delta \in (0,\infty)$. Clearly, for all $n\in\N$, 
	\begin{equation}\label{two0}
		\EE\bigl[\|X-\overline X_n\|_\infty^p\bigr]^{1/p} 
		\le \EE\bigl[\|X-\eul_n\|_\infty^p\bigr]^{1/p} +\EE\bigl[\|\eul_n-\overline X_n\|_\infty^p\bigr]^{1/p}. 
	\end{equation}
	Moreover, by Theorem~\ref{Thm1} there exists $c\in (0,\infty)$ such that for all $n\in\N$,
	\begin{equation}\label{two1}
		\EE\bigl[\|X-\eul_n\|_\infty^p\bigr]^{1/p} \le c/n^{1/2 - \delta}.
	\end{equation}
	
	For every
	$n\in\N$ let $\overline W_n = (\overline W_{n,t})_{t\in[0,1]}$ denote the equidistant piecewise linear interpolation 
	 of the Brownian motion $W$, i.e.
	\[
	\overline W_{n,t} = (n\cdot t-i)\cdot W_{(i+1)/n} + (i+1-n\cdot t)\cdot  W_{i/n}
	\]
	for $t\in [i/n,(i+1)/n]$ 	and $i\in\{0,\ldots,n-1\}$. 
	Then for every $r\in [1,\infty)$
	there exists $c\in (0,\infty)$ such that for all $n\in\N$,
	\begin{equation}\label{two01}
		\EE\bigl[\|W-\overline W_n\|_\infty^r\bigr]^{1/r}\leq  c\sqrt{\ln (n+1)}/\sqrt{n},
	\end{equation}
	see, e.g.~\cite{F92}.

	Note that for all $n\in\N$ and all $t\in[0,1]$, 
	\[
		\|\eul_{n,t} - \overline X_{n,t}\|  = 
		\Bigl\|\sum_{i=0}^{n-1}\sigma(\eul_{n,i/n})\cdot \one_{[i/n,(i+1)/n]}(t)\cdot (W_t- \overline W_{n,t})\Bigr\|
		 \le \|\sigma(\eul_{n})\|_\infty \cdot \|W_t-\overline W_{n,t}\|.
	\]
Employing  Lemma \ref{lingrowth},  Lemma~\ref{eulprop}
and~\eqref{two01}
 we thus conclude that there exist $c_1,c_2\in (0,\infty)$ such that for all $n\in\N$,
	\begin{align*}
		\EE\bigl[\|\eul_n-\overline X_n\|_\infty^p\bigr]^{1/p} &\le c_1  \bigl(1 + \EE\bigl[\|\eul_n\|_\infty^{2p}\bigr]^{1/(2p)}\bigr)\cdot \EE\bigl[\|W-\overline W_n\|_\infty^{2p}\bigr]^{1/{(2p)}}\\
&\le c_2  \sqrt{\ln (n+1)}/\sqrt{n} ,
	\end{align*}
	which jointly with~\eqref{two0} and~\eqref{two1} yields \eqref{l4} and completes the proof of Theorem~\ref{Thm2}. \qed

\section{Examples}\label{Exmpl}
We present a class of coefficients $\mu$ and $\sigma$ satisfying the  conditions (A) and (B), which extends Example 2.6 in \cite{LS18}.

Let $\emptyset\neq\Theta \subset \R^{d}$ be a compact, orientable $C^{4}$-hypersurface with normal vector $\nor$ along $\Theta$. 
Note that in this case $\Theta$ is always of positive reach, see  
\cite[Proposition 14]{T2008}.
Let $n \in \N$ and let $K_{1}, \dots, K_{n}\subset \R^d$ be open and pairwise disjoint sets with
\begin{equation}\label{Exm2}
\bigcup_{i = 1}^{n} K_{i}=\R^{d}\setminus \Theta.
\end{equation}
Let the drift coefficient $\mu \colon \R^{d} \to \R^{d}$ be of the form
\begin{equation}\label{Exm1}
\mu= f_{0} \one_{\Theta} + \sum_{i = 1}^{n} f_{i} \one_{K_{i}},
\end{equation}
where the functions $f_{0},\dots,f_{n} \colon \R^{d} \to \R^{d}$ satisfy
	\begin{itemize}
		\item[(i)] $f_{0}$ is bounded on $\Theta$,
                  \item[(ii)] $f_{i}$ is Lipschitz continuous on $K_i$ for all $i \in \{1,\dots,n \}$,
         	\item[(iii)] there exists an open set $U\subset \R^d$ such that $\Theta\subset U$ and  $f_1, \ldots, f_n$ are $C^{3}$ on $U$. 		
	\end{itemize}
Moreover, let the diffusion coefficient $\sigma \colon \R^{d} \to \R^{d \times d}$
	satisfy
	\begin{itemize}
	  \item[(iv)] $\sigma$ is Lipschitz continuous, 
	  	\item[(v)] there exists an open set $U\subset \R^d$ such that $\Theta\subset U$ and  $\sigma$ is $C^{3}$ on $U$.
	  \item[(vi)] $\nor(x)^{\top}\sigma(x)  \neq 0$ for all $x \in \Theta$. 
	\end{itemize}

	We show that $\mu$ and $\sigma$ satisfy  (A) and (B). Clearly, we only need to prove that (A) is satisfied.
	
	By Lemma~\ref{normalreg0} in the appendix,  the function $\nor$ is $C^{3}$.
	Hence there exists an open set $U\subset \R^d$ with $\Theta \subset U$ and a $C^{3}$-mapping $g\colon U\to \R^d$ such that $\nor = g_{|\Theta}$, see \cite[Remark 1.1]{Fu95}. Since $\Theta$ is compact, all partial derivatives of $g$ are bounded on $\Theta$. Hence, (A)(i) is satisfied.

	Next, we prove 
	that (A)(ii) 	 is satisfied. 
	 By  (iv) and the continuity of $\nor$, the function $g:= \|\nor^{\top} \sigma\|\colon \Theta\to \R$ is continuous. Since $\Theta$ is compact, there exists $x_0\in\Theta$ such that $g(x_0) = \inf_{x\in\Theta}g(x)$. By (vi) we have $g(x_0) > 0$.

	Now, we prove that (A)(iii) is satisfied. Choose $\varepsilon \in (0,\reach(\Theta))$ 
	according to 
	\eqref{remcond0}
	and choose 
$U\subset \Theta^{\varepsilon}$
according to (iii) and (v). By Lemma \ref{connect0} in the appendix there exists an open set 
$V\subset \R^d$ such that $x\in V$ and $V\cap\Theta$ is connected.
Define $\phi \colon \Theta \times \R \to \R^{d}$ by
\[
\phi(y,h) = y+ h \nor(y)
\]  
and put
\[
B_{1} = \phi((V\cap\Theta) \times (0,\varepsilon)), \quad 
B_{2} = \phi((V\cap\Theta) \times (-\varepsilon,0)).
\]
Since $V$ is open there exists $\delta \in (0,\varepsilon)$ such that $B_{\delta}(x)  \subset V$.
Clearly, 
\[
\{y + h \nor(y) \mid y \in B_{\delta}(x)\cap \Theta, \; h \in (0,\delta) \} \subset B_{1}
\]
and
\[
\{y + h \nor(y) \mid y \in B_{\delta}(x)\cap \Theta, \; h \in (-\delta,0) \} \subset B_{2}.
\]
Since $\phi$ is continuous and $(V\cap\Theta) \times (0,\varepsilon)$ as well as $(V\cap\Theta) \times (-\varepsilon,0)$ are connected, the sets $B_{1}$ and $B_{2}$ are also connected. Moreover, by  Lemma~\ref{fed0} in the appendix we have $B_1, B_2 \subset \R^{d}\setminus\Theta=\bigcup_{i = 1}^{n} K_{i}$. Thus there 
exist $i,j \in \{1,\dots,d \}$ such that $B_{1} \subseteq K_{i}$ and $B_{2} \subseteq K_{j}$.
Using (iii) we hence obtain that for all $y \in B_{\delta}(x) \cap \Theta$, the value $\alpha(y)$ is well-defined with
\begin{equation}\label{extra}
	\alpha(y)= \frac{f_{j}(y) - f_{i}(y)}{2 \|\sigma(y)^\top\nor(y)\|^2}=\frac{f_{j}(y) - f_{i}(y)}{2 \|\sigma(y)^\top\nor(\pr_{\Theta}(y))\|^2}.
\end{equation}
By Lemmas \ref{projdist}(i)  and \ref{normalreg0} in the appendix, the function $\nor \circ \pr_{\Theta}$ is $C^{3}$ on 
$\Theta^{\varepsilon}$. 
Thus, using (iii), (v)
and \eqref{remcond0}
 we conclude  
 that the right hand side in \eqref{extra} defines a $C^{3}$-mapping on $U$. This proves that $\alpha$ can be extended to a $C^{3}$-function on $U$.
Moreover, 
since $\Theta$ is compact, $\alpha$ has bounded partial derivatives up to order $3$ on $\Theta$. Thus, (A)(iii) holds.

We turn to the proof of (A)(iv). 
Choose $U$ according to (iii). 
We first prove that there  exists $\varepsilon\in (0,\infty)$ such that 	 		 $\Theta^\varepsilon\subset U$. In fact, assume that 
for every $n\in\N$ there  exists $x_n\in \Theta^{1/n}$ with $x_n\in \R^d\setminus U$. Since $\Theta$ is bounded, the sequence  $(x_n)_{n\in\N}$ is bounded, which implies the existence of $x_0\in \R^d$ and of a subsequence $(x_{n_k})_{k\in\N}$ of $(x_n)_{n\in\N}$ with $\lim_{k\to\infty} x_{n_k} = x_0$. Since $d(x_{n_k},\Theta) \le 1/n$ for every $k\in\N$, we conclude that $x_0\in \Theta$. Since $\R^d\setminus U$ is closed, we conclude that $x_0\in \R^d\setminus U$, which contradicts $\Theta\subset U$.
By (iii) and (iv), the functions $f_{1},\dots,f_{n}$ and $\sigma$ are continuous on the compact set $\cl(\Theta^{\varepsilon/2}) \subset \Theta^{\varepsilon}$ and therefore bounded on $\cl(\Theta^{\varepsilon/2})$. This jointly with (i) yields (A)(iv).

		Finally, we prove that (A)(v) is satisfied.
	Note that for all $i \in \{1,\dots,n\} $, all $x\in K_{i}$ and $y \in \bigcup_{j\not=i} K_{j}$ and every function $\gamma \colon [0,1] \to \R^{d}\setminus \Theta$ with $\gamma(0) = x$ and $\gamma(1) = y$, the set  $\gamma([0,1]) $ is not connected. Hence, $\gamma$ is not continuous. Thus, for all  $x,y \in \R^{d}\setminus \Theta$ with $\rho_{\R^{d} \setminus \Theta}(x,y) < \infty$ we have $x,y \in K_i$ for some $i \in \{1,\dots,n \}$. Therefore, the piecewise Lipschitz continuity of $\mu$ follows from the Lipschitz continuity of $f_{i}$ on $K_{i}$ for $i \in \{1,\dots,n \}$.

\section{Numerical results}\label{Num}
In this section we present numerical simulations for the performance
of the $L_p$-error  
\[
\varepsilon_{p,n}=\bigl(\EE\bigl[\|X_1-\eul_{n,1}\|^p\bigr]\bigr)^{1/p}
\]
of the Euler-Maruyama scheme $\eul_n$ with step-size $1/n$ at the final time point $1$.
We use $\eul_{N,1}$
with $N$ large as a reference estimate of $X_1$ and we approximate the $L_p$-error $\varepsilon_{p,n}$
by the 
corresponding
empirical $p$-th mean 
error
\begin{equation}\label{emperr}
	\widehat \varepsilon_{p,n}= \Bigl(\frac{1}{m} \sum_{i = 1}^{m} \|\eul_{N,1}^{i} - \eul_{n,1}^{i}\|^{p} \Bigr)^{1/p}
\end{equation}
based on $m$ 
Monte Carlo
repetitions
 $(\eul_{N,1}^{1}, \eul_{n,1}^{1}), \ldots, (\eul_{N,1}^{m}, \eul_{n,1}^{m})$ of 
 $(\eul_{N,1}, \eul_{n,1})$.

In the following examples we choose $n = 2^7,2^8,\dots,2^{15}$, $N = 2^{17}$  and $m = 10^6$
unless otherwise 
stated.

\begin{exs}\label{example1}

We  consider a $2$-dimensional  SDE \eqref{SDE} with initial value $x_0=(0,2)^\top$ and coefficients $\mu \colon \R^{2} \to \R^{2}$ and $\sigma\colon \R^{2} \to \R^{2 \times 2}$ given by
\begin{equation}\label{ex1}
	\mu(x)=\begin{cases}
		\begin{pmatrix}
				1 \\
				1 
		\end{pmatrix}-x, & \text { if } \norm{x} < 2 \\
		-x, & \text{ if } \norm{x} \geq 2
	\end{cases}, \qquad\qquad
\sigma(x)= \phi\bigl(\norm{x} - 2\bigr) 
\begin{pmatrix}
	1 & 0 \\
	0 & 1 
\end{pmatrix},
	\end{equation}
where $\phi \colon \R \to \R$ is defined by \eqref{fctphi}. Thus, the drift coefficient $\mu$ is discontinuous 
on the set
$\Theta = \{x \in \R^{2} \mid \norm{x} = 2 \}$ and the diffusion coefficient $\sigma$ vanishes outside  the set $\{x\in\R^d\mid 1<\|x\|<3\}$.

We show that $\mu$ and $\sigma$ satisfy the conditions (A) and (B). 
Note
that $\mu$ is of the form \eqref{Exm1} with 
$n=2$, 
 $K_1$ and $K_2$ given by
\[
	K_{1} = \{x \in \R^{2} \mid \norm{x} < 2 \},\qquad
	K_{2} = \{x \in \R^{2} \mid \norm{x} > 2 \}
\]
and  
$f_0, f_1, f_2\colon\R^2\to\R^2$ given by
\begin{align*}
		f_{0}(x) =f_2(x)=-x, \qquad 
	f_{1} (x)= \begin{pmatrix}
		1 \\
		1
	\end{pmatrix}-x.
\end{align*}
Clearly,  $\Theta$ is a compact $C^{\infty}$-hypersurface,  $K_{1}$ and $K_{2}$ are open and disjoint sets satisfying \eqref{Exm2} and the functions $f_0,f_1$ and $f_2$ fulfill the conditions (i)--(iii) 
from
Section \ref{Exmpl}. Moreover, since  $\phi$ is $C^{3}$, see Lemma \ref{diff2}, and
 the Euclidean norm $\|\cdot\|$ is $C^{\infty}$ on $\R^{d}\setminus \{0\}$, we conclude that $\sigma$ is  $C^{3}$
and hence the condition (v) from Section \ref{Exmpl} holds. Since $\sigma$ has compact support, the condition (iv) from Section \ref{Exmpl} holds as well. 
Finally, 
  $\nor \colon \Theta \to \R^{2}, \, x \mapsto \frac{1}{2}x$, is a normal vector along $\Theta$ and for all $x \in \Theta$ we have
\[
\nor(x)^\top\sigma(x) = 	\frac{1}{2} \phi(0)x^\top\begin{pmatrix}
	1 & 0 \\
	0 & 1 
\end{pmatrix}  = \frac{1}{2} x^\top \neq 0,
\]
 which proves 
 that  the condition (vi) from Section \ref{Exmpl} is satisfied as well.

Figure \ref{fig1} shows, on a double logarithmic scale, the plot of 
 a realization of
 the empirical $L_2$-error $\widehat \varepsilon_{2,n}$  of the Euler-Maruyama scheme $\eul_{n,1}$ versus the number of time-steps $n$. Additionally, the
 resulting
 least-squares regression line and a line with the slope $-0.5$ are plotted. The empirical $L_2$-error rate of the Euler-Maruyama scheme is  $0.52$. \\

\begin{figure}[H]
	\centering
	\caption{ Empirical $L_2$-error vs. number of time steps}
	\makebox[\textwidth]{\includegraphics[scale=0.23]{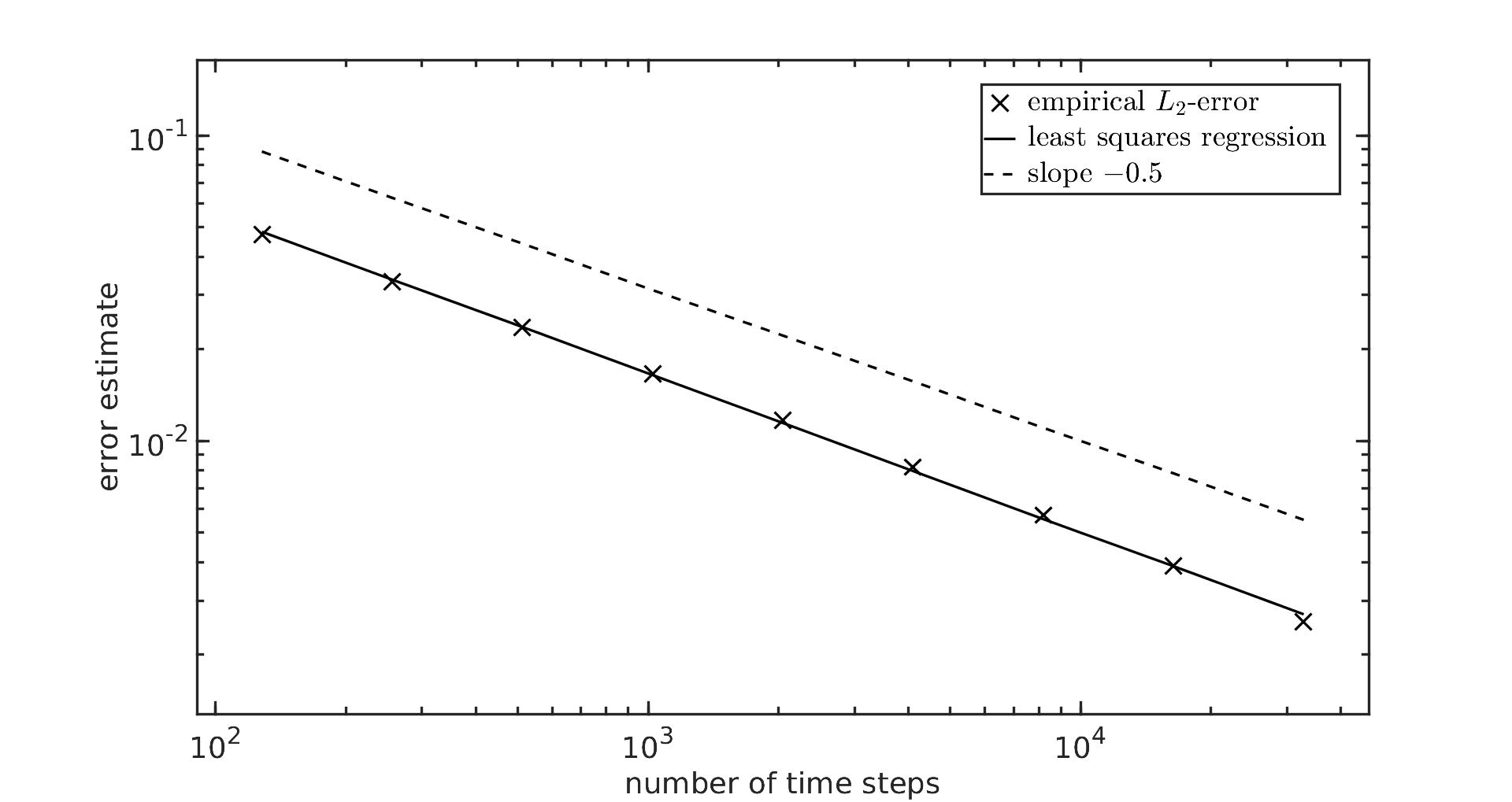}}
	\label{fig1}
\end{figure}

We have furthermore computed empirical $L_p$-error rates for $p=1$, $p=4$ and $p=8$, see Table \ref{tab1}. 
The empirical $L_p$-error rate 
slightly decreases
with increasing $p$, but remains close to $0.5$,
which  provides some numerical evidence for the theoretical finding in Theorem~\ref{Thm1}
that the $L_p$-error of the Euler-Maruyama scheme $\eul_n$ at the final time point $1$ is at least 
$0.5-$.
However, the following example shows that it can also be difficult to provide numerical evidence.

\begin{table}[H]
\caption {Empirical $L_p$-error rates} \label{tab1} 
\begin{center}
\begin{tabular}{|c|c|c|c|c|}
\hline
	$p$ & 1 & 2 & 4 & 8 \\
	\hline 
 & 0.52 & 0.52 & 0.50 &  0.47   \\
	\hline
\end{tabular}
\end{center}
\end{table}

\end{exs}

\begin{exs}\label{example2}
We 
consider  a $2$-dimensional  SDE \eqref{SDE} with initial value $x_0\in\R^2$, drift coefficient  $\mu \colon \R^{2} \to \R^{2}$ given by
	\[
	\mu(x)= \begin{cases}
		\begin{pmatrix}
			a \\
			a \\ 
		\end{pmatrix}, & \text { if } \norm{x} < 2 \\[.6cm]
		\begin{pmatrix}
			b \\
			b \\
		\end{pmatrix} \norm{x}, & \text{ if } \norm{x} \geq 2 
	\end{cases},
	\]
where $a,b \in \R$, and diffusion coefficient  $\sigma$ as in \eqref{ex1}.
As in Example \ref{example1},  the set of points of discontinuity of $\mu$ is given by the 
circular line
$\Theta = \{x \in \R^{2} \mid \norm{x} = 2 \}$
and,
similarly to  Example \ref{example1},
it is easy to see
 that $\mu$ and $\sigma$ satisfy the conditions (A) and (B).

We first choose $a = -3$, $b = 1$ and $x_{0} = (0,2)^\top$, i.e. the SDE \eqref{SDE} starts at time $0$ at a point of discontinuity  of $\mu$. Figure \ref{fig2} shows,
on a double logarithmic scale,
 the plot of
a realization of
 the  empirical $L_2$-error $\widehat \varepsilon_{2,n}$  of the Euler-Maruyama scheme  versus the number of time-steps $n$. The empirical $L_2$-error rate  is  $0.39$ in this case, which is significantly smaller than 
 $0.5$.

\begin{figure}[H]
	\centering
	\caption{Empirical $L_2$-error vs. number of time steps: $a = -3, b = 1, x_{0} = (0,2)^\top$}
	\makebox[\textwidth]{\includegraphics[scale=0.23]{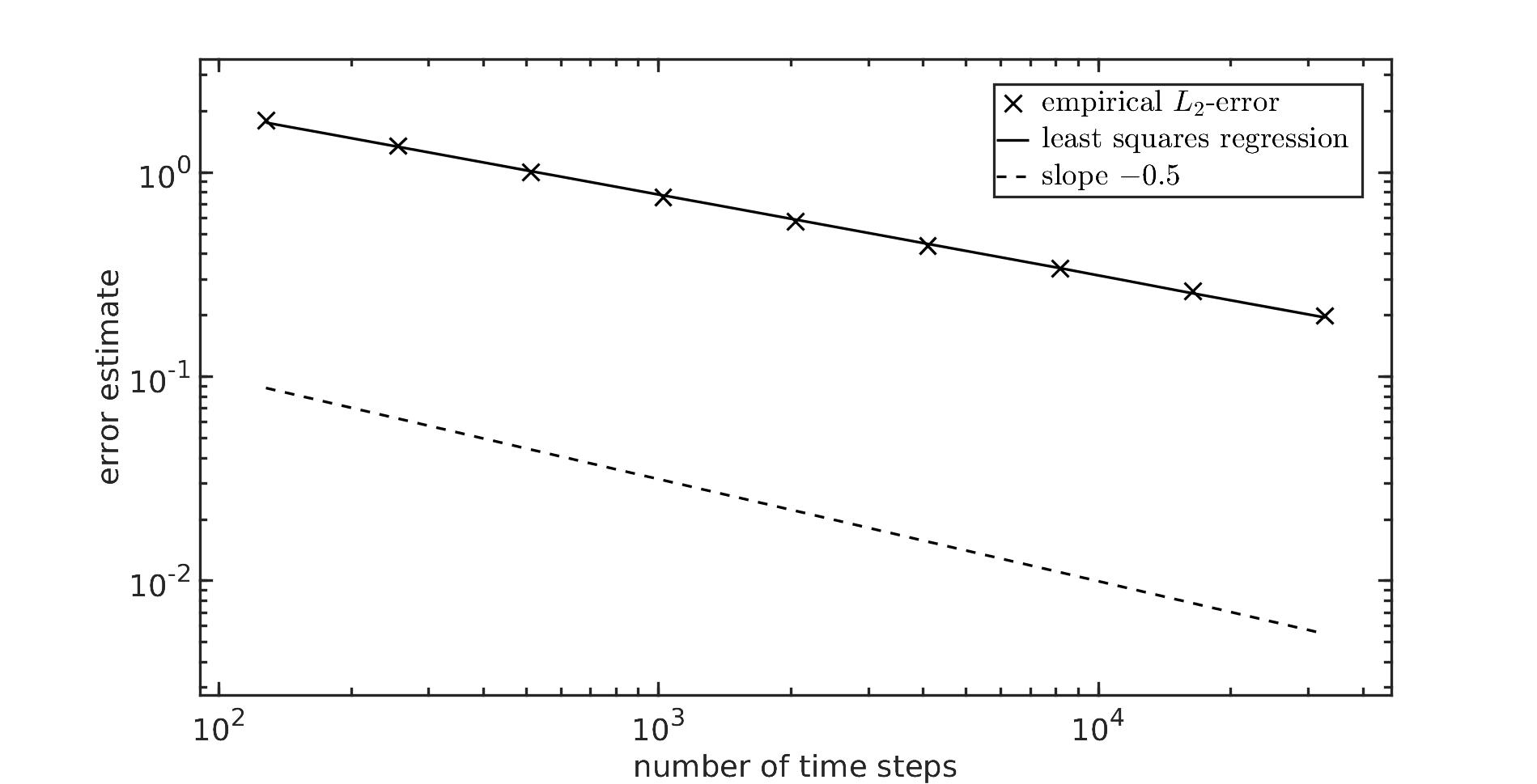}}
	\label{fig2}
\end{figure}
The phenomenon that the empirical $L_2$-error rate of the Euler-Maruyama scheme is significantly smaller than the theoretical one 
 has already been observed in \cite{GLN17} for  SDEs with a discontinuous drift coefficient in the case of $d=1$ and additive noise. More precisely, in the latter setting,
the $L_2$-error rate of the Euler-Maruyama scheme is known to be at least
$0.75-$
if the drift coefficient $\mu$ 
has finitely many jumps at points $x_1<\ldots<x_K$, is bounded and has bounded first and second derivatives on each of the intervals $(x_k, x_{k+1})$, $k=0, \ldots, K$, where $x_0=-\infty$ and $x_{K+1}=\infty$,
see \cite {NS19}.
However, for $\mu=10\sgn$ and $x_0=0$, 
an
empirical $L_2$-error rate  of $0.25$ was observed in 
\cite{GLN17} and for $\mu=-3 \one_{(-\infty, -1.4)}+4 \one_{[1.4, \infty)}$ and $x_0\in\{1, 1.2, 1.25, 1.4, 2\}$,  empirical $L_2$-error rates between  $0.31$ and $0.4$ were observed in 
\cite{GLN17}. Moreover, 
an
empirical $L_2$-error rate 
 significantly smaller than $0.5$ was observed in \cite{LS18} for the Euler-Maruyama scheme for the $2$-dimensional  SDE \eqref{SDE} with drift coefficient $\mu \colon \R^{2} \to \R^{2}$, 
 $(x_1,x_2)^\top \mapsto (3(\one_{[0,\infty)}(x_1)-\one_{(-\infty,0)}(x_1)),1)^{\top}$,
 and diffusion coefficient $\sigma=\text{id}_{\R^2}$.

Similar to Example 1, for the SDE currently under consideration,  the empirical $L_p$-error rate decreases 
with increasing $p$, however, the decay is much stronger than for the SDE studied in Example 1.
We observe an empirical    $L_1$-error rate of $0.67$, see Figure \ref{fig3}, which is  
consistent with 
the
theoretical $L_1$-error rate of at least 
$0.5-$,
 while for $p=4$ and $p=8$ we observe the rates $0.21$ and  $0.12$, respectively, see Table \ref{tab2}.
\begin{figure}[H]
	\centering
	\caption{Empirical $L_1$-error vs. number of time steps: $a = -3, b = 1, x_{0} = (0,2)^\top$}
	\makebox[\textwidth]{\includegraphics[scale=0.23]{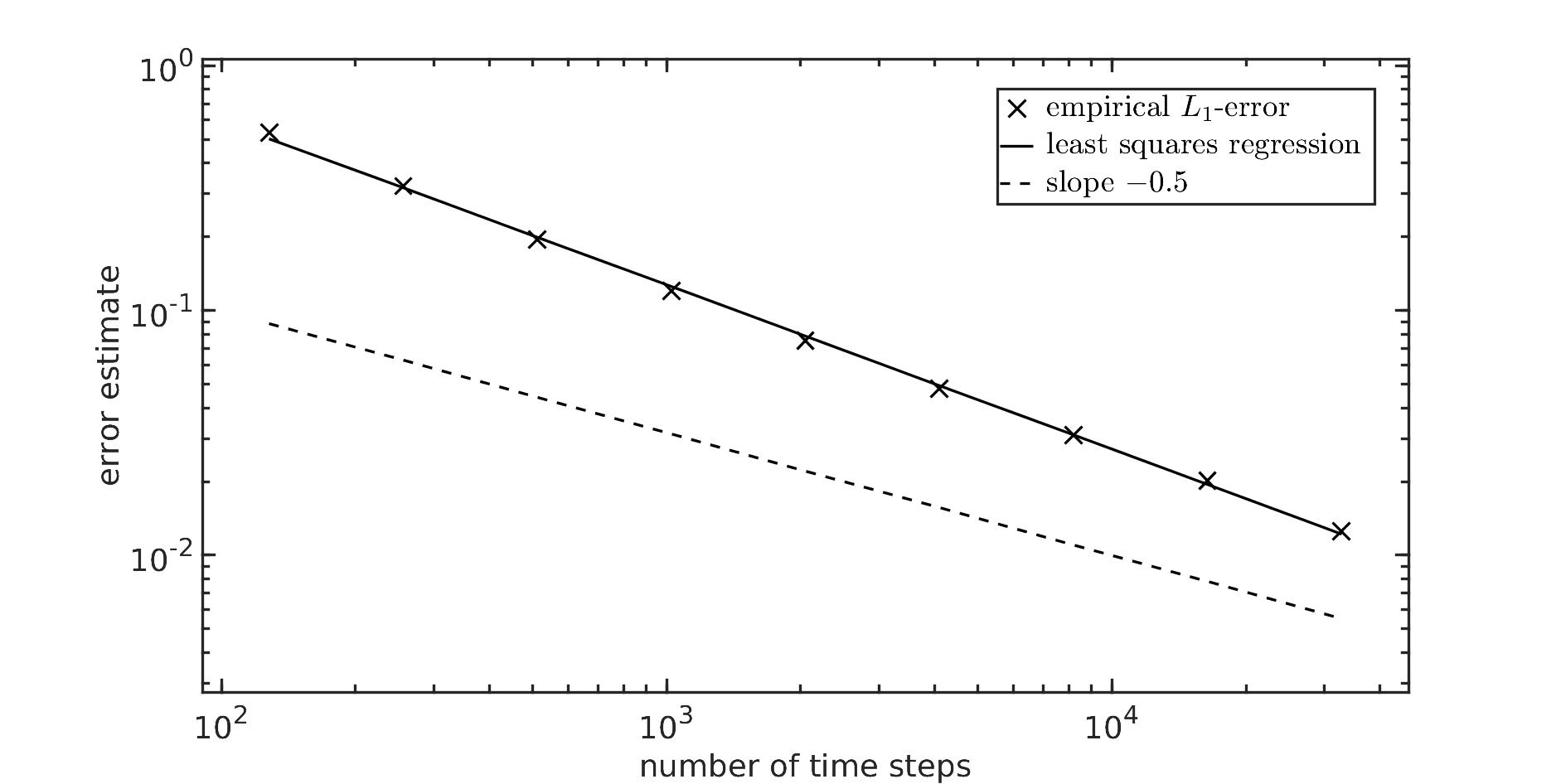}}
	\label{fig3}
\end{figure}

\begin{table}[H]
\caption {Empirical $L_p$-error rates: $a = -3, b = 1, x_{0} = (0,2)^\top$} \label{tab2} 
\begin{center}
\begin{tabular}{|c|c|c|c|c|}
\hline
	$p$ & 1 & 2 & 4 & 8 \\
	\hline
	 &0.67 & 0.39 & 0.21 &  0.12   \\
	\hline
\end{tabular}
\end{center}
\end{table}

  To  exclude a possible negative influence of approximating $X_1$ by the reference estimate $\eul_{N,1}$  on the empirical $L_p$-error rate, 
  we also study, for the current SDE,
  the performance of the $L_p$-norm
  \[
  \delta_{p,n}=\bigl(\EE\bigl[\|\eul_{2n,1}-\eul_{n,1}\|^p\bigr]\bigr)^{1/p}
  \]
  of the difference   of the Euler-Maruyama schemes with step-sizes $1/2n$ and $1/n$ at the final time point $1$. It follows from Theorem~\ref{Thm1} that for every $p\geq 1$ and every $\delta>0$ there exists $c>0$ such that for all $n\in\N$,
  \begin{equation}\label{eulest}
	\delta_{p,n}\leq\frac{c}{n^{1/2-\delta}},
	\end{equation}
  i.e. the sequence $(\delta_{p,n})_{n\in\N}$ converges to $0$ with a rate of at least $0.5-$.
 
  We approximate $\delta_{p,n}$
  by the 
  corresponding
  empirical 
  $p$-th mean  norm
  \[
  	\widehat \delta_{p,n}= \Bigl(\frac{1}{m} \sum_{i = 1}^{m} \|\eul_{2n,1}^{i} - \eul_{n,1}^{i}\|^{p} \Bigr)^{1/p}
  \]
  based on $m$ 
  Monte Carlo
  repetitions
  $(\eul_{2n,1}^{1}, \eul_{n,1}^{1}), \ldots, (\eul_{2n,1}^{m}, \eul_{n,1}^{m})$ of 
  $(\eul_{2n,1}, \eul_{n,1})$. 
  Table \ref{tab6} presents for $p=1$, $p=2$, $p=4$ and $p=8$ an
  estimated rate of convergence of $(\delta_{p,n})_{n\in\N}$ based on
   a realization of the empirical $p$-th mean  norms 	$\widehat\delta_{p,n}$,
   $n=2^7, \ldots, 2^{14}$, with $m=10^6$ Monte Carlo repetitions.
  The observed empirical rates for the $L_p$-norms of the differences of the Euler-Maruyama schemes  do not differ significantly from the respective empirical $L_p$-error rates of the Euler-Maruyama scheme in Table \ref{tab2}. We again observe the phenomenon, that the rates decrease rapidly  with increasing $p$. Even for $m=10^8$ Monte Carlo repetitions, we obtain the same empirical rates as in Table \ref{tab6}.

\begin{table}[H]
	\caption {Empirical rates for $\delta_{p,n}$:
		$a = -3, b = 1, x_{0} = (0,2)^\top$} \label{tab6} 
	\begin{center}
		\begin{tabular}{|c|c|c|c|c|}
			\hline
			$p$ & 1 & 2 & 4 & 8 \\
			\hline
		&0.64 & 0.37& 0.20 &  0.11   \\
			\hline
		\end{tabular}
	\end{center}
\end{table}

 To 
 study this phenomenon more closely,
 we consider the 
 $p$-th 
 power
 \[
  d_{p,n}=\bigl(n^{0.45}\cdot\norm{\eul_{2n,1}- \eul_{n,1}}\bigr)^p
 \]
 of the  difference of the Euler-Maruyama schemes with step-sizes $1/2n$ and $1/n$ at the final time point $1$, scaled by $n^{0.45}$.
 Note that \eqref{eulest} yields that for every $p\geq 1$  and every $\delta>0$ there exists $c>0$ such that for all $n\in\N$,
 \[
 	\EE[d_{p,n}]\leq\frac{c}{n^{0.05p-\delta}}.
 \]
In particular,
 \[
 \lim_{n\to\infty}\EE[d_{p,n}]=0,
 \]
 i.e., the sequence $(d_{p,n})_{n\in\N}$ convergece to $0$ in $L_1$. Hence, it also converges to $0$ in probability.

  Figures \ref{fig1b} and \ref{fig2b} show, on a double logarithmic scale, 
  histograms of $m=10^8$ realizations of $ d_{1,n}$ and $ d_{2,n}$, respectively,  with $n=2^4, \ldots, 2^{14}$ (from top to bottom). The length of the bins on the $x$-axis is $10^{-1}$ and the height of a block equals the relative frequency of the corresponding bin.
Red dots represent the empirical means of the data and blue dots represent the $0,99$-quantiles of the relative frequencies.

The histograms point to an outlier problem.  While the proportion of outliers decreases with increasing $n$ (the $0,99$-quantiles decrease), 
the size of the outliers increases.
See also Figure \ref{fig3b}  for the right  tails of the histograms of realizations of $ d_{2,n}$  cut off 
at the value 
$10^1$. In the latter case,  the length of the bins on the $x$-axis is $10^{1}$.

  	 For $p=1$, the outliers 
  	 do not seem to strongly influence the size of the empirical means 
  	 -- the empirical means
  	 of realizations of 
  	 $d_{1,n}$
  	 decrease with increasing $n$
  	 as expected,
  	 see Figure \ref{fig1b}. However,  for  $p=2$ 
  	 a
  	 significant influence is visible -- the empirical means increase with increasing $n$ and for $n=2^{11}, \ldots,  2^{14}$ they even become larger than the  respective $0,99$-quantiles of the relative frequencies,  see Figure \ref{fig2b} and Figure \ref{fig3b}. For $p=4$ the influence of outliers becomes enormous -- the empirical means explode with increasing $n$ and for $n=2^{11}, \ldots,  2^{14}$ they  become much larger than the  respective $0,99$-quantiles of the relative frequencies, see Figure \ref{fig4b}, which shows  the right  tails of the histograms of $m=10^8$ realizations of $ d_{4,n}$  cut off 
  	 at the value
  	 $10^3$. In the latter case,  the length of the bins on the $x$-axis is 
  	 $10^{3}$. 
 
Thus,  the reason for the phenomenon that the observed empirical rates for $\widehat\delta_{p,n}$ decrease rapidly with increasing $p$ appears to lie in 
an
outlier problem. 
The underlying reason for the outlier problem (e.g. whether  $m$ is too small or whether $n$ is to small) 
is hard to investigate, due to limitations of computing power, and so far remains unclear to us.

 \begin{figure}[H]
 	\centering
 	\caption{Histograms of realizations of $ d_{1,n}$} 
 	\makebox[\textwidth]{\includegraphics[scale=0.52]{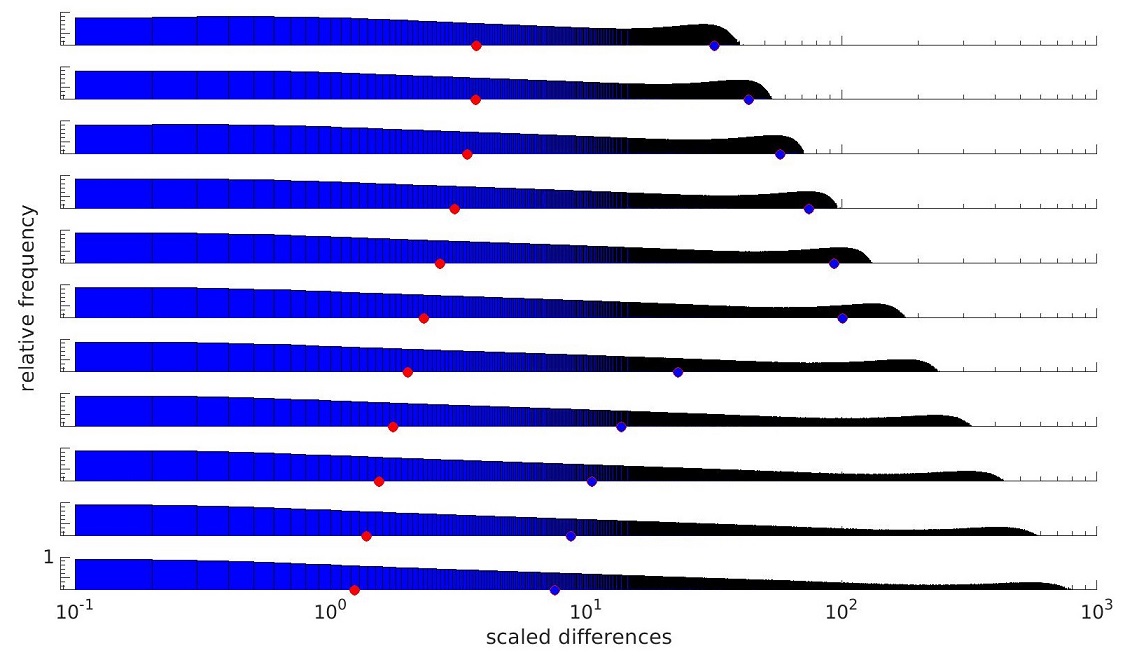}}
 	\label{fig1b}
 \end{figure}
 
 \begin{figure}[H]
 	\centering
 	\caption{Histograms of realizations of $ d_{2,n}$} 
 	\makebox[\textwidth]{\includegraphics[scale=0.52]{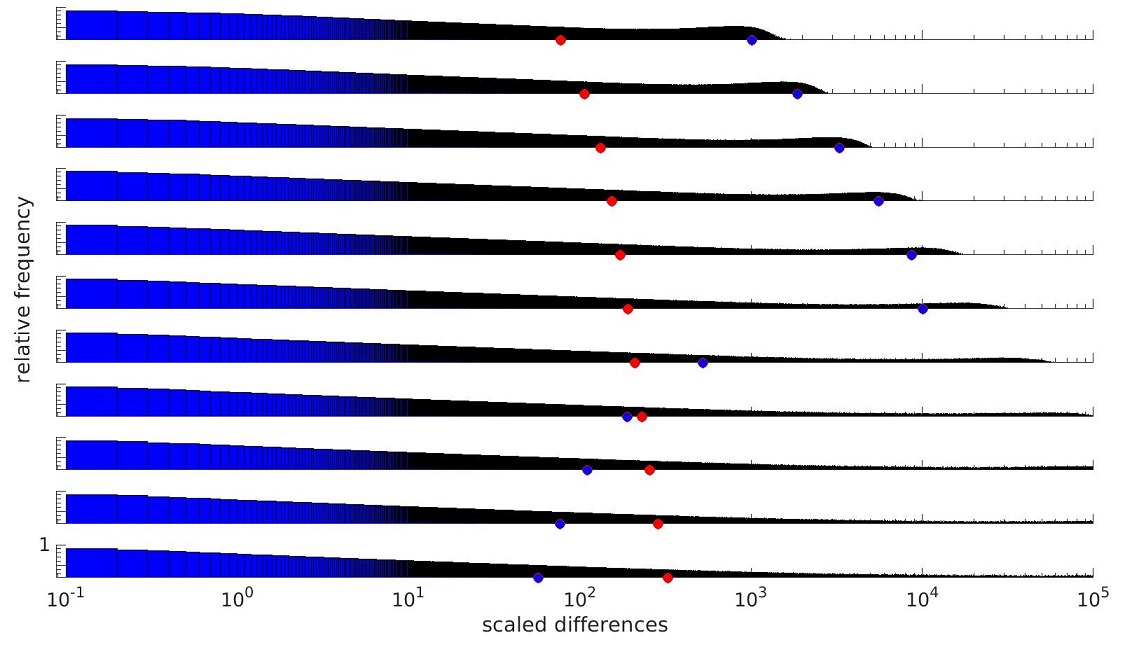}}
 		\label{fig2b}
 \end{figure}
 \vspace{-5cm}
 
 \begin{figure}[H]
 	\centering
 	\caption{Tails of histograms of realizations of $ d_{2,n}$}
 	\makebox[\textwidth]{\includegraphics[scale=0.52]{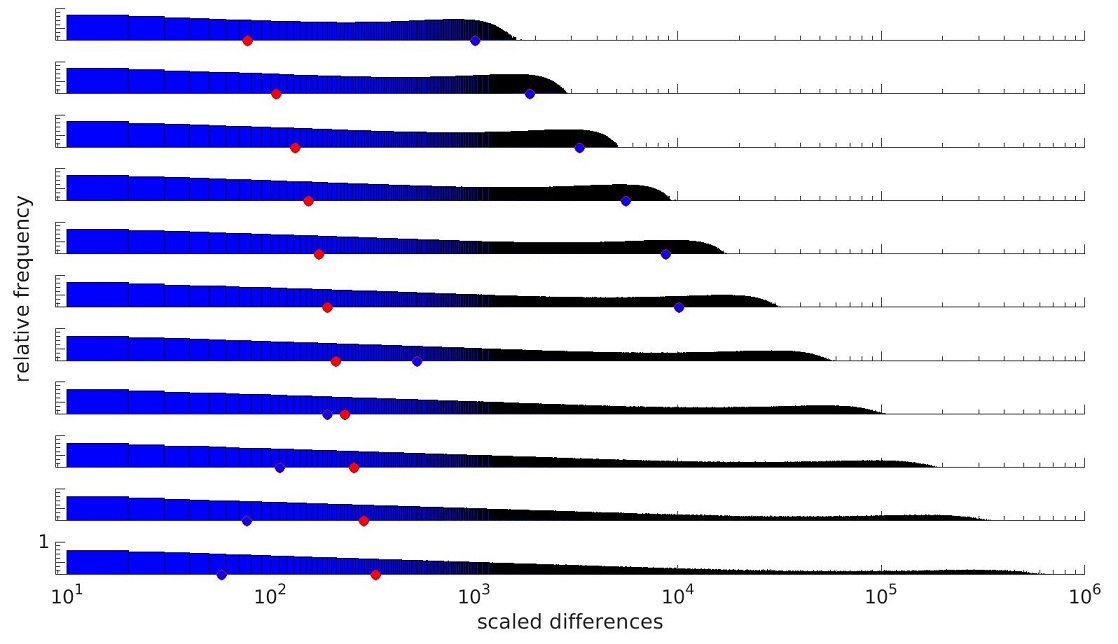}}
 	\label{fig3b}
 \end{figure}
 
 \begin{figure}[H]
 	\centering
 	\caption{Tails of histograms of realizations of $ d_{4,n}$}
 	\makebox[\textwidth]{\includegraphics[scale=0.52]{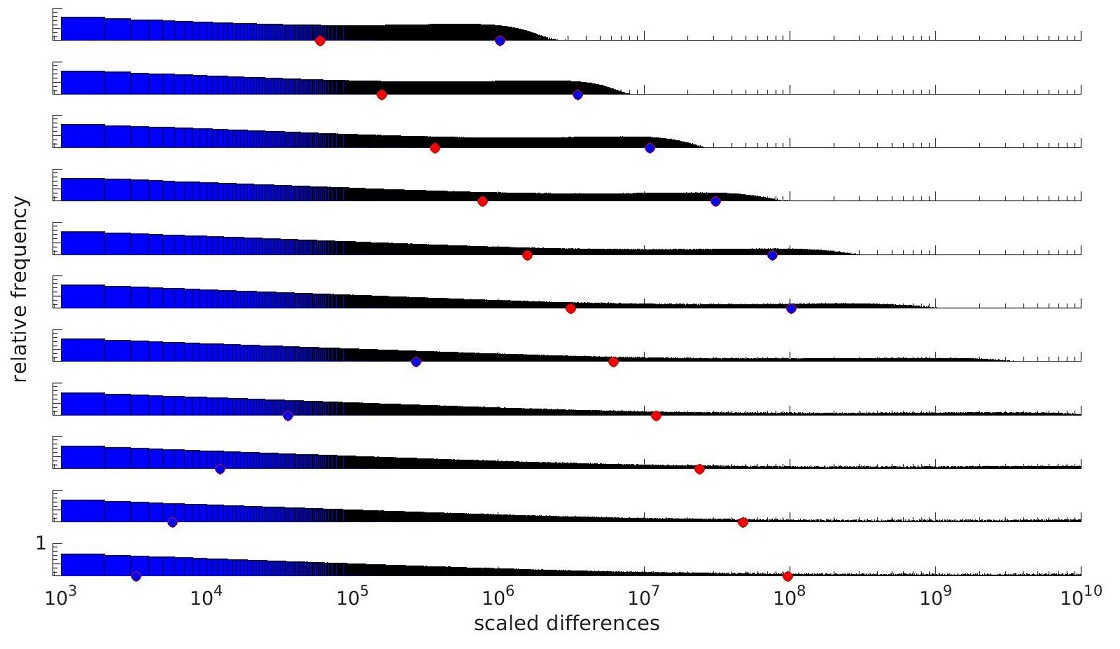}}
 	\label{fig4b}
 \end{figure}
 
For comparison, we present in Figures \ref{fig5b}, \ref {fig6b} and \ref{fig7b} histograms of  $m=10^6$ realizations of the random variables $ d_{1,n}$, $ d_{2,n}$ and  $d_{4,n}$ respectively,  for the SDE from Example \ref{example1} with $n=2^7, \ldots, 2^{14}$ (from top to bottom). The length of the bins on the $x$-axis is $10^{-3}$. 
As before, red dots represent the empirical means of the data and blue dots represent the $0,99$-quantiles of the relative frequencies.

In contrast to the histograms for the SDE from Example \ref{example2}, there are no outliers in the histograms for the SDE from Example \ref{example1}.  For  $p=1$,  $p=2$ and $p=4$, the empirical means decrease with increasing $n$ and do not exceed the   respective $0,99$-quantiles of the relative frequencies.

 \begin{figure}[H]
 	\centering
 	\caption{Histograms of realizations of $ d_{1,n}$: Example 1}
 	\makebox[\textwidth]{\includegraphics[scale=0.52]{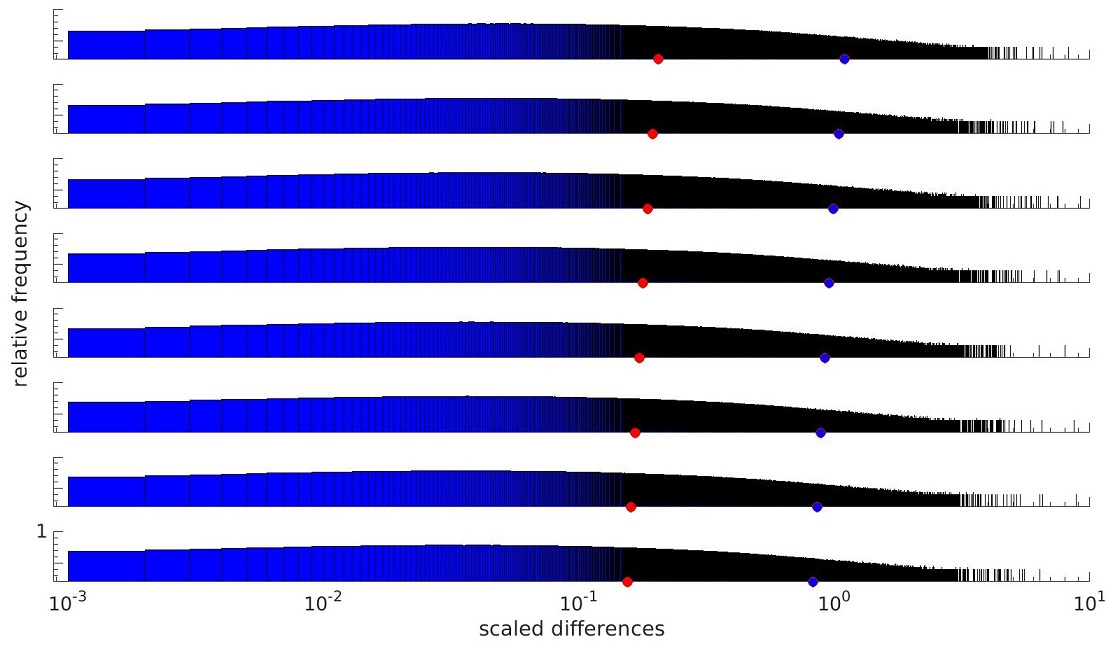}}
 	\label{fig5b}
 \end{figure}
 
 \vspace{-5cm}
 
 \begin{figure}[H]
 	\centering
 	\caption{Histograms of realizations of $ d_{2,n}$: Example 1}
 	\makebox[\textwidth]{\includegraphics[scale=0.52]{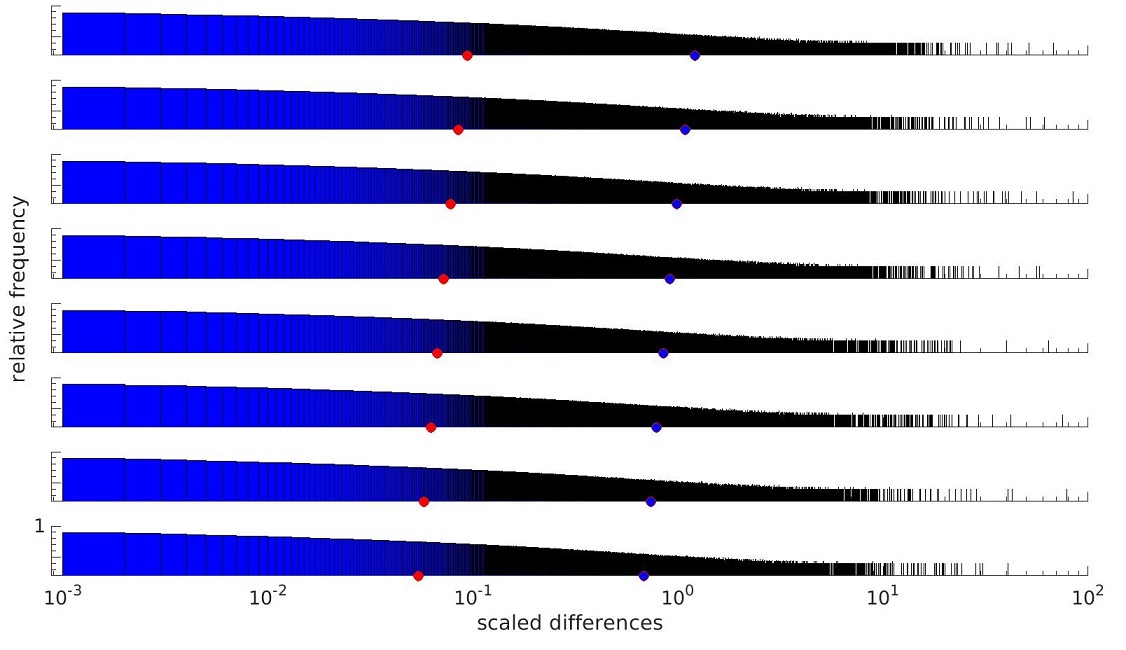}}
 	\label{fig6b}
 \end{figure}

 \begin{figure}[H]
 	\centering
 	\caption{Histograms of realizations of $ d_{4,n}$: Example 1}
 	\makebox[\textwidth]{\includegraphics[scale=0.52]{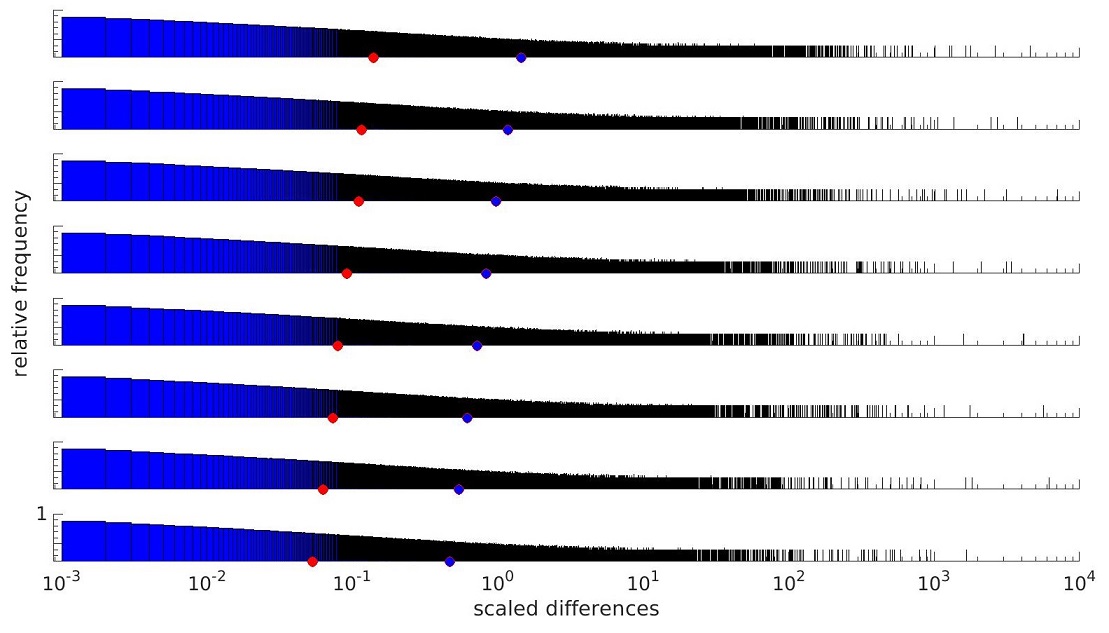}}
 	\label{fig7b}
 \end{figure}

Finally, we present results of numerical experiments for different choices of the parameters $a$ and $b$ 
and
the initial value $x_0$.
Table \ref{tab4} presents empirical $L_p$-error rates
of the Euler-Maruyama scheme
 for  $a = -3, b = 1$ and  $x_{0} \in\bigl{\{}  (0,2.5)^\top, (0,0)^\top\bigr{\}}$. Similar to the case $x_{0} = (0,2)^\top$, see Table \ref{tab2},  we observe a strong decay of the  empirical $L_p$-error rate with increasing $p$ in the case $x_{0} = (0,2.5)^\top$.
In the case $x_{0} = (0,0)^\top$, the empirical $L_p$-error rate decreases slightly with increasing $p$ and remains greater than $0.5$. Note that in the latter case, the distance from the initial value $x_0$ to the set of 
the
points of discontinuity of $\mu$ is significantly larger than in the case $x_{0} = (0,2.5)^\top$.

\begin{table}[H]
\caption {Empirical $L_p$-error rates: $a = -3, b = 1$} \label{tab4} 
\begin{center}
\begin{tabular}{|c|c|c|c|c|}
\hline
	$p$ & 1 & 2 & 4 & 8 \\
	\hline 
	 $x_{0} = (0,2.5)^{\top^{ \phantom{A}}}$& 0.75 & 0.38 & 0.17 &  0.11   \\
	\hline
	 $x_{0} = (0,0)^{\top^{ \phantom{A}}}$& 0.55 & 0.55 & 0.54 &  0.54   \\
	\hline 
\end{tabular}
\end{center}
\end{table}
Table \ref{tab5} presents empirical $L_p$-error rates 
of the Euler-Maruyama scheme
for  $(a,b)\in\{(3,-1),$ $(1,1), (-0.1, 0.1)\}$ and  $x_{0} = (0,2)^\top$.  In each of the three cases, the empirical $L_1$-error rate is consistent with the theoretical $L_1$-error rate of at least 
$0.5-$ 
and the empirical $L_p$-error rate decreases with increasing $p$. However, the decay becomes slower as the 
size $\sqrt{2}|a-2b|$ of the jump of the drift coefficient at the set of the discontinuity points $\Theta$
decreases.

\begin{table}[H]
\caption {Empirical $L_p$-error rates: $x_{0} = (0,2)^\top$} \label{tab5} 
\begin{center}
\begin{tabular}{|c|c|c|c|c|}
\hline
	$p$ & 1 & 2 & 4 & 8 \\
	\hline
	$a=3, b=-1$ & 0.53 & 0.43 & 0.22 &  0.12   \\
	\hline
	$a=1, b=1$ & 0.53 & 0.47 & 0.30 &  0.20   \\
	\hline
	$a=-0.1, b=0.1$ & 0.52 & 0.50 & 0.44 &  0.35  \\
	\hline
\end{tabular}
\end{center}
\end{table}

\end{exs}
\section{Appendix - Basic facts on hypersurfaces}\label{Appendix}

In this section, we present a number of known basic facts on hypersurfaces, distance functions, normal vectors, projections and intrinsic Lipschitz continuity that are mainly used for the proofs in Section~\ref{Proofs}. For the convenience of the reader, we provide a proof of a statement 
if we were not able to find a corresponding, 
directly applicable
reference in the literature.

\begin{Lem}\label{leb0} 
	Let $d\in\N$ and let $\emptyset\not=M \subseteq \R^{d}$ be a $C^{1}$-hypersurface. Then $M$ is 
	a Borel set 
	with $\lambda_d(M) = 0$.
\end{Lem}
\begin{proof}
		For $x \in M$ let $(\phi_{x},U_{x})$ be a $C^{1}$-chart for $M$
	at $x$, i.e. $\phi_x(U_{x})\subset \R^d$
	is open, $\phi_{x}\colon U_x \to \phi_x(U_{x}) $ is a $C^1$-diffeomorphism and $\phi_x(M \cap U_{x}) = \R_{0}^{d-1}\cap\phi_x(U_{x})$. Since 
	$\R_{0}^{d-1}$ is a Borel set we obtain that $M \cap U_{x}=\phi_{x}^{-1}(\R_{0}^{d-1}\cap\phi_x(U_{x}) )$ is a Borel set. Moreover, by the change of variables formula,
	\[
	\lambda_d(\phi_{x}^{-1}(\R_{0}^{d-1}\cap\phi_x(U_{x}) )) = \int_{\R_{0}^{d-1}\cap\phi_x(U_{x})} \bigl|\text{det}\bigl( (\phi_x^{-1})'(y)\bigr)\bigr|\, \lambda_d(dy),
	\]
	which yields $\lambda_d(\phi_{x}^{-1}(\R_{0}^{d-1}\cap\phi_x(U_{x}) ))= 0$ because $\lambda_d(\R_{0}^{d-1})=0$. 
	
	Finally note that $\R^d$ has a countable basis and therefore every subspace of $\R^d$ is a Lindelöf-space. Since $M = \bigcup_{x \in M} (M \cap U_{x})$ we conclude that there exists a countable subset $\widetilde{M} \subset M$ such that $M=  \bigcup_{x \in \widetilde{M}} \phi_{x}^{-1}(\R_{0}^{d-1}\cap\phi_x(U_{x}) ) $, which completes the proof of the lemma. 
\end{proof}

\begin{Lem} \label{connect0} 

	Let  $d\in\N$, let 
	$\emptyset\not=M \subset \R^{d}$ be a $C^{0}$-hypersurface, let $x \in M$ and let $A\subset \R^d$ be open with $x\in A$. Then there exists an open set $\widetilde A\subset A$ with $x\in \widetilde A$ such that $\widetilde A\cap M$ is connected. 
\end{Lem}
\begin{proof}
	Choose any chart $(\phi, U)$ for $M$ at $x$. Since $A$ is open, we may assume $U \subset A$.
Put $V=\phi(U)$. 
Then $\phi(M\cap U) = \R_{0}^{d-1} \cap V$.
Since $V$ is open, there exists $\varepsilon > 0$ such that $B_{\varepsilon}(\phi(x)) \subset V$. The set $\R_{0}^{d-1} \cap B_{\varepsilon}(\phi(x))$ is convex and therefore connected. By the continuity of $\phi^{-1}$ we conclude that $\phi^{-1}(\R_{0}^{d-1} \cap B_{\varepsilon}(\phi(x)))$ is also connected. 
Put $\widetilde A= \phi^{-1}(B_{\varepsilon}(\phi(x)))$. Clearly, $x\in \widetilde A\subset U\subset A$, and since  $B_{\varepsilon}(\phi(x))$ is open and 
$\phi$ 
is continuous we get that $\widetilde A$ is open. 
Since $\phi^{-1}$ is an injection we furthermore have
	\begin{align*}
		M\cap \widetilde A & =(M\cap U)\cap \phi^{-1}(B_{\varepsilon}(\phi(x))) = \phi^{-1}(\R_{0}^{d-1} \cap V) \cap \phi^{-1}(B_{\varepsilon}(\phi(x))) \\
		& =  \phi^{-1}(\R_{0}^{d-1} \cap V\cap B_{\varepsilon}(\phi(x))) = \phi^{-1}(\R_{0}^{d-1} \cap B_{\varepsilon}(\phi(x))) ,
	\end{align*}
	which completes the proof of the lemma.
\end{proof}

\begin{Lem} \label{tangkern} 
	Let  $d\in \N$, let 
	$\emptyset\not=M \subset \R^{d}$ be a $C^{1}$-hypersurface and let $x \in M$. Let $U\subset \R^{d}$ be open with $x\in U$ and let $f\colon U \to \R$ be a 
$C^1$-function with $M \cap U \subset f^{-1}(\{0 \})$ and  $f'(x) \neq 0$. Then 
	\[
	T_{x}(M) = \text{Ker}(f'(x)).
	\]
\end{Lem}
\begin{proof}
	First we show $T_{x}(M) \subset \text{Ker}(f'(x))$. Let $v \in T_{x}(M)$. Then there exist $\varepsilon > 0$ and a $C^1$-mapping $\gamma\colon (-\varepsilon, \varepsilon) \to M$ such that $\gamma(0) = x$ and $\gamma'(0) = v$. Since
	$\gamma$ is continuous,
	 $U$ is open and $x\in U$ we may assume that $\gamma((-\varepsilon,\varepsilon))\subset U\cap M$. Since $f(M \cap U) = \{0 \}$ we have $f \circ \gamma = 0$. Thus $0 = (f \circ \gamma)'(0) = f'(\gamma(0)) \gamma'(0) = f'(x) v$, which implies $v\in \text{Ker}(f'(x))$.
	
	Since $M$ is a $C^{1}$-hypersurface, 
	$T_{x}(M)$ is known to be a $(d-1)$-dimensional vector space. 
Furthermore, $f'(x)\neq 0$ implies that 
$\dim(\text{Ker}(f'(x))) = d-1$. 
The latter two facts and $T_{x}(M) \subset \text{Ker}(f'(x))$ imply that
	$T_{x}(M) = \text{Ker}(f'(x))$.
\end{proof}

\begin{Lem} \label{explicit0} 
	Let  $d\in\N$, let 
	$\emptyset\not=M \subset \R^{d}$ be a $C^{1}$-hypersurface, let $U \subset \R^{d}$ be open with $M\subset U$ and let $f\colon U \to \R$ be a $C^1$-function such that $M \subset f^{-1}(\{0\})$ and $ f'(x) \neq 0$ for all $x \in M$. Then 
	\[
	\nor\colon M \to \R^{d}, \, x \mapsto \frac{ f'(x)^{\top}}{\| f'(x)\|} 
	\]
	is a normal vector along $M$.
\end{Lem}
\begin{proof}
	Since $f$ is a $C^{1}$-function, $\nor$ is  continuous. Clearly, $\|\nor(x)\| = 1$ for all $x \in M$. Finally, by Lemma \ref{tangkern} we have $f'(x) v = 0$ for all $x\in M$ and all $v \in T_{x}(M)$, which implies $\langle \nor(x),v\rangle =0$  for all $x\in M$ and all $v \in T_{x}(M)$. Thus $\nor$ is a normal vector along $M$. 
\end{proof}

\begin{Lem} \label{nrep0}
	Let $d\in\N$ and let
 $\emptyset\neq M \subset \R^{d}$ 
	be a $C^{1}$-hypersurface.
	\begin{itemize}
		\item[(i)]  If $\nor\colon M \to \R^{d}$ is a normal vector along $M$ then, for every $x \in M$,
		\[
		T_{x}(M)^{\perp} = \text{span}(\{\nor(x)\}).
		\]
		\item[(ii)] 	If $M$ 
		is of 
		positive reach and $\eps\in (0,\reach(M))$ then,
		for every $x \in M$,
		\[
		T_{x}(M)^{\perp} =\{\lambda v \mid \lambda \geq 0, v \in \R^{d}, \norm{v} = \eps, \text{pr}_{M}(x+v) = x  \}.
		\]
	\end{itemize}
\end{Lem}

\begin{proof}
	Let $x \in M$.
	Since $M$ is a $C^{1}$-hypersurface, the tangent space $T_{x}(M)$ is known to be a $(d-1)$-dimensional vector space.
	Thus $T_{x}(M)^{\perp}$ is one-dimensional. Clearly, $\nor(x) \in T_{x}(M)^{\perp} \setminus \{0 \}$, which implies that $\text{span}(\{\nor(x)\})$ is a one-dimensional subspace of   $T_{x}(M)^{\perp}$. This finishes the proof of part (i) of the lemma.  Part (ii) of the lemma is a consequence of
	\cite[Theorem 4.8 (12)]{F59}.
\end{proof}

\begin{Lem} \label{fed0}
	Let $d\in\N$, let  
	$\emptyset\neq M \subset \R^{d}$ 
	 be a $C^{1}$-hypersurface of positive reach and let
	$x \in M$. If $u \in T_{x}(M)^{\perp}$ 
	and
	$\norm{u} < \reach(M)$ then  $x+u \in \text{Unp}(M)$ and $\text{pr}_{M}(x+u) = x$.
\end{Lem}

\begin{proof}
	Let 
	$u \in T_{x}(M)^{\perp}$ with $\norm{u} < \reach(M)$. 
	Clearly, the statement of the lemma holds for $u=0$. Assume $u\not=0$.
	Let $\delta\in(\|u\|,\reach(M))$.
	We have
	$x+u\in M^{\delta}$,
	which implies $x+u\in \text{Unp}(M)$.  
	Applying Lemma \ref{nrep0}(ii) with $\eps = \norm{u}$ we conclude that there exist $v \in \R^{d}$ and $\lambda \ge 0$ such that $\norm{v} = \norm{u}$, $\text{pr}_{M}(x+v) = x$ and $\lambda v = u$. We have $\lambda \norm{v} = \norm{u} = \norm{v}$, which implies $\lambda = 1$. Hence,  $u = v$ and $\text{pr}_{M}(x+u) = x$, which finishes the proof of the lemma. 
\end{proof}

\begin{Lem} \label{ort0} 
	Let $d\in\N$ and let 
	$\emptyset\not=M \subset \R^{d}$. Then for every $x \in \text{Unp}(M)$ we have $(x- \text{pr}_{M}(x)) \in T_{ \text{pr}_{M}(x)}(M)^\perp$.
\end{Lem}

\begin{proof}
	Let $x \in \text{Unp}(M)$ and let $v\in T_{ \text{pr}_{M}(x)}(M)$. Choose $\eps\in (0,\infty)$ and a $C^1$-mapping $\gamma\colon (-\eps,\eps) \to M$ such that $\gamma(0)= \pr_{M}(x)$ and $\gamma'(0) =v$. Assume that $\langle v, x-\pr_M(x)\rangle > 0$. Let $\delta\in (0,\langle v, x-\pr_M(x)\rangle)$. Since $\lim_{h\downarrow 0} \frac{\gamma(h) - \gamma(0)}{h} = v$ we conclude that there exists  $h_0\in (0,\eps)$ such that for all $h\in (0,h_0)$,
	\[
	\frac{\|\gamma(h)-\gamma(0)\|^2}{h}  <\delta \quad \text{and}\quad \frac{ \langle\gamma(h)-\gamma(0), x-\pr_M(x)\rangle}{h} > \langle v, x-\pr_M(x)\rangle -\delta/2.
	\]
	Thus, for all $h\in (0,h_0)$, 
	\begin{align*}
		\|x-\gamma(h)\|^2 & = \|x-\pr_M(x) -(\gamma(h)-\gamma(0)) \|^2\\
		& = \|x-\pr_M(x)\|^2 -2\langle\gamma(h)-\gamma(0),x-\pr_M(x)\rangle +  \|\gamma(h)-\gamma(0)\|^2\\
		& <  \|x-\pr_M(x)\|^2 -2h(\langle v, x-\pr_M(x)\rangle -\delta/2) + h\delta\\
		& = \|x-\pr_M(x)\|^2 -2h(\langle v, x-\pr_M(x)\rangle -\delta)
		\\&<  \|x-\pr_M(x)\|^2,
	\end{align*}
	which is a contradiction, since $\gamma(h)\in M$ for all $h\in (0,h_0)$. In a similar way one can show that $(x-\pr_M(x))^\top v < 0$ can not be true, which completes the proof of the lemma. 
\end{proof}

\begin{Lem} \label{unique} 
	Let  $d\in\N$ and let 
	$\emptyset\neq M \subset \R^{d}$ be a connected $C^{1}$-hypersurface. Then $M$ has either no or two normal vectors.
\end{Lem}
\begin{proof}
	Assume that  $\nor\colon M\to \R^d$ is a normal vector along $M$. Then $- \nor $ is also a normal vector along $M$.
	Assume that $\tilde \nor \colon M\to \R^d$ is a further normal vector along $M$. Since $\nor$ and $\tilde \nor$ are continuous we obtain by the Cauchy-Schwarz inequality that the mapping
	$S\colon M \to \R, x \mapsto \langle \nor(x), \tilde{\nor}(x)\rangle$ is continuous as well. 
	
	Let $x \in M$. By Lemma \ref{nrep0}(i) we have $\tilde{\nor}(x) \in \text{span}(\nor(x))$. Hence there exists $c \in \R$ such that $\nor(x) = c \,\tilde{\nor}(x)$. Since $\|\nor(x)\| = \|\tilde \nor(x)\|=1$ we obtain $S(x) = c \in \{1, -1 \}$. 
	
	Thus, $S(M) \subset \{1,-1\}$. Since $M$ is connected and $S$ is continuous, $S(M)$ is connected as well. It follows that either $S = 1$ or $S = -1$. This implies that either  $  \tilde{\nor}=\nor$ or $\tilde{\nor}=-\nor$, which finishes the proof of the lemma.
\end{proof}

\begin{Lem} \label{projdist} 
	Let  $d\in\N$, let 
	$k \in \N \cup \{\infty \}$, let $\emptyset\not=M \subset\R^{d}$ be a $C^{k}$-hypersurface of positive reach and let $\eps\in (0,\reach(M))$. Then the following statements hold true.
	\begin{itemize}
		\item[(i)] $\pr_M $ is a $C^{k-1}$-mapping on $M^\eps$.
			\item[(ii)] 	 $d(\cdot,M)$ is a $C^{k}$-function on $M^\eps\setminus M$.
		\item[(iii)] If $k\ge 2$ then for all $x\in M^\eps$ and all $v\in\R^d$,
		\[
		(\pr_{M}'(x))v\in T_{ \text{pr}_{M}(x)}(M).
		\]
		In particular,
		\[
		(x-\pr_M(x))^\top \pr'_M(x) = 0.
		\]
	\end{itemize} 
\end{Lem}

\begin{proof}
	See~\cite[Theorem 1.3]{DH94} for part (i) in the case $k=1$ and \cite[Theorem 4.1]{DH94} for part (i) in the case $k\ge 2$. See  \cite[Theorem 2]{Fo84} for part (ii) in the case $k=1$ and  \cite[Corollary 4.5]{DH94}  for part (ii) in the case $k\ge 2$.
	For the proof of part (iii), assume $k\ge 2$, let $x\in M^\eps$ and let $v\in \R^d$. 
	Assume, without loss of generality, that $v\neq 0$. 
	Since $M^\eps$ is open, there exists $\delta\in (0,\infty)$ such that $B_\delta(x)\subset M^\eps$. Let $r\in (0,\delta/\|v\|)$. Then $x+tv\in B_\delta(x)$ for every $t\in (-r,r)$ and therefore the function $\gamma\colon (-r,r)\to M^\eps$, $t\mapsto x+tv$ is well-defined. Clearly, $\gamma$ is a
	$C^1$-mapping
	with $\gamma(0)=x$ and $\gamma'(0)=v$. By part (i), $(\text{pr}_{M})_{|M^{\epsilon}}\colon M^\eps \to M$ is a 
	$C^1$-mapping
	as well. It follows that $(\text{pr}_{M})_{|M^{\epsilon}}\circ \gamma \colon (-r,r) \to M$ is a 
	$C^1$-mapping
	with 
	 $(\text{pr}_{M})_{|M^{\epsilon}}\circ \gamma(0) = \text{pr}_{M}(x)$ and $((\text{pr}_{M})_{|M^{\epsilon}}\circ \gamma)'(0) = (\text{pr}_{M})_{|M^{\epsilon}}'(\gamma(0)) \gamma'(0) = 	(\pr_{M}'(x))v$. Hence $(\pr_{M}'(x))v\in T_{ \text{pr}_{M}(x)}(M)$. 
	Finally, by Lemma~\ref{ort0} we obtain  $(x- \text{pr}_{M}(x)) \in T_{ \text{pr}_{M}(x)}(M)^\perp$, which completes the proof of the lemma.
\end{proof}

\begin{Lem} \label{sides0}
	Let  $d\in\N$, let  
	$\emptyset\neq M \subset \R^{d}$ be a $C^{1}$-hypersurface of positive reach, let $\nor\colon M \to \R^{d}$ be a normal vector along $M$.
Then for all $\eps \in (0,\reach(M))$, the sets 
	\[Q_{\varepsilon, +} = \{x +\lambda \nor(x) \mid x \in M, \, \lambda \in (0,\eps)\}
	\] 
	and
	\[
	Q_{\varepsilon, -} = \{x+\lambda \nor(x) \mid x \in M,\, \lambda \in (-\eps,0) \}
	\] 
	are open and disjoint and we have $M^{\eps}\setminus M = Q_{\varepsilon, +} \cup Q_{\varepsilon, -}$.
\end{Lem}

\begin{proof}
Let $\eps \in (0,\reach(M))$.
	Note that for all $x \in M$ and all $\lambda \in (-\eps,\eps)$ we have $d(x+\lambda \nor(x),M) \leq \norm{x-(x+\lambda \nor(x))} = | \lambda| < \eps$, which implies that the mapping
	\[
	F\colon  M \times (-\eps,\eps) \to M^{\eps}, \quad (x,\lambda) \mapsto x + \lambda \nor(x)
	\]
	is well-defined. 
	
	We  first show that $F$ is a homeomorphism.
	Let $x\in M^{\eps}$. By Lemma~ \ref{ort0} and Lemma~\ref{nrep0}(i) we have $x-\text{pr}_M(x) \in  T_{\text{pr}_M(x)}(M)^{\perp} = \text{span}(\{\nor(\text{pr}_M(x))\})$, and therefore $x-\text{pr}_M(x) = \langle \nor(\text{pr}_M(x)),x-\text{pr}_M(x) \rangle \nor(\text{pr}_M(x))$. As a consequence,
	\[
	| 	\langle \nor(\text{pr}_M(x)),x-\text{pr}_M(x) \rangle   | 
	= \|x-\text{pr}_M(x)\|  = d(x,M) < \eps.
	\]
	We conclude that $(\text{pr}_M(x), \langle \nor(\text{pr}_M(x)),x-\text{pr}_M(x) \rangle ) \in M\times (-\eps,\eps)$ and
	\begin{align*}
		F(\text{pr}_M(x), \langle\nor(\text{pr}_M(x)),x-\text{pr}_M(x)\rangle) & = \text{pr}_M(x) + \langle\nor(\text{pr}_M(x)),x-\text{pr}_M(x)\rangle \nor(\text{pr}_M(x)) \\ & = \text{pr}_M(x) + (x-\text{pr}_M(x)) = x,
	\end{align*}
	which shows that $F$ is surjective. 

Next, consider the mapping
	\[
	G\colon M^{\eps} \to M \times (-\eps,\eps), \quad x \mapsto (\text{pr}_M(x),\langle\nor(\text{pr}_M(x)),x-\text{pr}_M(x)\rangle).
	\]
	Employing Lemma \ref{nrep0}(i)  and Lemma~\ref{fed0} we obtain that for all $(x,\lambda)\in M\times ( -\eps, \eps)$ we have 
$\text{pr}_M(x+\lambda \nor(x))=x$
and hence
	\begin{align*}
		G(F(x,\lambda)) & = (\text{pr}_M(x+\lambda \nor(x)), \langle\nor(\text{pr}_M(x+\lambda \nor(x))),x+\lambda \nor(x) - \text{pr}_M(x+\lambda \nor(x))\rangle)\\ &  = (x,\langle\lambda \nor(x),\nor(x)\rangle) = (x,\lambda).
	\end{align*}
	Thus, $F$ is bijective and $G=F^{-1}$. Since $\text{pr}_M$ is continuous on $M^\eps$, see Lemma~\ref{projdist}(i), and since $\nor$ is continuous by definition we conclude that both $F$ and $F^{-1}$ are continuous as well. Thus,  $F$ is a homeomorphism.

	Clearly, $Q_{\varepsilon, +}=  F(M\times (0,\varepsilon))$ and $Q_{\varepsilon, -} =  F(M\times (-\varepsilon,0) )$. Since $F$ is injective we conclude that $Q_{\varepsilon, +}\cap Q_{\varepsilon, -}=\emptyset$
and
	\begin{align*}
		M^\varepsilon\setminus M & = F(M\times (-\varepsilon,\varepsilon))\setminus F(M\times \{0\}) = F((M\times (-\varepsilon,\varepsilon))\setminus (M\times \{0\})) \\
		& = F((M\times (-\varepsilon,0) )\cup (M\times (0,\varepsilon) )) = F(M\times (-\varepsilon,0) )\cup F(M\times (0,\varepsilon) ) =Q_{\varepsilon, +}\cup Q_{\varepsilon, -}.
	\end{align*}
	Finally, observe that $M \times (-\varepsilon,0)$
and $M \times (0,\varepsilon)$ are open sets in $M\times (-\varepsilon,\varepsilon)$. Since $G=F^{-1}$ is continuous we thus obtain that $Q_{\varepsilon, +}=  G^{-1}(M\times (0,\varepsilon))$ and $Q_{\varepsilon, -} = G^{-1}(M\times (-\varepsilon,0) )$ are open sets in $M^\varepsilon$. Since $M^\varepsilon$ is open in $\R^d$, we conclude that both $Q_{\varepsilon, +}$ and $Q_{\varepsilon, -}$ are open in $\R^d$ as well. This completes the proof of the lemma.
\end{proof}

\begin{Lem} \label{normalreg0}
	Let $d\in\N$, let 
$k \in \N \cup \{\infty \}$, let $\emptyset\not=M \subset\R^{d}$ be a $C^{k}$-hypersurface and let $\nor\colon M\to \R^d$ be a normal vector along $M$. Then $\nor$ is a $C^{k-1}$-function. Moreover, if $k\ge 2$ then for all $x \in M$ and all $v \in T_{x}(M)$ we have $\nor'(x)v \in T_{x}(M)$.
\end{Lem}

\begin{proof}
	Let $x \in M$ and choose a chart $(\phi, U)$ for $M$ at $x$. By Lemma \ref{connect0} we may assume that $M \cap U$ is connected. Clearly, 
	$M\cap U$ is a $C^k$-hypersurface. Since $\phi$ is a $C^{1}$-diffeomorphism we have $ \phi_{d}'(x) \neq 0$ for all $x \in U$. Moreover, $M \cap U \subset \phi_{d}^{-1}(\{0 \})$. By Lemma \ref{explicit0}
we may therefore conclude that the mapping 
$\nu\colon  M\cap U \to \R^{d}, x\mapsto  \frac{ \phi_{d}'^{\top}(x)}{\| \phi_{d}'(x)\|}$ 
is a normal vector along $M \cap U$. Clearly,
	 $\nor_{|M\cap U}$ is a normal vector along $M\cap U$ as well. By Lemma \ref{unique}
	we thus have 
$\nor_{|M\cap U} = \nu$ or $\nor_{|M\cap U} = -\nu$. Note that $\nu$ and $-\nu$ are $C^{k-1}$-functions since $\phi_d$ is a $C^k$-function with $ \phi_{d}'(x) \neq 0$ for all $x \in U$. This completes the proof of the first statement of the lemma. 
	
	Next, let $k\ge 2$, let $x \in M$ and let $v \in T_{x}(M)$. Then $\nor$ is a $C^1$-function and there exist $\varepsilon > 0$ and a $C^1$-mapping $\gamma\colon (-\varepsilon,\varepsilon) \to M$ such that $\gamma(0)=x$ and $\gamma'(0)=v$. Using  the fact that $\|\nor\circ \gamma\| = 1$ we obtain that for all $t\in  (-\varepsilon,\varepsilon)$,
	\[
	0 = (\|\nor \circ\gamma\|^{2})'(t) = 2 \nor(\gamma(t))^{\top}\nor'(\gamma(t))\gamma'(t).
	\]
	For $t = 0$ we get $0 = 2\nor(x)^{\top}(\nor'(x)v)$. Hence, $\nor'(x)v\in \text{span}\{\nor(x) \}^\perp$. By Lemma \ref{nrep0}(i) we have $\text{span}\{\nor(x) \}=T_{x}(M)^{\perp}$, which finishes the proof of the lemma.
\end{proof}

\begin{Lem} \label{pdiffbd0}	Let $d\in\N$, let $k \in \N$ with $ k\ge 2$, let $\emptyset\not=M \subset \R^{d}$ be a $C^{k}$-hypersurface of positive reach and let $\nor\colon M \to \R^{d}$ be a normal vector along $M$ such that for all $\ell\in\{1, \ldots, k-1\}$,
\[
	\sup_{x \in M}  \norm{\nor^{(\ell)}(x)}_\ell < \infty.
	\]
	Then, for all $\varepsilon \in (0,\reach(M))$ and all $\ell\in\{1, \ldots, k-1\}$,
	\begin{equation}\label{Lproj}
	\sup_{x \in M^{\varepsilon}}\norm{\pr_{M}^{(\ell)}(x)}_\ell < \infty.
	\end{equation}
\end{Lem}
\begin{proof}
See~\cite[Corollary 3]{LS2021}. We add that there is a typo in the formulation of Corollary 3 in~\cite{LS2021}. 
The bound \eqref{Lproj} is proven for $\ell=1$ as well,
see (3) in the proof of the latter result.
\end{proof}

\begin{Lem} \label{pdiff01} 
	Let  $d\in\N$, let 
$\emptyset\neq M \subset \R^{d}$ be a $C^{2}$-hypersurface of positive reach, let $\eps \in (0,\reach(M))$ and let $\nor\colon M\to \R^{d}$ be a normal vector along $M$. Then for all $x \in M^{\eps}$ we have
	\[
	I_{d} - \nor(\pr_{M}(x))\nor(\pr_{M}(x))^{\top} = \bigl(I_{d} + \langle x-\pr_{M}(x),\nor(\pr_{M}(x))\rangle \nor'(\pr_{M}(x)) \bigr)\pr_{M}'(x).
	\]
	In particular, for all $x \in M$ we have
	$\pr_{M}'(x) = I_{d} - \nor(x)\nor(x)^{\top}$.
\end{Lem}

\begin{proof}
	See~\cite[Theorem C]{LS2021}.
	\end{proof}

\begin{Lem} \label{Lipconnec0}
	Let  $d\in\N$, let
$\emptyset\neq A \subset \R^{d}$ and let $\rho_A$ be the intrinsic metric for $A$. Then   $\|x-y\| \leq \rho_A(x,y)$ for all $x,y \in A$ and $\|x-y\| = \rho_A(x,y)$ for all $x,y \in A$ with $\overline{x,y}\subset A$. In particular, if $A$ is convex then $\rho_A$ coincides with the Euclidean distance.
\end{Lem}

\begin{proof}
	Let $x,y \in A$ and let $\gamma\colon [0,1] \to A$ be continuous  with $\gamma(0)= x$ and $\gamma(1) = y$. Then
	\[
	l(\gamma)  \geq \|\gamma(0) - \gamma(1)\| = \|x-y\|. 
	\] 
	Hence $\|x-y\| \leq \rho_A(x,y)$. Next, assume that $\overline{x,y}\subset A$ and consider the function $\gamma\colon [0,1] \to A, \lambda \mapsto (1-\lambda)x + \lambda y$. Clearly, $\gamma$ is continuous with $\gamma(0)=x$ and $\gamma(1)=y$. Hence $\rho_A(x,y) \le l(\gamma) = \|x-y\|$. 
	Thus $\rho_A(x,y) = \|x-y\|$ in this case.
\end{proof}

\begin{Lem}\label{lipimpl0} 
	Let $d, m \in \N$, 
	let $\emptyset \neq A \subset \R^{d}$ and let $f\colon A\to \R^{m}$ be a function.
	\begin{itemize}
		\item [(i)] If $f$ is Lipschitz continuous then $f$ is intrinsic Lipschitz continuous.
		\item[(ii)] If $A$ is convex then $f$ is intrinsic Lipschitz continuous with intrinsic Lipschitz constant $L$ if and only if $f$ is Lipschitz continuous with Lipschitz constant $L$.
		\item[(iii)] If $A$ is open and $f$ is  intrinsic Lipschitz continuous then $f$ is locally Lipschitz continuous and, in particular, $f$ is continuous.
	\end{itemize}  
\end{Lem}

\begin{proof}
	The lemma is an immediate consequence of Lemma \ref{Lipconnec0}.
\end{proof}

\begin{Lem} \label{condfullf0} 
	Let $d \in \N$ and let
	 $\emptyset \not=M \subset \R^{d}$ be 
	 closed and a $C^{1}$-hypersurface.
	 Then for all $x,y \in \R^{d}$ and all $\varepsilon > 0$ there exists a continuous function $\gamma\colon [0,1]\to \R^{d}$ such that $\gamma(0) = x$, $\gamma(1) = y$, $l(\gamma) < \|x-y\| + \varepsilon$ and $|\gamma([0,1]) \cap M|< \infty$.
\end{Lem}
\begin{proof}
	See~\cite[Lemma 31]{LeSt2022}.
\end{proof}

\begin{Lem}\label{top1}
	Let  $d\in \N$, let 
	$\emptyset \not=A\subset\R^d$
	be open and let $K\subset A$ be compact. Then there exists $\eps\in (0,\infty)$ such that $K^\eps \subset A$.
\end{Lem}

\begin{proof}
	Assume, in contrary, that, for every $n\in\N$, there exists $x_n\in K^{1/n} \setminus A$. Since $K$ is bounded, the sequence $(x_n)_{n\in \N}$ is bounded. Hence, there exists $x_0\in\R^d$ and a subsequence $(x_{n_k})_{k\in\N}$ such that $\lim_{k\to \infty} x_{n_k} = x_0$. Since $x_{n_k}\in \R^d\setminus A$ for every $k\in\N$ and $\R^d\setminus A$ is closed, we conclude that $x_0\in \R^d\setminus A$. On the other hand, we have $d(x_0,K) \le \|x_0 - x_{n_k}\| + d(x_{n_k} ,K) \le  \|x_0 - x_{n_k}\| + 1/n_k$ for every $k\in\N$, which implies $d(x_0,K) =0$. Since $K$ is closed, we conclude $x_0\in K$, which contradicts  $x_0\in \R^d\setminus A$.
\end{proof}

\begin{Lem} \label{parting0} 
	Let $d\in \N$, 
	let
	$\emptyset\neq A \subset \R^{d}$ be open and let $\gamma\colon [0,1] \to A$ be continuous. Then there exists
	$n_0 \in \N$ such that for every $n\geq n_0$,
	\[
	\bigcup_{i = 1}^{n} \overline{\gamma((i-1)/n),\gamma(i/n)} \subset A. 
	\]
\end{Lem}

\begin{proof}
	Since $\gamma$ is continuous, the set $\gamma([0,1])$ is a compact subset of $A$. Hence we obtain by Lemma~\ref{top1} the existence of $\eps\in (0,\infty)$ such that $(\gamma([0,1]))^\eps \subset A$. Moreover, $\gamma$ is uniformly continuous on $[0,1]$, and therefore there exists 
	$n_0 \in \N$ such that for every $n\geq n_0$,
	every $i\in \{1,\dots,n\}$ and every $t\in [(i-1)/n,i/n]$, 
	\begin{align*}
		&	\|\gamma(t) - (nt-(i-1)) \gamma(i/n) - (i- nt)\gamma((i-1)/n)\| \\
		& \qquad\qquad \le	\|\gamma(t) - \gamma(i/n)\| + \| \gamma(t)- \gamma((i-1)/n)\| < \eps.
	\end{align*}
	Thus, $\overline{\gamma((i-1)/n),\gamma(i/n)} \subset  (\gamma([0,1]))^\eps$, which completes the proof of the lemma.
\end{proof}

\begin{Lem} \label{piecewtolip0}  
	Let  $d, m,k \in \N$, 
	let $\emptyset \not=M \subset \R^{d}$,
	let $f\colon  \R^{d} \to \R^{k\times m}$ be continuous
	on $\R^d$
 as well as 
	intrinsic Lipschitz continuous on $\R^d\setminus M$
	and assume that  for all $x,y \in \R^{d}$ and all $\varepsilon > 0$ there exists a continuous function $\gamma\colon [0,1]\to \R^{d}$ such that $\gamma(0) = x$, $\gamma(1) = y$, $l(\gamma) < \|x-y\| + \varepsilon$ and $|\gamma([0,1]) \cap M| < \infty$.
	Then $f$ is Lipschitz continuous
	on $\R^d$.
\end{Lem}

Lemma~\ref{piecewtolip0} is proven in~\cite[Lemma 3.6]{LS17}, see, however
Remark~\ref{SL1}(i).
For the convenience of the reader we  present a proof of Lemma~\ref{piecewtolip0}.
\begin{proof}
	Let $L\in (0,\infty)$ be an intrinsic Lipschitz constant for $f$ on $\R^d\setminus M$, let $x,y\in \R^d$ and let $\eps>0$. By assumption, there exists a continuous function $\gamma\colon [0,1]\to \R^{d}$ such that $\gamma(0) = x$, $\gamma(1) = y$, $l(\gamma) < \|x-y\| + \varepsilon$ and $|\gamma([0,1]) \cap M| < \infty$.
	
	Let $K\in\N$ and let $0=t_0 < \dots < t_K = 1$ such that $\gamma([0,1]) \cap M\subset \{t_0,\dots, t_K\}$. Since $f\circ \gamma$ is continuous we obtain
	\begin{equation}\label{f1}
	\|f(x) - f(y)\| \le \sum_{k=1}^K \| f(\gamma(t_k)) -f(\gamma(t_{k-1}))\| = \lim_{h\downarrow 0}\sum_{k=1}^K \| f(\gamma(t_k-h)) -f(\gamma(t_{k-1}+h))\|.
 	\end{equation}
 Let $k\in \{1,\dots,K\}$ and $h>0$ such that $t_{k-1} + h < t_k -h $. Then $\gamma([t_{k-1}+h,t_k-h]) \subset \R^d\setminus M$ and we obtain
 \begin{equation}\label{f2}
 	\begin{aligned}
 \| f(\gamma(t_k-h)) -f(\gamma(t_{k-1}+h))\| &  \le L\rho_{\R^d\setminus M} (\gamma(t_{k-1}+h), \gamma(t_k-h))\\
  & \le L l(\gamma_{| [t_{k-1}+h, t_k-h]} ) \le L l(\gamma_{| [t_{k-1}, t_k]} ).
 \end{aligned}
\end{equation}
 Combining \eqref{f1} and \eqref{f2} we obtain
\[
	\|f(x) - f(y)\| \le L \sum_{k=1}^K l(\gamma_{| [t_{k-1}, t_k]} ) = L l(\gamma) <  L (\|x-y\| + \eps).
\]
Letting $\eps $ tend to zero completes the proof of the lemma.
\end{proof}

\begin{Lem} \label{diffintr0} 
	Let $d,m\in \N$, let
	$\emptyset\neq A \subset \R^{d}$ be open and let $f\colon  A \to \R^{m}$ be differentiable with $\|f'\|_{\infty} < \infty$. Then $f$ is intrinsic Lipschitz continuous with intrinsic Lipschitz constant $\|f'\|_{\infty}$.
\end{Lem}

Lemma~\ref{diffintr0} is  proven in~\cite[Lemma 3.8]{LS17}, see, however, Remark~\ref{SL1}(ii).
For the convenience of the reader we present a proof of Lemma~\ref{diffintr0}.

\begin{proof}
Let $x,y \in A$. Clearly, we may assume $\rho_A(x,y) <\infty$. Then there exists a continuous function $\gamma\colon[0,1] \to A$ with $\gamma(0) = x$ and $\gamma(1) = y$. By Lemma \ref{parting0} there exist $n\in\N$ and $0 = t_{0} < t_{1} <  \dots < t_{n} = 1$ such that  $\overline{\gamma(t_{i-1}),\gamma(t_{i})} \subset A$ for all $i \in \{ 1, \dots, n \}$. Hence, by the mean value theorem, 
\begin{align*} 
\|f(y)-f(x)\| & \leq \sum_{i = 1}^{n} \|f(\gamma(t_{i}))- f(\gamma(t_{i-1}))\| \\
& \leq 			\sum_{i = 1}^{n}\sup_{x\in \overline{\gamma(t_{i-1}),\gamma(t_{i})}}\|f'(x)\| \|\gamma(t_{i}) - \gamma(t_{i-1}) \| \\
&   \leq \|f'\|_{\infty} \sum_{i = 1}^{n} \|\gamma(t_{i}) - \gamma(t_{i-1}) \| \\
& \leq  \norm{f'}_{\infty} \sum_{i = 1}^{n} l(\gamma_{|[t_{i-1}, t_{i}]}) = 
\|f'\|_{\infty}\, l(\gamma).\qedhere
\end{align*}
\end{proof}

\begin{Lem} \label{comp0}
	Let $d,m \in \N$, let $\emptyset\neq B \subset \R^{d}$ and $\emptyset\neq A \subset \R^{m}$ be open, let $g\colon  B \to A$
	be intrinsic Lipschitz continuous with intrinsic Lipschitz constant $L_{g}$, let $f \colon  A \to \R^{m}$ be intrinsic Lipschitz continuous with intrinsic Lipschitz constant $L_{f}$.
	Then $f \circ g\colon  B \to \R^{m}$ is intrinsic Lipschitz continuous with intrinsic Lipschitz constant $L_{f} L_{g}$.
\end{Lem}

Lemma~\ref{comp0} is  proven in~\cite[Lemma 3.9]{LS17}, see, however, Remark~\ref{SL1}(iii). 
For the convenience of the reader we  present a proof of Lemma~\ref{comp0}.

\begin{proof}
	Let $\rho_{B}$ be the intrinsic metric for $B$ and let $\rho_{A}$ be the intrinsic metric for $A$. Let $x,y \in B$. Clearly, we may assume $\rho_B(x,y) <\infty$. Then there exists a continuous function  $\gamma\colon [0,1] \to B$  with $\gamma(0) = x$ and $\gamma(1) = y$. Lemma \ref{lipimpl0}(iii) implies that $g \circ \gamma\colon  [0,1] \to A$ is  continuous.
	 We therefore obtain by Lemma \ref{parting0}  that there exist 
		 $n\in\N$ and
	 $0 = t_{0} < t_{1} < \dots < t_{n} = 1$ such that for all $i \in \{ 1,\dots, n\}$ we have $\overline{g(\gamma(t_{i-1})),g(\gamma(t_{i}))} \subset A$. Employing  Lemma \ref{Lipconnec0} we conclude that
	\begin{align*} 
		\|(f \circ g)(x) - (f \circ g)(y)\| &  \leq L_{f} \rho_{A}(g(x), g(y) )\leq L_{f} \sum_{i = 1}^{n} \rho_{A}(g(\gamma(t_{i-1})),g(\gamma(t_{i})))  \\
		& = L_{f} \sum_{i = 1}^{n} \|g(\gamma(t_{i})) -g(\gamma(t_{i-1}))\|
		\leq L_{f} L_{g} \sum_{i = 1}^{n} \rho_{B}(\gamma(t_{i-1}),\gamma(t_{i})) \\&\leq L_{f} L_{g} \sum_{i= 1}^{n} l(\gamma_{|[t_{i-1},t_{i}]}) = L_{f}L_{g} l(\gamma),
	\end{align*}
which completes the proof of the lemma.
\end{proof}

\begin{Rem}\label{SL1} We comment on the proofs of Lemma 3.6, Lemma 3.8 and Lemma 3.9 in~\cite{LS17} corresponding to Lemma~\ref{piecewtolip0}, Lemma~\ref{diffintr0} and Lemma~\ref{comp0}, respectively. We use the notation from~\cite{LS17}. 
	\begin{itemize}
		\item[(i)] In the proof of Lemma 3.6 in~\cite{LS17}, the case distinction is not complete: since $f$ is assumed to be intrinsic Lipschitz continuous on $\R^d\setminus \Theta$, the inequality $\|f(x)-f(y)\| \le L\rho(x,y) $ holds only for $x,y\in \R^d\setminus \Theta$ but not for all $x,y\in\R^d$ as stated. Furthermore, $\rho(x,y)$ is only defined for  $x,y\in \R^d\setminus \Theta$ but not for all $x,y\in\R^d$.
		\item[(ii)]  In the proof of Lemma 3.8 in~\cite{LS17}, for $x,y\in A$, a continuous curve $\gamma\colon[0,1]\to A$ is considered, which connects $x$ and $y$. It is stated that, {\it without loss of generality}, there exist $n\in\N$ and points $0  =t_0 < t_1 < \dots < t_n =1$ such that every line segment $s(\gamma(t_{k-1}),\gamma(t_k))$ is in $A$. 
		It seems to us that the latter fact is not straightforward but needs an argument like Lemma~\ref{parting0}, which is applicable because  $A$ is open. 
		\item[(iii)] In the proof of Lemma 3.9 in~\cite{LS17}, the inequality (correcting for obvious typos)
		\[
		\sum_{k=1}^n \|f\circ g(\gamma(t_k))-f\circ g(\gamma(t_{k-1}))\| \le L_f \sum_{k=1}^n \|g(\gamma(t_k))-g(\gamma(t_{k-1}))\| 
		\]
		is wrong, because $f$ is only assumed to be intrinsic Lipschitz continuous on $A$. One can only state that 
		\[
		\|f\circ g(\gamma(t_k))-f\circ g(\gamma(t_{k-1}))\| \le L_f \rho(g(\gamma(t_k)),g(\gamma(t_{k-1})))
		\] 
		for $k=1,\dots,n$. 
 The subsequent inequality
		\[
	 L_f \sum_{k=1}^n \|g(\gamma(t_k))-g(\gamma(t_{k-1}))\| \le L_f \rho (g(x),g(y))	
		\]
		is wrong as well. 
		A simple counterexample can be constructed by taking $n=2$ and $g$ such that $g(x)=g(y)\not=g(\gamma(t_1))$. 
	\end{itemize}
\end{Rem}

\begin{Lem} \label{productnew0}
	Let $d,m,k,\ell\in \N$, let $\emptyset\neq A,B,
	C
	 \subset \R^{d}$ with $A,B\subset C$, let $f\colon C \to \R^{m \times k}$ and $g\colon C \to \R^{k\times\ell}$ be intrinsic Lipschitz continuous on $A$ and bounded on $B$, and let $f$ be constant on $C\setminus B$. Then $f  g\colon C\to\R^{m\times \ell}$ is intrinsic Lipschitz continuous on $A$.
\end{Lem}

\begin{proof}
	Note that $\|f\|_\infty <\infty$ since $f$ is bounded on $B$ and constant on $C\setminus B$. 
	Let $\rho_A$  denote the intrinsic metric for $A$ and let $K\in (0,\infty)$ be an intrinsic Lipschitz constant for  $f$ and for $g$ on $A$. Let $x,y \in A$. First, assume that 
	$x \in B$ or $y \in B$. Without loss of generality we assume that  $x \in B$. We then have
	\begin{align*}
		\|(fg)(x) - (fg)(y)\| &  \leq \|f(x) - f(y)\| \|g(x)\| + \|f(y)\| \|g(x)-g(y)\| \\
		&  \leq K (\|g\|_{\infty,B}  + \|f\|_\infty)\rho_A(x,y).
	\end{align*}
	Next, assume that $x,y \in A\setminus B$. In this case we have $f(x)=f(y)$, and therefore 
	\[
	\|(fg)(x) - (fg)(y)\| \leq \|f(x)\| \|g(x)-g(y)\| \leq K  \|f\|_\infty \rho_A(x,y),
	\]
	which finishes the proof of the lemma.
\end{proof}

\begin{Lem} \label{a0} 
	Let $d\in\N$, let $\emptyset\neq M \subset \R^{d}$ be a $C^{1}$-hypersurface 
	of positive reach
	and let $\nor\colon M \to \R^{d}$ be a normal vector along $M$. Let $f\colon \R^{d} \to \R^{d}$ be piecewise Lipschitz continuous with exceptional set $M$ and
	assume that for all $x \in M$, the limit $\lim_{h \to 0} f(x+h\nor(x))$  exists and coincides with $f(x)$. Then $f$ is continuous.
\end{Lem}

\begin{proof} 
	Note that $M$ is closed because $\reach(M) >0$. 
	Thus, by Lemma \ref{lipimpl0}(iii),    $f$ is continuous on the open set $\R^{d} \setminus M$.
	It remains to show that $f$ is continuous at every $x\in M$.
	
	Let $\eps\in (0,\reach(M))$, let $x\in M$ and let $(x_k)_{k\in\N}$ be a sequence in $\R^d$ with $\lim_{k\to\infty} x_k = x$. We show that $\lim_{k\to\infty} f(x_k) = f(x)$. 
	Without loss of generality
	we may assume that  $(x_k)_{k\in\N}$ is either a sequence in $M\setminus\{x\}$ or a sequence in  $M^\eps\setminus M$.  
	
	Assume first that
	$(x_k)_{k\in\N}\subset M\setminus \{x\}$.
	For  $k\in \N$ put
	\[
	h_k = 2(\|x-x_k\| + \eps\|\nor(x)-\nor(x_k)\| )
	 \in (0,\infty).
	\]
Since $\nor$ is continuous we have $\lim_{k\to\infty} h_k = 0$. Therefore,
	\begin{equation}\label{eqq1}
		\lim_{k\to\infty} f(x+ h_k\nor(x)) = f(x).
	\end{equation}	
        Below we prove that 
    \begin{equation}\label{eqq2}
    	\lim_{k\to\infty} 	(f(x+ h_k\nor(x) ) - f(x_k+ h_k\nor(x_k))) = 0
     \end{equation}
    as well as
    \begin{equation}\label{eqq3}
    	\lim_{k\to\infty} 	(f(x_k+ h_k\nor(x_k) ) - f(x_k)) = 0.
     \end{equation}
    Combining~\eqref{eqq1} to \eqref{eqq3} yields $\lim_{k\to\infty} f(x_k) = f(x)$.
    
    We next show \eqref{eqq2} and \ref{eqq3}.
    Let $k\in\N$.
	Without loss of generality
	we may assume that $h_k <\eps$.
	We have
	\[
	\|x+ h_k\nor(x) - x-\eps \nor (x)\| = |h_k-\eps| = \eps -h_k < \eps
	\]
	as well as 
	\[
	\|x_k+ h_k\nor(x_k) - x-\eps \nor (x)\| \le \|x-x_k\| + \eps\|\nor(x)-\nor(x_k)\| +  |h_k-\eps| = \eps -h_k/2  < \eps.
	\]
	Hence, $x+ h_k\nor(x), x_k+ h_k\nor(x_k)\in B_\eps(x+\epsilon\nor(x))$. Using Lemma~\ref{fed0} we obtain that
	\[
	\overline{x+ h_k\nor(x), x_k+ h_k\nor(x_k)}\subset B_\eps(x+\eps\nor(x))\subset \R^d\setminus M.
	\]
	Applying Lemma~\ref{Lipconnec0} we thus conclude that 
	\begin{align*}
		\|f(x+ h_k\nor(x) ) - f(x_k+ h_k\nor(x_k))\| & 
		\le L\rho_{\R^{d} \setminus M} (x+ h_k\nor(x),x_k+ h_k\nor(x_k))\\
			& = L\|x+ h_k\nor(x)-x_k- h_k\nor(x_k)\| \\
		& \le L(\|x-x_k\| + h_k\|\nor(x)-\nor(x_k)\|),
	\end{align*}
	where $L$ is an intrinsic Lipschitz constant for 
	$f$ on $\R^d\setminus M$.
	This yields \eqref{eqq2}.
	For the proof of \eqref{eqq3}, observe that 
	there exists $m_0\in\N$ such that $h_m < h_k$ for all $m\ge m_0$. 
	Hence, for all $m\ge m_0$ we have
	\[
	\|x_k+ h_m\nor(x_k) - x_k-h_k\nor (x_k)\| = h_k - h_m  < h_k.
	\]
	Using  Lemma~\ref{fed0} we obtain that for all $m\ge m_0$,
	\[
	\overline{x_k + h_m\nor(x_k), x_k + h_k\nor(x_k)}\subset B_{h_k}(x_k+h_k\nor(x_k)) \subset \R^d\setminus M.
	\]
	Applying Lemma~\ref{Lipconnec0} we thus conclude that for all $m\ge m_0$,
	\begin{align*}
		\|f(x_k+ h_m\nor(x_k) ) - f(x_k+ h_k\nor(x_k))\| & \le L\rho_{\R^{d} \setminus M} (x_k+ h_m\nor(x_k),x_k+ h_k\nor(x_k))\\
			& = L\|x_k+ h_m\nor(x_k)-x_k- h_k\nor(x_k)\|\\
		&< h_k  . 
	\end{align*}
	Since $f(x_k) = \lim_{m\to\infty} f(x_k + h_m\nor(x_k))$ we obtain
	\begin{equation*}
		\|f(x_k+ h_k\nor(x_k) ) - f(x_k)\| = \lim_{m\to\infty}\| f(x_k+ h_k\nor(x_k) ) -  f(x_k + h_m\nor(x_k))\| \le h_k,
	\end{equation*}
	which implies
	\eqref{eqq3}.
	
	Next, assume that $(x_k)_{k\in\N}\subset M^\eps \setminus M$.
		 Since $\text{pr}_M$ is continuous on $M^\eps$, see Lemma~\ref{projdist}(i), we obtain that  $\lim_{k\to\infty} \|\text{pr}_M(x_k) - x\|  = \lim_{k\to\infty} \|\text{pr}_M(x_k) - \text{pr}_M(x)\| = 0$. It 
		 follows
		 from the first case that
	\begin{equation}\label{eqq4}
		\lim_{k\to\infty} 	\|f(\text{pr}_M(x_k) ) - f(x)\| = 0.
	\end{equation}
	We next show that
	\begin{equation}\label{eqq5}
		\lim_{k\to\infty} 	\|f(\text{pr}_M(x_k) ) - f(x_k)\| = 0,
	\end{equation}
    which jointly with ~\eqref{eqq4}  yields $\lim_{k\to\infty} f(x_k) = f(x)$.
    
    For the proof of \eqref{eqq5}, observe first that
	  Lemma~\ref{nrep0}(i) and Lemma~\ref{ort0} imply that for every $k\in \N$ we have
	\[
	x_k =  \text{pr}_M(x_k) + \lambda_k\nor(\text{pr}_M(x_k)),
	\]
	where $\lambda_k\in\R$ satisfies $0< |\lambda_k| = \|x_k - \text{pr}_M(x_k) \| \le \|x_k-x\|$. 	As a consequence, $\lim_{k\to\infty} \lambda_k = 0$. 
	Let $k\in\N$. For $m\in\N$  put 
	\[
	\widetilde \lambda_m = \begin{cases} |\lambda_m|,& \text{ if }\lambda_k >0, \\ 
		-|\lambda_m|,& \text{ if }\lambda_k <0, \end{cases}
	\]
	and choose $m_0\in\N$ such that $|\widetilde \lambda_m| < |\lambda_k|$ for all $m\ge m_0$.  Then for all $m\ge m_0$ we have
	\[
	\| \text{pr}_M(x_k) + \widetilde \lambda_m \nor( \text{pr}_M(x_k)) -  \text{pr}_M(x_k)- \lambda_k\nor( \text{pr}_M(x_k))\| = |\widetilde \lambda_m - \lambda_k| = |\lambda_k| - |\lambda_m| < |\lambda_k|.
	\]
	Using Lemma~\ref{fed0} we conclude that for all $m\ge m_0$,
	\[ \overline{ \text{pr}_M(x_k) + \widetilde \lambda_m \nor( \text{pr}_M(x_k)),  \text{pr}_M(x_k)+ \lambda_k\nor( \text{pr}_M(x_k)) }\subset B_{|\lambda_k|}(\text{pr}_M(x_k)+ \lambda_k\nor( \text{pr}_M(x_k)))\subset  \R^d\setminus M.
	\]
	Employing Lemma~\ref{Lipconnec0} we obtain that for all $m\ge m_0$,  
	\begin{align*}
		& \|f(  \text{pr}_M(x_k) + \widetilde \lambda_m \nor( \text{pr}_M(x_k)) ) - f(  \text{pr}_M(x_k) +  \lambda_k \nor( \text{pr}_M(x_k)) )\|\\
		& \qquad \qquad \le  L\rho_{\R^{d} \setminus M} (\text{pr}_M(x_k) + \widetilde \lambda_m \nor( \text{pr}_M(x_k)) ,  \text{pr}_M(x_k) +  \lambda_k \nor( \text{pr}_M(x_k)) )\\
		& \qquad\qquad = L \|\text{pr}_M(x_k) + \widetilde \lambda_m \nor( \text{pr}_M(x_k))  -  \text{pr}_M(x_k) -  \lambda_k \nor( \text{pr}_M(x_k)) \| \\
		 & \qquad\qquad <L|\lambda_k|.
	\end{align*}
	Since $f(\text{pr}_M(x_k)) = \lim_{m\to\infty} f(  \text{pr}_M(x_k) + \widetilde \lambda_m \nor( \text{pr}_M(x_k)) ) $ we conclude that
	\begin{align*}
		\|f(\text{pr}_M(x_k) ) - f(x_k)\| & = 	\|f(\text{pr}_M(x_k) ) - f(\text{pr}_M(x_k) + \lambda_k\nor(\text{pr}_M(x_k))\| \\
		& =	    	\lim_{m\to\infty}\| f(  \text{pr}_M(x_k) + \widetilde \lambda_m \nor( \text{pr}_M(x_k)))  - f(  \text{pr}_M(x_k) +  \lambda_k \nor( \text{pr}_M(x_k)) )   \| \\	
& \le L|\lambda_k|,
	\end{align*}  
	which implies \eqref{eqq5} and completes the proof of the lemma.
\end{proof}

\bibliographystyle{acm}
\bibliography{bibfile}

\end{document}